\documentclass{amsart}
\usepackage{amsmath}
\usepackage{amssymb}
\usepackage{verbatim}
\usepackage{dsfont}
\usepackage{enumerate}
\usepackage[all]{xy}
\usepackage[mathscr]{eucal}

\usepackage{amsthm}
\usepackage{stmaryrd}
\usepackage{amscd}
\usepackage{amsfonts}
\usepackage{latexsym}

\usepackage{amscd}

\usepackage{graphicx}

\usepackage{amsmath}
\usepackage{mathrsfs,amssymb, amscd,amsmath,amsthm}
\usepackage[enableskew,vcentermath]{youngtab}
\usepackage{multicol}\multicolsep=0pt
\usepackage{tikz}

\newcommand{\clr}{rgb:black,1;blue,4;red,1}

\newcommand{\ob}[1]{\mathsf{#1}}

\newcommand{\Uqpn}{\mathrm{U}_q(\mathfrak{p}_n)}
\newcommand{\Uq}{\mathrm{U}_q}

\newcommand{\lcap}{
\begin{tikzpicture}[baseline = 3pt, scale=0.5, color=\clr]
        \draw[-,thick] (1,0) to[out=up, in=right] (0.53,0.5) to[out=left, in=right] (0.47,0.5);
        \draw[-,thick] (0.49,0.5) to[out=left,in=up] (0,0);
\end{tikzpicture}
}
\newcommand{\lcup}{
\begin{tikzpicture}[baseline = 6pt, scale=0.5, color=\clr]
        \draw[-,thick] (1,1) to[out=down, in=right] (0.53,0.5) to[out=left, in=right] (0.47,0.5);
        \draw[-,thick] (0.49,0.5) to[out=left,in=down] (0,1);
\end{tikzpicture}
}

 \providecommand{\og}{``}
\providecommand{\fg}{''} \providecommand{\smfandname}{and}

\usepackage{amssymb}
\usepackage{mathrsfs,amssymb}
\usepackage{multicol}\multicolsep=0pt
\usepackage{pstricks,pst-node}

\usepackage[enableskew,vcentermath]{youngtab}
\usepackage[sort]{cite}
\usepackage{xcolor,graphicx}

\def\crulefill{\leavevmode\leaders\hrule height 1pt\hfill\kern 0pt}
\long\def\QUERY#1{%
\leavevmode\newline%
\noindent$\star\star\star$\thinspace\textsf{Comment/Query}\crulefill\newline%
   \space #1\newline\hbox to 120mm{\crulefill}$\star\star\star$\newline}
\newtheorem{Theorem}{Theorem}[section]%[chapter] theorem number will %continue
\newtheorem{Lemma}[Theorem]{Lemma}
\newtheorem{Cor}[Theorem]{Corollary}
\newtheorem{Prop}[Theorem]{Proposition}

 \theoremstyle{definition}

\newtheorem{Defn}[Theorem]{Definition}

\numberwithin{equation}{section}
\theoremstyle{definition}
\makeatletter
\def\enumerate{\begingroup\ifnum\@enumdepth>3\@toodeep\else
      \advance\@enumdepth\@ne
      \edef\@enumctr{enum\romannumeral\the\@enumdepth}%
      \topsep\z@\parskip\z@
      \list{\csname label\@enumctr\endcsname}
        {\@nmbrlisttrue\let\@listctr\@enumctr
         \parsep\z@\itemsep\z@\topsep\z@
         \setcounter{\@enumctr}{0}
         \def\makelabel##1{\hss\llap{\rm ##1}}
       }\fi}

\makeatother

\let\bar=\overline
\let\epsilon=\varepsilon
\def\({\big(}
\def\){\big)}

\def\0{\underline{0}}

\DeclareMathOperator{\End}{End}

\def\Std{\mathscr{T}^{std}}

\def\s{\mathfrak s}

\def\t{\mathfrak t}
\def\u{\mathfrak u}

\def\Hom{\text{Hom}}

\def\U{\mathbf U}

{\catcode`\|=\active
  \gdef\set#1{\mathinner{\lbrace\,{\mathcode`\|"8000%
                                   \let|\midvert #1}\,\rbrace}}
  \gdef\seT#1{\mathinner{\Big\lbrace\,{\mathcode`\|"8000%
                                   \let|\midverT #1}\,\Big\rbrace}}
}
\def\midvert{\egroup\mid\bgroup}
\def\midverT{\egroup\,\Big|\,\bgroup}

\def\Set[#1]#2|#3|{\Big\{\ #2\ \Big| \
           \vcenter{\hsize #1mm\centering #3}\Big\}}

\def\Hom{{\rm Hom}}

\def\Set{{\rm Set}}

\newcommand{\C}{\mathcal{C}}

\def\Std{\mathscr{T}^{std}}%
\def\s{\mathfrak s}%
\def\t{\mathfrak t}%
\def\u{\mathfrak u}%
\def\Hom{\text{Hom}}%
\def\U{\mathbf U}%
\def\textsf#1{{\textit{#1}}}%

\definecolor{white}{HTML}{FFFFFF}
\definecolor{darkblue}{HTML}{111199}
\definecolor{darkgreen}{HTML}{336633}
\definecolor{darkred}{HTML}{993333}

\definecolor{darkpurple}{HTML}{995599}

\begin{document}
\title[{\tiny  The periplectic $q$-Brauer category}]{The periplectic $q$-Brauer category}
\author{ Hebing Rui and  Linliang Song} \address{H.R.  School of Mathematical Science, Tongji University,  Shanghai, 200092, China}\email{hbrui@tongji.edu.cn}
\address{L.S.  School of Mathematical Science, Tongji University,  Shanghai, 200092, China}\email{llsong@tongji.edu.cn}
\thanks{ H. Rui is supported  partially by NSFC (grant No.  11971351).  L. Song is supported  partially by NSFC (grant No.  12071346) and  China Scholarship Council. }
\date{\today}
\sloppy
\maketitle

\begin{abstract} We introduce the periplectic $q$-Brauer category over an integral domain of characteristic not $2$. This  is a  strict monoidal supercategory and   can be considered as a  $q$-analogue of the periplectic  Brauer category in \cite{KT}.   We prove that the  periplectic $q$-Brauer category admits a split triangular decomposition in the sense of \cite{BS}. When the ground ring  is an algebraically closed field, the category of locally finite dimensional right modules
for the periplectic $q$-Brauer category is an upper finite fully stratified category in the sense of \cite{BS}.  We prove that   periplectic $q$-Brauer algebras defined in \cite{AGG} are isomorphic to  endomorphism algebras in the periplectic $q$-Brauer category.  Furthermore, a  periplectic $q$-Brauer algebra  is a standardly based algebra  in the sense of~\cite{DR}. We  construct  Jucys-Murphy basis for any  standard module of  the   periplectic $q$-Brauer algebra with respect to a family of commutative elements called Jucys-Murphy elements. Via them, we  classify blocks for both  periplectic  $q$-Brauer category and    periplectic $q$-Brauer algebras in generic case. Our result shows that  both  periplectic  $q$-Brauer category and    periplectic $q$-Brauer algebras are always not semisimple over any algebraically closed  field.
   \end{abstract}

\maketitle
\section{Introduction} The  periplectic Brauer algebra was introduced by Moon when he studied tensor product of the natural module for  the periplectic Lie superalgebra $\mathfrak p_n$~\cite{Moon}.  Later on,  Kujawa and Tharp \cite{KT} introduced the periplectic Brauer category and proved that any  periplectic Brauer algebra appears  as an endomorphism algebra of a corresponding object in the periplectic Brauer category.
 Since then the representation theory of periplectic Brauer algebras have been studied extensively~\cite{Col,Col1,Col2,BCV}, including the classification of  blocks, decomposition numbers and the related (weak) categorification. These results have important applications in the representation theory of $\mathfrak p_n$. See \cite{Col} for  the classification of  blocks in  the category of finite dimensional  $\mathfrak p_n$-modules.

 There are many directions to generalize previous results. For example,
 the affine version of periplectic Brauer algebra (resp., periplectic Brauer category ) was introduced in \cite{CP} (resp.,\cite{B+9}) and    a
 general monoidal supercategory related to representations of $\mathfrak p_n$ was introduced in \cite{DKM} and has applications in the study of the Howe duality between $\mathfrak p_m$ and $\mathfrak p_n$~\cite{DKM1}. In \cite{AGG}, Ahmed-Grantcharov-Guay introduced the    periplectic $q$-Brauer algebra when they studied the tensor product of the natural module of the quantized enveloping superalgebra   $\U_q(\mathfrak p_n)$.
The aim of this paper is to introduce the periplectic $q$-Brauer category $\mathcal B$ and to study its representation theory.
As expected in  the introduction of  \cite{DKM}, the periplectic $q$-Brauer category is a $q$-analogue of the periplectic  Brauer category and the    periplectic $q$-Brauer algebra appears as an endomorphism algebra of a corresponding object in the periplectic $q$-Brauer category. Motivated by \cite{Col}, we expect that  the periplectic $q$-Brauer category will be a useful tool  to study finite dimensional $\U_q(\mathfrak p_n)$-modules. Details will be given elsewhere.

We briefly explain  this paper as follows. After introducing the notion of the   periplectic $q$-Brauer category  $\mathcal B$, we prove that any morphism set
in $\mathcal B$ is free over an integral  domain $\Bbbk$ of characteristic not $2$. As an application, we prove that the  periplectic $q$-Brauer algebra $\mathcal B_{q,l}$ is isomorphic to the endomorphism algebra of the object $l$ for any natural number $l$. Moreover, we prove that   $\mathcal B$ admits a split triangular decomposition in the sense of \cite[Remark~5.32]{BS}. When the ground ring $\Bbbk$ is an algebraically closed field, the previous result implies that the category of locally finite dimensional right $\mathcal B$-module is an upper finite fully stratified category in the sense of \cite[Definition~3.34]{BS}.  When $q$ is not a root of unity,  the category of locally finite dimensional right $\mathcal B$-module is an upper finite highest weight category in the sense of \cite[Definition~3.34]{BS}. In this case, we classify blocks of   $\mathcal B$. For this purpose, we construct a standard basis of $\mathcal B_{q,l}$ in the sense of \cite[Definition~1.2.1]{DR}. By studying classical branching rule for $\mathcal B_{q,l}$, we construct a nice   basis, called Jucys-Murphy basis  for  any standard module in the category of right $\mathcal B_{q,l}$-modules.
We also construct a family of commutative elements of $\mathcal B_{q,l}$ so that each of them acts upper-triangularly on the Jucys-Murphy basis of a standard module.  Therefore, these commutative elements are Jucys-Murphy elements of  $\mathcal B_{q,l}$ with respect to the Jucys-Muphy basis of standard modules for $\mathcal B_{q,l}$ in the sense of \cite[Definition~2.4]{MA1}\footnote{One can also construct  Jucys-Murphy basis of any standard module in the category of left $\mathcal B_{q,l}$-modules in a similar way. Gluing two kinds of Jucys-Murphy basis together, one will get the Jucys-Murphy basis of $\mathcal B_{q,l}$ in a standard way. We will not give details since we do not need it in this paper.}.   Restricting any standard module to the commutative subalgebra generated by Jucys-Murphy elements, we obtain partial results on the classification of blocks. When $q$ is not a root of unity, we investigate composition factors of certain standard modules for  $\mathcal B_{q,l}$ and finally obtain a classification of blocks for both  $\mathcal B$ and $\mathcal B_{q,l}$ for any natural numbers $l$.  In the later case, $q$ can be a root of unity. However, we need to assume that the quantum characteristic of $q^2$ is bigger than $l$. Our results which can be considered as a counterpart of that for  periplectic Brauer algebras in \cite{Col} may be used to  study   finite dimensional $\U_q(\mathfrak p_n)$-modules. %Details will be given elsewhere.

The paper is organized as follows. After recalling  the notion of   monoidal supercategories, we introduce the periplectic $q$-Brauer category $\mathcal B$ and give a spanning set of any  morphism space in section~2. In section~3, we construct a monoidal functor from  $\mathcal B$ to the category of  representations of $\U_q(\mathfrak p_n)$. Using this we prove that any morphism space of
$\mathcal B$ is free over $\Bbbk$ with required rank in section~4. As an application, we prove that     periplectic $q$-Brauer algebra $\mathcal B_{q,l}$ is isomorphic to $\End_{\mathcal B}(l)$  for any natural number $l$.
In section 5, we prove that $\mathcal B$ admits a split triangular decomposition in the sense of \cite{BS}. Consequently, we prove that the category of locally finite dimensional
right $\mathcal B$-modules is an upper finite fully stratified category if $\Bbbk$ is an algebraically closed field. In section 6, we study the classical  branching rule for $\mathcal B_{q,l}$. Via it, we construct the Jucys-Murphy basis of any  standard module in the category of right  $\mathcal B_{q,l}$-modules  in section~7.
 Jucys-Murphy elements act on the  Jucys-Murphy basis upper-triangularly. This gives a partial result on the classification of blocks in section~8. Finally, we assume $q$ is not a root of unity (resp., quantum characteristic of $q^2$ is bigger than $l$). We study composition factors of certain standard modules in details and finally obtain a classification of blocks for both     periplectic $q$-Brauer category and     periplectic $q$-Brauer algebras in generic case.

\textbf{Acknowledgement:} This paper is motivated by the definition of periplectic $q$-Brauer algebra~\cite[Definition~5.1]{AGG}. We wish to thank Dr. Y. Wang for  telling us this information.

\section{The periplectic  $q$-Brauer category}
The aim of  this section is to introduce  the periplectic  $q$-Brauer category.
Throughout,  $\Bbbk$ is an integral domain  of characteristic not $2$. First, we recall the notion of a strict monoidal supercategory in \cite{BCK}. See also \cite{BEA}
when  $\Bbbk$ is a field.
\subsection {A strict monoidal supercategory}
A $\Bbbk$-supermodule $M=M_{\bar 0}\oplus M_{\bar 1}$ is  a $\mathbb Z_2$-graded $\Bbbk$-module. Elements in $ M_{\bar 0}$ (resp., $ M_{\bar 1}$ ) have parity $\bar 0$ or  are said to be even (resp., parity $\bar 1$ or  odd).
Following \cite{BCK},     the parity
of a homogeneous element $m\in M$ is denoted by $[m]$. Given two $\Bbbk$-supermodules $M$ and $N$, the tensor product $M\otimes N$  is again a  $\Bbbk$-supermodules with the $\mathbb Z_2$-grading by declaring that $[m\otimes n]=[m]+[n]$ for all
homogeneous elements $(m, n)\in M\times N$. Given two homogeneous homomorphisms $f: M\rightarrow M'$ and $g: N\rightarrow N'$ between $\Bbbk$-supermodules,
the tensor product $f\otimes g$   is defined as
\begin{equation}\label{tensorofsuper}
(f\otimes g)(m\otimes n)=(-1)^{[g] [m]}f(m)\otimes g(n),
\end{equation}
where $[g]$ is the parity of $g$. The following is the   Koszul sign rule where
$f, g, h$ and $k$ are homogeneous homomorphisms:
\begin{equation}\label{superinterchange}
(f\otimes g)\circ (h\otimes k)=(-1)^{[g] [h]}(f\circ h)\otimes (g\circ k).
\end{equation}

  A supercategory  is a category enriched in $\Bbbk$-supermodules in the sense that each morphism set  is a $\Bbbk$-supermodule and composition induces an even $\Bbbk$-homomorphism.  A superfunctor $\mathcal F$ between two supercategories is an even functor enriched in $\Bbbk$-supermodules.

   \begin{Defn}\label{sc}  \cite[\S6]{BCK} A   strict monoidal supercategory  is a supercategory $\C$ equipped with  a  bi-superfunctor
$-\otimes -: \C \times  \C \rightarrow  \C$ and a unit object $\mathds{1}$
such that for all objects $a,b,c$  of $\C$, we have
 $(a\otimes b)\otimes c=a\otimes(b\otimes c)$ and
 $\mathds 1\otimes a=a=a\otimes \mathds 1$,  and for all morphisms $f,g$ and $h$ of $\C$, we have $(f\otimes g)\otimes h=f\otimes (g\otimes h)$ and
 $  1_{\mathds 1}\otimes f=f=f\otimes 1_{\mathds 1}$,
 where  $1_a: a\rightarrow a$ is  the identity morphism for  any object $a$.
\end{Defn}

%Suppose that $\mathcal C$ and $\mathcal C'$  are two monoidal categories such that  $\alpha$ and $\ob1$ (resp., $\alpha'$ and $1'$)  are the coherence map and the unit object of $\mathcal C$ (resp.,  $\mathcal C'$). A monoidal functor from $ \mathcal C$ to $ \mathcal C'$ is a pair $(F,J)$, where $F: \mathcal C\rightarrow \mathcal C'$ is a functor and
%$J_{ a, b}: F( a)\otimes F( b)\rightarrow F( a\otimes  b)$ is a natural isomorphism such that $F( 1)$ is isomorphic to $ 1'$ and
%$$F(\alpha_{ a, b, c})\circ J_{ a\otimes  b, c}\circ (J_{ a, b}\otimes 1_{F( c)})= J_{ a, b\otimes  c}\circ (1_{F( a)}\otimes J_{ b, c})\circ \alpha'_{F( a),F( b),F( c) }. $$

There is a well-defined graphical calculus for morphisms in a strict  monoidal supercategory $\C$ (c.f. \cite[\S1.2]{BEA}).
For any two objects $ a, b$ in  $\C$,  $ a b$ represents  $ a\otimes  b$.
    A  morphism $g: a\to b$ is drawn as a coupon label by $g$
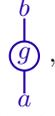
    $$\begin{tikzpicture}[baseline = 12pt,scale=0.5,color=\clr,inner sep=0pt, minimum width=11pt]
        \draw[-,thick] (0,0) to (0,2);
        \draw (0,1) node[circle,draw,thick,fill=white]{$g$};
        \draw (0,-0.2) node{$ a$};
        \draw (0, 2.3) node{$ b$};
    \end{tikzpicture}~,
     $$ where  $ a$ is  at the bottom and $ b$ is at the top.
The identity morphism $1_a:a\rightarrow a$ is drawn as a string with no coupon
$$\begin{tikzpicture}[baseline = 12pt,scale=0.5,color=\clr,inner sep=0pt, minimum width=11pt]
        \draw[-,thick] (0,0) to (0,2);
               \draw (0,-0.2) node{$ a$};
        \draw (0, 2.3) node{$ a$};
    \end{tikzpicture}.
     $$
 We often omit the  labels  of objects (i.e. $ a$ and $b$ above) if there is no confusion on the objects. Moreover, the axioms of a strict monoidal supercategory make it natural to omit the identity morphism $1_\mathds1$ of the unit object. So, any  morphism $g\in \End(\mathds1)$  will be drawn as a free-floating coupon:  $\begin{tikzpicture}[baseline = 12pt,scale=0.5,color=\clr,inner sep=0pt, minimum width=11pt]
              \draw (0,1) node[circle,draw,thick,fill=white]{$g$};
           \end{tikzpicture}. $
The composition  (resp., tensor product) of two morphisms  is given by vertical stacking (resp., horizontal concatenation) as follows:

\begin{equation}\label{com1}
      g\circ h= \begin{tikzpicture}[baseline = 19pt,scale=0.5,color=\clr,inner sep=0pt, minimum width=11pt]
        \draw[-,thick] (0,0) to (0,3);
        \draw (0,2.2) node[circle,draw,thick,fill=white]{$g$};
        \draw (0,0.8) node[circle,draw,thick,fill=white]{$h$};
        % \draw (0,-0.2) node{$ c$};
        % \draw (0.3,1.5) node{$ b$};
        % \draw (0, 3.3) node{$ a$};
    \end{tikzpicture}
    ~,~ \ \ \ \ \ g\otimes h=\begin{tikzpicture}[baseline = 19pt,scale=0.5,color=\clr,inner sep=0pt, minimum width=11pt]
        \draw[-,thick] (0,0) to (0,3);
        \draw[-,thick] (2,0) to (2,3);
        \draw (0,1.5) node[circle,draw,thick,fill=white]{$g$};
        \draw (2,1.5) node[circle,draw,thick,fill=white]{$h$};
        % \draw (0,-0.2) node{$ c$};
        % \draw (0, 3.3) node{$ d$};
        % \draw (2,-0.2) node{$ a$};
        % \draw (2, 3.3) node{$ b$};
    \end{tikzpicture}~.
\end{equation}
%\begin{equation}\label{ts1}
% \end{equation}
% To simplify the notation, $g^{\otimes r} $ is drawn as
%\begin{equation}\label{ts2}
%    \begin{tikzpicture}[baseline = 12pt,scale=0.5,color=\clr,inner sep=0pt, minimum width=11pt]
%        \draw[-,thick] (0,0) to (0,2);
%        \draw (0,1) node[circle,draw,thick,fill=white]{$g$};
%    \end{tikzpicture}^r
%        ~.
%\end{equation}
 Since $-\otimes -: \C\times\C \rightarrow \C$ is a superfunctor, there is a  super interchange law
\begin{equation}\label{wwcirrten}
 (f\otimes g)\circ(k\otimes h)=(-1)^{[k][g]}(f\circ k)\otimes (g\circ h),
\end{equation}
for any homogeneous morphisms $f, g, k$ and $h$.
The left side of \eqref{wwcirrten} is drawn as
 
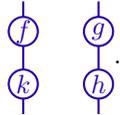
\begin{equation*}
    \begin{tikzpicture}[baseline = 19pt,scale=0.5,color=\clr,inner sep=0pt, minimum width=11pt]
        \draw[-,thick] (0,0) to (0,3);
        \draw[-,thick] (2,0) to (2,3);
        \draw (0,2.2) node[circle,draw,thick,fill=white]{$f$};
        \draw (2,2.2) node[circle,draw,thick,fill=white]{$g$};
        \draw (0,0.8) node[circle,draw,thick,fill=white]{$k$};
        \draw (2,0.8) node[circle,draw,thick,fill=white]{$h$};
    \end{tikzpicture}.
\end{equation*}
which is not the same as $(f\circ k)\otimes (g\circ h)$ in general.
By the super interchange law, we have
\begin{equation}
  \begin{tikzpicture}[baseline = 19pt,scale=0.5,color=\clr,inner sep=0pt, minimum width=11pt]
        \draw[-,thick] (0,0) to (0,3);
        \draw[-,thick] (2,0) to (2,3);
        \draw (0,1.8) node[circle,draw,thick,fill=white]{$f$};
        \draw (2,1.2) node[circle,draw,thick,fill=white]{$g$};
        \end{tikzpicture} = \begin{tikzpicture}[baseline = 19pt,scale=0.5,color=\clr,inner sep=0pt, minimum width=11pt]
        \draw[-,thick] (0,0) to (0,3);
        \draw[-,thick] (2,0) to (2,3);
        \draw (0,1.2) node[circle,draw,thick,fill=white]{$f$};
        \draw (2,1.2) node[circle,draw,thick,fill=white]{$g$};
        \end{tikzpicture}=(-1)^{[f][g]}\begin{tikzpicture}[baseline = 19pt,scale=0.5,color=\clr,inner sep=0pt, minimum width=11pt]
        \draw[-,thick] (0,0) to (0,3);
        \draw[-,thick] (2,0) to (2,3);
        \draw (0,1.2) node[circle,draw,thick,fill=white]{$f$};
        \draw (2,1.8) node[circle,draw,thick,fill=white]{$g$};
        \end{tikzpicture}.
\end{equation}

 %  Recall that  a  right tensor ideal $I$ of a $\Bbbk$-linear strict monoidal supercategory $\mathcal C$ is the collection of  super $\Bbbk$-submodules $ \{I( a, b)\mid I( a,  b)\subset  \Hom_{\mathcal C}( a, b),  \forall  a, b \in \mathcal C\}$, such that
%\begin{equation}\label{idea} h_3\circ h_2\circ h_1\in I( a , d), \text{ and $ h_2\otimes 1_{ d}\in I( b\otimes  d , c\otimes  d)$}\end{equation}
% whenever $(h_1, h_2, h_3)
% \in  \Hom_{\mathcal C}( a, b)\times I( b, c)\times \Hom_{\mathcal C}( c, d)$ for all  objects $ a, b,  c, d$ in $\C$. Similarly we have the notions of  a left tensor ideal and a two-sided tensor ideal.

    \subsection{The periplectic $q$-Brauer category}
     Throughout, we assume $q,q-q^{-1}\in \Bbbk^\times$ where $\Bbbk^\times $ is the set of all invertible elements in $\Bbbk$.

    \begin{Defn}\label{pqbc}
    The  periplectic $q$-Brauer category $\mathcal B$ is a strict monoidal supercategory generated   by a single object  \begin{tikzpicture}[baseline = 1.5mm]
	\draw[-,thick,darkblue] (0.18,0) to (0.18,.4);
\end{tikzpicture}
and two even  morphisms
  $$ \begin{tikzpicture}[baseline = 2.5mm]
	\draw[-,thick,darkblue] (0.28,0) to[out=90,in=-90] (-0.28,.6);
	\draw[-,line width=4pt,white] (-0.28,0) to[out=90,in=-90] (0.28,.6);
	\draw[-,thick,darkblue] (-0.28,0) to[out=90,in=-90] (0.28,.6);
\end{tikzpicture}:\  \begin{tikzpicture}[baseline = 2.5mm]
	\draw[-,thick,darkblue] (0.18,0) to (0.18,.6);
\end{tikzpicture} \ \ \begin{tikzpicture}[baseline = 2.5mm]
	\draw[-,thick,darkblue] (0.18,0) to (0.18,.6);
\end{tikzpicture}\longrightarrow\begin{tikzpicture}[baseline = 2.5mm]
	\draw[-,thick,darkblue] (0.18,0) to (0.18,.6);
\end{tikzpicture}\ \  \begin{tikzpicture}[baseline = 2.5mm]
	\draw[-,thick,darkblue] (0.18,0) to (0.18,.6);
\end{tikzpicture}\quad\text{ and }\quad
\begin{tikzpicture}[baseline = 2.5mm]
	\draw[-,thick,darkblue] (-0.28,-.0) to[out=90,in=-90] (0.28,0.6);
	\draw[-,line width=4pt,white] (0.28,-.0) to[out=90,in=-90] (-0.28,0.6);
	\draw[-,thick,darkblue] (0.28,-.0) to[out=90,in=-90] (-0.28,0.6);
\end{tikzpicture}:\ \begin{tikzpicture}[baseline = 2.5mm]
	\draw[-,thick,darkblue] (0.18,0) to (0.18,.6);
\end{tikzpicture}\ \ \begin{tikzpicture}[baseline = 2.5mm]
	\draw[-,thick,darkblue] (0.18,0) to (0.18,.6);
\end{tikzpicture}\longrightarrow\begin{tikzpicture}[baseline = 2.5mm]
	\draw[-,thick,darkblue] (0.18,0) to (0.18,.6);
\end{tikzpicture}\ \  \begin{tikzpicture}[baseline = 2.5mm]
	\draw[-,thick,darkblue] (0.18,0) to (0.18,.6);
\end{tikzpicture}\,$$
 and two odd morphisms
 $$\lcup:\ \mathds{1}\longrightarrow  \begin{tikzpicture}[baseline = 2.5mm]
	\draw[-,thick,darkblue] (0.18,0) to (0.18,.6);
\end{tikzpicture}\ \ \begin{tikzpicture}[baseline = 2.5mm]
	\draw[-,thick,darkblue] (0.18,0) to (0.18,.6);
\end{tikzpicture}\quad \text{ and }\quad
 \lcap:\  \begin{tikzpicture}[baseline = 2.5mm]
	\draw[-,thick,darkblue] (0.18,0) to (0.18,.6);
\end{tikzpicture}\ \ \begin{tikzpicture}[baseline = 2.5mm]
	\draw[-,thick,darkblue] (0.18,0) to (0.18,.6);
\end{tikzpicture}\longrightarrow \mathds{1} ,$$
 subject to the following defining relations:

 \begin{itemize}
 \item [(R1)] The braid relations:
 $\begin{tikzpicture}[baseline = -1mm]
	\draw[-,thick,darkblue] (0.28,-.6) to[out=90,in=-90] (-0.28,0);
	\draw[-,thick,darkblue] (-0.28,0) to[out=90,in=-90] (0.28,.6);
	\draw[-,line width=4pt,white] (-0.28,-.6) to[out=90,in=-90] (0.28,0);
	\draw[-,thick,darkblue] (-0.28,-.6) to[out=90,in=-90] (0.28,0);
	\draw[-,line width=4pt,white] (0.28,0) to[out=90,in=-90] (-0.28,.6);
	\draw[-,thick,darkblue] (0.28,0) to[out=90,in=-90] (-0.28,.6);
\end{tikzpicture}=
\mathord{
\begin{tikzpicture}[baseline = -1mm]
	\draw[-,thick,darkblue] (0.18,-.6) to (0.18,.6);
	\draw[-,thick,darkblue] (-0.18,-.6) to (-0.18,.6);
\end{tikzpicture}
}\:=
\mathord{
\begin{tikzpicture}[baseline = -1mm]
	\draw[-,thick,darkblue] (0.28,0) to[out=90,in=-90] (-0.28,.6);
	\draw[-,line width=4pt,white] (-0.28,0) to[out=90,in=-90] (0.28,.6);
	\draw[-,thick,darkblue] (-0.28,0) to[out=90,in=-90] (0.28,.6);
	\draw[-,thick,darkblue] (-0.28,-.6) to[out=90,in=-90] (0.28,0);
	\draw[-,line width=4pt,white] (0.28,-.6) to[out=90,in=-90] (-0.28,0);
	\draw[-,thick,darkblue] (0.28,-.6) to[out=90,in=-90] (-0.28,0);
\end{tikzpicture}}\quad \text{and }\quad \mathord{
\begin{tikzpicture}[baseline = -1mm]
	\draw[-,thick,darkblue] (0.45,-.6) to (-0.45,.6);
        \draw[-,thick,darkblue] (0,-.6) to[out=90,in=-90] (-.45,0);
        \draw[-,line width=4pt,white] (-0.45,0) to[out=90,in=-90] (0,0.6);
        \draw[-,thick,darkblue] (-0.45,0) to[out=90,in=-90] (0,0.6);
	\draw[-,line width=4pt,white] (0.45,.6) to (-0.45,-.6);
	\draw[-,thick,darkblue] (0.45,.6) to (-0.45,-.6);
\end{tikzpicture}
}
\begin{tikzpicture}[baseline = -1mm]
	 \node at (0.,0) {$\text {=}$};
\end{tikzpicture}
\mathord{
\begin{tikzpicture}[baseline = -1mm]
	\draw[-,thick,darkblue] (0.45,-.6) to (-0.45,.6);
        \draw[-,line width=4pt,white] (0,-.6) to[out=90,in=-90] (.45,0);
        \draw[-,thick,darkblue] (0,-.6) to[out=90,in=-90] (.45,0);
        \draw[-,thick,darkblue] (0.45,0) to[out=90,in=-90] (0,0.6);
	\draw[-,line width=4pt,white] (0.45,.6) to (-0.45,-.6);
	\draw[-,thick,darkblue] (0.45,.6) to (-0.45,-.6);
\end{tikzpicture}},$
\item [(R2)] The skein relation:
$\mathord{
\begin{tikzpicture}[baseline = 2.5mm]
	\draw[-,thick,darkblue] (0.28,0) to[out=90,in=-90] (-0.28,.6);
	\draw[-,line width=4pt,white] (-0.28,0) to[out=90,in=-90] (0.28,.6);
	\draw[-,thick,darkblue] (-0.28,0) to[out=90,in=-90] (0.28,.6);
\end{tikzpicture}
}-\mathord{
\begin{tikzpicture}[baseline = 2.5mm]
	\draw[-,thick,darkblue] (-0.28,-.0) to[out=90,in=-90] (0.28,0.6);
	\draw[-,line width=4pt,white] (0.28,-.0) to[out=90,in=-90] (-0.28,0.6);
	\draw[-,thick,darkblue] (0.28,-.0) to[out=90,in=-90] (-0.28,0.6);
\end{tikzpicture}
}\begin{tikzpicture}[baseline = -1mm]
	 \node at (0.,0) {$\text {=}$};
 \end{tikzpicture}
(q-q^{-1})\:\mathord{
\begin{tikzpicture}[baseline = 2.5mm]
	\draw[-,thick,darkblue] (0.18,0) to (0.18,.6);
	\draw[-,thick,darkblue] (-0.18,0) to (-0.18,.6);
\end{tikzpicture}}\ ,
$
\item [(R3)] The snake relations:
$\begin{tikzpicture}[baseline = -0.5mm]
%\draw[-,thick,darkblue] (0,0) to (0,0.3);
\draw[-,thick,darkblue] (0.5,0) to[out=down,in=left] (0.75,-0.25) to[out=right,in=down] (1,0);
\draw[-,thick,darkblue] (0,0) to[out=up,in=left] (0.25,0.25) to[out=right,in=up] (0.5,0);
\draw[-,thick,darkblue] (1,0) to (1,0.3);
\draw[-,thick,darkblue] (0,0) to (0,-0.3);
  \end{tikzpicture}~=~\begin{tikzpicture}[baseline = -0.5mm]
  \draw[-,thick,darkblue] (0,-0.25) to (0,0.25);
  \end{tikzpicture}\quad \text{ and }\quad \begin{tikzpicture}[baseline = -0.5mm]
\draw[-,thick,darkblue] (0,0) to (0,0.3);
\draw[-,thick,darkblue] (0,0) to[out=down,in=left] (0.25,-0.25) to[out=right,in=down] (0.5,0);
\draw[-,thick,darkblue] (0.5,0) to[out=up,in=left] (0.75,0.25) to[out=right,in=up] (1,0);
\draw[-,thick,darkblue] (0,0) to (0,0.3);
\draw[-,thick,darkblue] (1,0) to (1,-0.3);
  \end{tikzpicture}
  ~=~
-  \begin{tikzpicture}[baseline = -0.5mm]
  \draw[-,thick,darkblue] (0,-0.25) to (0,0.25);
  \end{tikzpicture}\ ,$
\item [(R4)] The untwisting relations:
$\begin{tikzpicture}[baseline = 6pt, scale=0.5, color=\clr]
        \draw[-,thick] (1,0.4) to[out=down, in=right] (0.53,-0.2) to[out=left, in=right] (0.47,-0.2);
        \draw[-,thick] (0.49,-0.2) to[out=left,in=down] (0,0.4);
        \draw[-,thick](0,1) to (1,0.4);
        \draw[-,line width=4pt,white](0,0.4) to (1,1);
        \draw[-,thick](0,0.4) to (1,1);
                \end{tikzpicture}~=~ q~\begin{tikzpicture}[baseline = 6pt, scale=0.5, color=\clr]
        \draw[-,thick] (1,0.8) to[out=down, in=right] (0.53,0.3) to[out=left, in=right] (0.47,0.3);
        \draw[-,thick] (0.49,0.3) to[out=left,in=down] (0,.8);
\end{tikzpicture}\quad\text{ and }\quad
\begin{tikzpicture}[baseline = 6pt, scale=0.5, color=\clr]
         \draw[-,thick,darkblue] (1.25,0) to[out=90,in=-90] (0.5,1);
         \draw[-,line width=4pt,white](1,1) to[out=down, in=right] (0.53,0) to[out=left, in=right] (0.47,0);
        \draw[-,thick] (1,1) to[out=down, in=right] (0.53,0) to[out=left, in=right] (0.47,0);
        \draw[-,thick] (0.49,0) to[out=left,in=down] (0,1);
\end{tikzpicture}=
  \begin{tikzpicture}[baseline = 6pt, scale=0.5, color=\clr]
        \draw[-,thick,darkblue] (-0.25,0) to[out=90,in=-90] (0.55,1);
        \draw[-,thick] (1,1) to[out=down, in=right] (0.53,0) to[out=left, in=right] (0.47,0);
        \draw[-,line width=4pt,white](0.49,0) to[out=left,in=down] (0,1);
        \draw[-,thick] (0.49,0) to[out=left,in=down] (0,1);
  \end{tikzpicture}\ ,$
  \item[(R5)] The loop removing  relation: $\mathord{
\begin{tikzpicture}[baseline = -0.5mm]
\draw[-,thick,darkblue] (0,0) to[out=down,in=left] (0.25,-0.25) to[out=right,in=down] (0.5,0);
  \draw[-,thick,darkblue] (0,0) to[out=up,in=left] (0.25,0.25) to[out=right,in=up] (0.5,0);
  \end{tikzpicture}}\begin{tikzpicture}[baseline = -1mm]
	 \node at (0.,0) {$\text {=}$};
 \end{tikzpicture}0$.
 \end{itemize}
  \end{Defn}
 Since  $\mathcal B$ is generated by a  single object
 \begin{tikzpicture}[baseline = 10pt, scale=0.5, color=\clr]
                \draw[-,thick] (0,0.5)to[out=up,in=down](0,1.2);
    \end{tikzpicture},  the set of  objects  is $\mathbb N$ where $0$ is the unit object
$\begin{tikzpicture}[baseline = 10pt, scale=0.5, color=\clr]
                \draw[-,thick] (0,0.5)to[out=up,in=down](0,1.2);
    \end{tikzpicture}^{\otimes 0}$  and $m$ is
$~ \begin{tikzpicture}[baseline = 10pt, scale=0.5, color=\clr]
                \draw[-,thick] (0,0.5)to[out=up,in=down](0,1.2);
                    \end{tikzpicture}^{\otimes m}$.
This notation is different from that in Definition~\ref{sc} where the unit object is denoted by  $\mathds{1}$. In the current case,
 each endpoint  at both rows of the  diagrams $\begin{tikzpicture}[baseline = 2.5mm]
	\draw[-,thick,darkblue] (0.28,0) to[out=90,in=-90] (-0.28,.6);
	\draw[-,line width=4pt,white] (-0.28,0) to[out=90,in=-90] (0.28,.6);
	\draw[-,thick,darkblue] (-0.28,0) to[out=90,in=-90] (0.28,.6);
\end{tikzpicture}, \begin{tikzpicture}[baseline = 2.5mm]
	\draw[-,thick,darkblue] (-0.28,-.0) to[out=90,in=-90] (0.28,0.6);
	\draw[-,line width=4pt,white] (0.28,-.0) to[out=90,in=-90] (-0.28,0.6);
	\draw[-,thick,darkblue] (0.28,-.0) to[out=90,in=-90] (-0.28,0.6);
\end{tikzpicture}$, $\lcup,\lcap$ represents the object $ 1$. If there is no endpoints at a row, then the object at this row is the unit object $0$. For example, the object  at the bottom (resp., top) row of $\lcup$ is $ 0$ (resp., $2$). Note that   the identity morphism $1_{\ob m}$  is drawn as the object itself when $m>0$. For example, $\begin{tikzpicture}[baseline = 10pt, scale=0.5, color=\clr]
                \draw[-,thick] (0,0.5)to[out=up,in=down](0,1.2);\draw[-,thick] (0.5,0.5)to[out=up,in=down](0.5,1.2);
    \end{tikzpicture}$ in (R1) is the identity morphism $1_2$.

   \begin{Lemma}\label{basicrel} In $\mathcal B$, we have
\begin{multicols}{2} \item [(1)] $\begin{tikzpicture}[baseline = 6pt, scale=0.6, color=\clr]
        \draw[-,thick] (1,0) to[out=up, in=right] (0.53,0.9) to[out=left, in=right] (0.47,0.9);
        \draw[-,thick] (0.49,0.9) to[out=left,in=up] (0,0);
        \draw[-,thick] (2.5,0) to[out=up, in=right] (2.03,0.5) to[out=left, in=right] (1.97,0.5);
        \draw[-,thick] (1.99,0.5) to[out=left,in=up] (1.5,0);
\end{tikzpicture}~=~\begin{tikzpicture}[baseline = 6pt, scale=0.6, color=\clr]
        \draw[-,thick] (1,0) to[out=up, in=right] (0.53,0.5) to[out=left, in=right] (0.47,0.5);
        \draw[-,thick] (0.49,0.5) to[out=left,in=up] (0,0);
        \draw[-,thick] (2.5,0) to[out=up, in=right] (2.03,0.5) to[out=left, in=right] (1.97,0.5);
        \draw[-,thick] (1.99,0.5) to[out=left,in=up] (1.5,0);
\end{tikzpicture}~=~-~\begin{tikzpicture}[baseline = 6pt, scale=0.6, color=\clr]
        \draw[-,thick] (1,0) to[out=up, in=right] (0.53,0.5) to[out=left, in=right] (0.47,0.5);
        \draw[-,thick] (0.49,0.5) to[out=left,in=up] (0,0);
        \draw[-,thick] (2.5,0) to[out=up, in=right] (2.03,0.9) to[out=left, in=right] (1.97,0.9);
        \draw[-,thick] (1.99,0.9) to[out=left,in=up] (1.5,0);
\end{tikzpicture}$,

\item[(2)] \begin{tikzpicture}[baseline = 7pt, scale=0.7, color=\clr]
\draw[-,thick] (0.6,0) to (1.2,0.8);
\draw[-,line width=4pt,white](1.2,.3) to[out=left,in=down] (1,.6);
        \draw[-,thick] (1.2,0) to[out=up, in=right] (0.63,0.6) to[out=left, in=right] (0.57,0.6);
        \draw[-,thick] (0.59,0.6) to[out=left,in=up] (0,0);
  \end{tikzpicture} =
     \begin{tikzpicture}[baseline = 7pt, scale=0.7, color=\clr]
\draw[-,thick] (0.5,0) to (0,0.8);
\draw[-,line width=4pt,white](.3,.4) to[out=left,in=down] (.1,.3);
        \draw[-,thick] (1.2,0) to[out=up, in=right] (0.63,0.6) to[out=left, in=right] (0.57,0.6);
        \draw[-,thick] (0.59,0.6) to[out=left,in=up] (0,0);
  \end{tikzpicture},

\item[(3)]
  $\begin{tikzpicture}[baseline = -0.3mm]
\draw[-,thick,darkblue] (-0.28,-.4) to[out=90,in=-90] (0.28,.2);
	\draw[-,line width=4pt,white](0.28,-0.4) to[out=90,in=-90] (-0.28,.2);
\draw[-,thick,darkblue] (0.28,-0.4) to[out=90,in=-90] (-0.28,.2);
        \draw[-,thick,darkblue] (0.28,0.2) to[out=up, in=right] (0.1,0.4) to[out=left, in=right] (0,0.4);
        \draw[-,thick,darkblue] (0,0.4) to[out=left,in=up] (-0.28,0.2);
  \end{tikzpicture}
  ~=~-q~\begin{tikzpicture}[baseline = -0.3mm, color=\clr]
        \draw[-,thick] (1,-0.3) to[out=up, in=right] (0.53,0.2) to[out=left, in=right] (0.47,0.2);
        \draw[-,thick] (0.49,0.2) to[out=left,in=up] (0,-0.3);
\end{tikzpicture} $,
\item[(4)]$\begin{tikzpicture}[baseline = -0.5mm]\draw[-,thick,darkblue] (0.3,.2) to [out=180,in=90](0,-0.3);
\draw[-,line width=4pt,white] (-0.1,0.2) to (.2,-0.1);
	\draw[-,thick,darkblue] (0,0.6) to (0,0.3);
\draw[-,thick,darkblue] (0.5,0) to [out=90,in=0](.3,0.2);
	\draw[-,thick,darkblue] (0,-0.3) to (0,-0.6);
\draw[-,thick,darkblue] (0.3,-0.2) to [out=0,in=-90](.5,0);
\draw[-,thick,darkblue] (0,0.3) to [out=-90,in=180] (.3,-0.2);
\end{tikzpicture}~=~\begin{tikzpicture}[baseline = -0.5mm]
	\draw[-,thick,darkblue] (0,0.6) to (0,0.3);
	\draw[-,thick,darkblue] (0.5,0) to [out=90,in=0](.3,0.2);
	\draw[-,thick,darkblue] (0,-0.3) to (0,-0.6);
	\draw[-,thick,darkblue] (0.3,-0.2) to [out=0,in=-90](.5,0);
	\draw[-,thick,darkblue] (0,0.3) to [out=-90,in=180] (.3,-0.2);
	\draw[-,line width=4pt,white] (0.3,.2) to [out=180,in=90](0,-0.3);
	\draw[-,thick,darkblue] (0.3,.2) to [out=180,in=90](0,-0.3);\end{tikzpicture}~=~
q^{-1}~\begin{tikzpicture}[baseline = -0.5mm]
	\draw[-,thick,darkblue] (0,0.6) to (0,-0.6);
\end{tikzpicture}$, \item[(5)]
$\begin{tikzpicture}[baseline = -0.5mm]
        \draw[-,thick,darkblue](-0.4,0) to[out=up, in=left] (-0.2,0.3) to[out=right, in=up](0.05,0);
   \draw[-,thick,darkblue](-0.4,0) to (-0.4,-0.05);
        \draw[-,line width=4pt,white] (-0.05,-0.05) to (0.05,0.05);
              % \draw[-,thick,darkblue](-0.05,-0.05) to (0.1,0.1);
               \draw[-,thick,darkblue] (0.1,0.1) to(0.1,0.6);
%\draw[-,thick,darkblue](0,0) to (0.05,-0.05);
        \draw[-,thick,darkblue] (0.11,-0.14)to (0.11,-0.6);
        \draw[-,thick,darkblue](-0.4,-0.05) to[out=down, in=left] (-0.18,-0.3) to[out=right, in=down](0.1,0.1);
          \end{tikzpicture}~=~\begin{tikzpicture}[baseline = -0.5mm]
        \draw[-,thick,darkblue](-0.4,0) to[out=down, in=left] (-0.2,-0.3) to[out=right, in=down](0,0);
   \draw[-,thick,darkblue](-0.4,0) to (-0.4,0.05);\draw[-,line width=4pt,white] (-0.05,0.05) to (0.05,-0.05);
   \draw[-,thick,darkblue](-0.4,0.05) to[out=up, in=left] (-0.18,0.3) to[out=right, in=up](0.06,-0.1);
\draw[-,thick,darkblue] (0.07,-0.1)to (0.07,-0.6);
\draw[-,thick,darkblue] (0.08,0.1) to(0.08,0.6);
          \end{tikzpicture}~=~-q~\begin{tikzpicture}[baseline = -0.5mm]
	\draw[-,thick,darkblue] (0,0.6) to (0,-0.6);
\end{tikzpicture}$.\end{multicols}
\end{Lemma}
\begin{proof} The  supercategroy structure means that the height moves satisfy the super interchange law  in \eqref{wwcirrten}. This proves (1).
We have
$$
\begin{tikzpicture}[baseline = 3pt, scale=0.5, color=\clr]
\draw[-,thick] (0.6,0) to (1.2,0.8);
\draw[-,line width=4pt,white](1.2,.3) to[out=left,in=down] (1,.6);
        \draw[-,thick] (1.2,0) to[out=up, in=right] (0.63,0.6) to[out=left, in=right] (0.57,0.6);
        \draw[-,thick] (0.59,0.6) to[out=left,in=up] (0,0);
  \end{tikzpicture} ~\overset{\text{(R3)}}=~- ~ \begin{tikzpicture}[baseline = 3pt, scale=0.5, color=\clr]
\draw[-,thick] (0.6,0) to (1.2,0.8);
\draw[-,line width=4pt,white](1.2,.3) to[out=left,in=down] (1,.6);
        \draw[-,thick] (1.2,0) to[out=up, in=right] (0.63,0.6) to[out=left, in=right] (0.57,0.6);
        \draw[-,thick] (0.59,0.6) to[out=left,in=up] (0,0);
       \draw[-,thick,darkblue] (1.2,0) to[out=down,in=left] (1.55,-0.6) to[out=right,in=down] (1.8,-0.1);
\draw[-,thick,darkblue] (1.8,-0.1) to[out=up,in=left] (2.15,0.25) to[out=right,in=up] (2.4,0);
\draw[-,thick](2.4,0) to(2.4, -0.6);
          \end{tikzpicture}~\overset{(1)}=
    ~ \begin{tikzpicture}[baseline = 3pt, scale=0.5, color=\clr]
\draw[-,thick] (0.6,0) to (1.2,0.8);
\draw[-,line width=4pt,white](1.2,.3) to[out=left,in=down] (1,.6);
        \draw[-,thick] (1.2,0) to[out=up, in=right] (0.63,0.6) to[out=left, in=right] (0.57,0.6);
        \draw[-,thick] (0.59,0.6) to[out=left,in=up] (0,0);
        \draw[-,thick,darkblue] (1.2,0) to[out=down,in=left] (1.55,-0.6) to[out=right,in=down] (1.8,-0.1);
\draw[-,thick,darkblue] (1.8,-0.1) to[out=up,in=left] (2.15,1) to[out=right,in=up] (2.4,0);
\draw[-,thick](2.4,0) to(2.4, -0.6);
        \end{tikzpicture}
        \overset{\text{(R4)}} =\begin{tikzpicture}[baseline = 3pt, scale=0.5, color=\clr]
%\draw[-,thick] (0.6,0) to (1.2,0.8);
%\draw[-,line width=4pt,white](1.2,.3) to[out=left,in=down] (1,.6);
        \draw[-,thick] (1.2,0) to[out=up, in=right] (0.63,0.6) to[out=left, in=right] (0.57,0.6);
        \draw[-,thick] (0.59,0.6) to[out=left,in=up] (0,0);
        \draw[-,thick] (1.55,0.2) to (2.1,-0.6);
      \draw[-,line width=2pt,white] (1.8,-0.6)to (1.8,0.2);
        \draw[-,thick,darkblue] (1.2,0) to[out=down,in=left] (1.55,-0.6) to[out=right,in=down] (1.8,-0.1);
\draw[-,thick,darkblue] (1.8,-0.1) to[out=up,in=left] (2.15,1) to[out=right,in=up] (2.4,0);
\draw[-,thick](2.4,0) to(2.4, -0.6);
        \end{tikzpicture}~\overset{\text{(R3)}}=~  \begin{tikzpicture}[baseline = 3pt, scale=0.5, color=\clr]
\draw[-,thick] (0.5,0) to (0,0.8);
\draw[-,line width=4pt,white](.3,.4) to[out=left,in=down] (.1,.3);
        \draw[-,thick] (1.2,0) to[out=up, in=right] (0.63,0.6) to[out=left, in=right] (0.57,0.6);
        \draw[-,thick] (0.59,0.6) to[out=left,in=up] (0,0);
  \end{tikzpicture},
$$ and  (2) follows.
The equation in (3) follows from the following equalities:
 $$\begin{aligned}
 \begin{tikzpicture}[baseline = -0.5mm]
 \draw[-,thick,darkblue] (-0.28,-.3) to[out=90,in=-90] (0.28,.3);
	\draw[-,line width=4pt,white](0.28,-0.3) to[out=90,in=-90] (-0.28,.3);
\draw[-,thick,darkblue] (0.28,-0.3) to[out=90,in=-90] (-0.28,.3);
	        \draw[-,thick,darkblue] (0.28,0.3) to[out=up, in=right] (0.1,0.5) to[out=left, in=right] (0,0.5);
        \draw[-,thick,darkblue] (0,0.5) to[out=left,in=up] (-0.28,0.3);
  \end{tikzpicture}~&\overset{\text{(R3)}}=~ -~\begin{tikzpicture}[baseline = -0.5mm]
	\draw[-,thick,darkblue] (-0.28,-.3) to[out=90,in=-90] (0.28,.3);
	\draw[-,line width=4pt,white](0.28,-0.3) to[out=90,in=-90] (-0.28,.3);
\draw[-,thick,darkblue] (0.28,-0.3) to[out=90,in=-90] (-0.28,.3);
        \draw[-,thick,darkblue] (0.28,0.3) to[out=up, in=right] (0.1,0.5) to[out=left, in=right] (0,0.5);
        \draw[-,thick,darkblue] (0,0.5) to[out=left,in=up] (-0.28,0.3);

        \draw[-,thick,darkblue] (0.28,-0.3) to[out=down,in=left] (0.58,-0.9) to[out=right,in=down] (0.88,-0.4);
\draw[-,thick,darkblue] (0.88,-0.4) to[out=up,in=left] (1.18,-0.05) to[out=right,in=up] (1.48,-0.3);
\draw[-,thick] (1.48,-0.3) to (1.48,-0.9);
  \end{tikzpicture}
  \overset{(1)}=~ ~\begin{tikzpicture}[baseline = -0.5mm]
	\draw[-,thick,darkblue] (-0.28,-.3) to[out=90,in=-90] (0.28,.3);
	\draw[-,line width=4pt,white](0.28,-0.3) to[out=90,in=-90] (-0.28,.3);
\draw[-,thick,darkblue] (0.28,-0.3) to[out=90,in=-90] (-0.28,.3);
        \draw[-,thick,darkblue] (0.28,0.3) to[out=up, in=right] (0.1,0.5) to[out=left, in=right] (0,0.5);
        \draw[-,thick,darkblue] (0,0.5) to[out=left,in=up] (-0.28,0.3);
        \draw[-,thick,darkblue] (0.28,-0.3) to[out=down,in=left] (0.58,-0.9) to[out=right,in=down] (0.88,-0.4);
\draw[-,thick,darkblue] (0.88,-0.4) to[out=up,in=left] (1.18,0.7) to[out=right,in=up] (1.48,-0.3);
\draw[-,thick] (1.48,-0.3) to (1.48,-0.9);
  \end{tikzpicture}
  ~  \overset{\text{(R4)}}=~ ~\begin{tikzpicture}[baseline = -0.5mm]
	%\draw[-,thick,darkblue] (-0.28,-.3) to[out=90,in=-90] (0.28,.3);
        \draw[-,thick,darkblue] (0.28,0.3) to[out=up, in=right] (0.1,0.5) to[out=left, in=right] (0,0.5);
        \draw[-,thick,darkblue] (0,0.5) to[out=left,in=up] (-0.28,0.3);
        \draw[-,thick,darkblue] (-0.28,0.3) to (-0.28,-0.3);
        \draw[-,thick,darkblue] (0.28,0.3) to (0.28,-0.2);
        \draw[-,thick,darkblue] (0.28,-0.2) to (0.98,-0.8);
        \draw[-,line width=4pt,white](0.48,-0.7) to  (1.08,0);
        \draw[-,thick,darkblue] (-0.28,-0.3) to[out=down,in=left] (0.28,-0.9) to[out=right,in=down] (0.68,-0.4);
\draw[-,thick,darkblue] (0.68,-0.4) to[out=up,in=left] (1.08,0.7) to[out=right,in=up] (1.48,-0.3);
\draw[-,thick] (1.48,-0.3) to (1.48,-0.9);
%\draw[-,thick,darkblue] (0.28,-0.3) to[out=90,in=-90] (-0.28,.3);
	%\draw[-,line width=4pt,white] (-0.28,-0.3) to[out=90,in=-90] (0.28,.3);
  \end{tikzpicture} \\
  ~\overset{\text{(R2), (R5)}}=~&~\begin{tikzpicture}[baseline = -0.5mm]
	%\draw[-,thick,darkblue] (-0.28,-.3) to[out=90,in=-90] (0.28,.3);
        \draw[-,thick,darkblue] (0.28,0.3) to[out=up, in=right] (0.1,0.5) to[out=left, in=right] (0,0.5);
        \draw[-,thick,darkblue] (0,0.5) to[out=left,in=up] (-0.28,0.3);
        \draw[-,thick,darkblue] (-0.28,0.3) to (-0.28,-0.3);
        \draw[-,thick,darkblue] (0.28,0.3) to (0.28,-0.2);
        \draw[-,thick,darkblue] (-0.28,-0.3) to[out=down,in=left] (0.28,-0.9) to[out=right,in=down] (0.68,-0.4);
\draw[-,thick,darkblue] (0.68,-0.4) to[out=up,in=left] (1.08,0.7) to[out=right,in=up] (1.48,-0.3);
\draw[-,thick] (1.48,-0.3) to (1.48,-0.9);\draw[-,line width=4pt,white](0.28,-0.2) to (0.98,-0.8);
        \draw[-,thick,darkblue] (0.28,-0.2) to (0.98,-0.8);
%\draw[-,thick,darkblue] (0.28,-0.3) to[out=90,in=-90] (-0.28,.3);
	%\draw[-,line width=4pt,white] (-0.28,-0.3) to[out=90,in=-90] (0.28,.3);
  \end{tikzpicture}~\overset{(2)}=~
    \begin{tikzpicture}[baseline = -0.5mm]
    \draw[-,thick,darkblue] (-0.28,-0.4) to[out=down,in=left] (0,-0.9) to[out=right,in=down] (0.28,-0.4);
    \draw[-,thick,darkblue](-0.28,0.1) to( 0.28,-0.4);
    \draw[-,line width=4pt,white](-0.28,-0.4) to( 0.28,0.1);
    \draw[-,thick,darkblue](-0.28,-0.4) to( 0.28,0.1);
    \draw[-,thick,darkblue]   ( 0.28,0.1) to[out=up,in=left] (0.48,0.3) to[out=right,in=up] (0.68,0.1);
    \draw[-,thick,darkblue](0.68,0.1)to( 0.68,-0.9);
    \draw[-,thick,darkblue]   ( -0.28,0.1) to[out=up,in=left] (0.68,0.6) to[out=right,in=up] (1.28,0.1);
    \draw[-,thick,darkblue](1.28,0.1)to( 1.28,-0.9);
  \end{tikzpicture}~\overset{\text{(R4)}}=~ q ~\begin{tikzpicture}[baseline = -0.5mm]
    \draw[-,thick,darkblue] (-0.28,0.1) to[out=down,in=left] (0,-0.9) to[out=right,in=down] (0.28,0.1);
    %\draw[-,thick,darkblue](-0.28,0.1) to( 0.28,-0.4);
%    \draw[-,line width=4pt,white](-0.28,-0.4) to( 0.28,0.1);
%    \draw[-,thick,darkblue](-0.28,-0.4) to( 0.28,0.1);
    \draw[-,thick,darkblue]   ( 0.28,0.1) to[out=up,in=left] (0.48,0.3) to[out=right,in=up] (0.68,0.1);
    \draw[-,thick,darkblue](0.68,0.1)to( 0.68,-0.9);
    \draw[-,thick,darkblue]   ( -0.28,0.1) to[out=up,in=left] (0.68,0.6) to[out=right,in=up] (1.28,0.1);
    \draw[-,thick,darkblue](1.28,0.1)to( 1.28,-0.9);
  \end{tikzpicture}~\overset{\text{(R3)}}=~-q\begin{tikzpicture}[baseline = 3pt, scale=0.5, color=\clr]
        \draw[-,thick] (1,0) to[out=up, in=right] (0.53,0.5) to[out=left, in=right] (0.47,0.5);
        \draw[-,thick] (0.49,0.5) to[out=left,in=up] (0,0);
\end{tikzpicture}.
    \end{aligned}
$$
Finally, one can check (4)-(5) via  Definition~\ref{pqbc}(R2)--(R5) and (2)-(3), easily.
\end{proof}

We are going to construct a spanning set of $\Hom_{\mathcal B}(m, s)$ via
$(m,s)$-{\em  tangle  diagrams}. An $(m,s)$-{ tangle  diagram} is a string diagram obtained by tensor products and compositions of
 the generating morphisms $\begin{tikzpicture}[baseline = 2.5mm]
	\draw[-,thick,darkblue] (0.28,0) to[out=90,in=-90] (-0.28,.6);
	\draw[-,line width=4pt,white] (-0.28,0) to[out=90,in=-90] (0.28,.6);
	\draw[-,thick,darkblue] (-0.28,0) to[out=90,in=-90] (0.28,.6);
\end{tikzpicture}, \begin{tikzpicture}[baseline = 2.5mm]
	\draw[-,thick,darkblue] (-0.28,-.0) to[out=90,in=-90] (0.28,0.6);
	\draw[-,line width=4pt,white] (0.28,-.0) to[out=90,in=-90] (-0.28,0.6);
	\draw[-,thick,darkblue] (0.28,-.0) to[out=90,in=-90] (-0.28,0.6);
\end{tikzpicture}$, $\lcup,\lcap$  and the identity  morphism \begin{tikzpicture}[baseline = 10pt, scale=0.5, color=\clr]
                \draw[-,thick] (0,0.5)to[out=up,in=down](0,1.5);\end{tikzpicture}
     such that there are $m$ (resp., $s$) endpoints  on the bottom (resp., top) row of the  resulting diagram.
 For example, the diagram in \eqref{exampleta} is a $(5,3)$-tangle diagram:
  \begin{equation}\label{exampleta}
     \begin{tikzpicture}[baseline = 25pt, scale=0.35, color=\clr]
         \draw[-,thick] (0,0) to[out=up,in=left] (2,1.5) to[out=right,in=up] (4,0);
     %  \draw[-,line width=4pt,white] (2,0) to[out=up,in=down] (0,5);
       \draw[-,thick] (1.5,0) to[out=up,in=down] (1.5,1.9);
       \draw[-,thick] (1.5,3.2) to[out=up,in=down] (1.5,5);

    %   \begin{tikzpicture}[baseline = -0.5mm]

\draw[-,thick,darkblue] (3,2.55) to [out=90,in=0](2.4, 3.1);
	%\draw[-,thick,darkblue] (0,-0.3) to (0,-0.6);
	\draw[-,thick,darkblue] (2.4,2) to [out=0,in=-90](3,2.55);
	\draw[-,thick,darkblue] (1.5,3.2) to [out=-90,in=180] (2.4,2);
	\draw[-,line width=4pt,white](2.4, 3.1) to [out=180,in=90](1.5,1.9);
	\draw[-,thick,darkblue] (2.4, 3.1) to [out=180,in=90](1.5,1.9);
	%\draw[-,thick,darkblue] (0.3,.2) to [out=180,in=90](0,-0.3);
%\end{tikzpicture}

        %      \draw (6,-0.34) node{\footnotesize{$4$}};
              % \draw[-,thick] (6,0) to[out=up,in=down] (6,5);%(shu xian)
                \draw[-,thick] (6.6,4.5) to (6.5,4.5) to[out=left,in=up] (4.5,2.5)
                        to[out=down,in=left] (6.5,0.5)
                        to[out=right,in=down] (8.5,2.5)
                        to[out=up,in=right] (6.5,4.5);

 %   \draw[-,thick] (6,2.3) to[out=up,in=down] (6,5);%(shu xian)
  \draw[-,line width=4pt,white] (8,0) to[out=up,in=down] (6,5);
    \draw[-,thick] (8,0) to[out=up,in=down] (6,5);
       \draw[-,line width=4pt,white] (8,2.5) to[out=up,in=down] (0,5);
     \draw[-,thick] (8,2.5) to[out=up,in=down] (0,5);

     % \draw[-,thick] (6,0) to[out=up,in=down] (6,2.4);%(shu xian)
       \draw[-,line width=4pt,white]  (2,0) to[out=up,in=down] (8,2.5);
        \draw[-,thick] (2,0) to [out=up,in=down] (8,2.5);
        \draw[-,line width=4pt,white] (0,0) to[out=up,in=left] (2,1.5) to[out=right,in=up] (4,0);
         \draw[-,thick] (0,0) to[out=up,in=left] (2,1.5) to[out=right,in=up] (4,0);%(cap)
 \end{tikzpicture}~~.
\end{equation}
  A  strand with no endpoints in a tangle diagram  is called a loop,  namely a loop   is obtained by projecting a connected framed link (i.e., a framed knot~\cite[Page~ 15]{Tu1}) in $\mathbb R^3$. For example there is a loop in the tangle diagram in
  \eqref{exampleta}.

 Let  $\mathbb{B}_{m,s}$  be  the set of  all $(m,s)$-tangle  diagrams.
     Obviously,, $\mathbb{B}_{m,s}$  spans  $\Hom_{\mathcal B}(m, s)$ when $m+s$ is even. In the remaining cases,  $\mathbb{B}_{m,s}=\emptyset$ and
   $\Hom_{\mathcal B}(m, s)=0$.

   We always assume that $m+s$ is even later on. In order to obtain a basis of $\Hom_{\mathcal B}(m, s)$, we need to consider special $(m, s)$-tangle diagrams as follows.
We label endpoints at   bottom (resp., top) row  of  an  $(m, s)$-tangle diagram $d$ as  $1, 2, \ldots, m$
(resp., $\bar 1, \bar 2,   \ldots,  \bar s$) from  left to right such that
\begin{equation}\label{ord12} 1<2<\cdots<m<\bar s<\cdots<\bar 2<\bar1.\end{equation}
For example, $d_1, d_2$ and $d_3$ are three
  $(4,2)$-tangles, where
   \begin{equation}\label{diagramc}d_1=\begin{tikzpicture}[baseline = 25pt, scale=0.35, color=\clr]
        \draw[-,thick] (2,0) to[out=up,in=down] (0,5);
       \draw[-,line width=4pt,white] (0,0) to[out=up,in=left] (2,1.5) to[out=right,in=up] (4,0);
       \draw[-,thick] (0,0) to[out=up,in=left] (2,1.5) to[out=right,in=up] (4,0);%(cap)
    \draw (0,-0.34) node{\footnotesize{$1$}};
     \draw (6,5.7) node{\footnotesize{$\bar 2$}};
      \draw (0,5.7) node{\footnotesize{$\bar 1$}};
     \draw (4,-0.34) node{\footnotesize{$3$}};
      \draw (2,-0.34) node{\footnotesize{$2$}};
              \draw (6,-0.34) node{\footnotesize{$4$}};
               \draw[-,thick] (6,0) to[out=up,in=down] (6,5);%(shu xian)
           \end{tikzpicture},\qquad d_2=\begin{tikzpicture}[baseline = 25pt, scale=0.35, color=\clr]
       %\draw[-,thick] (2,0) to[out=up,in=down] (0,5);
     %  \draw[-,line width=4pt,white] (0,0) to[out=up,in=left] (2,1.5) to[out=right,in=up] (4,0);
       \draw[-,thick] (6,2.3) to[out=up,in=down] (6,5);%(shu xian)
       \draw[-,line width=4pt,white] (8,2.5) to[out=up,in=down] (0,5);
     \draw[-,thick] (8,2.5) to[out=up,in=down] (0,5);
      \draw[-,thick] (6,0) to[out=up,in=down] (6,2.4);%(shu xian)
       \draw[-,line width=4pt,white]  (2,0) to[out=up,in=down] (8,2.5);
        \draw[-,thick] (2,0) to [out=up,in=down] (8,2.5);
        \draw[-,line width=4pt,white] (0,0) to[out=up,in=left] (2,1.5) to[out=right,in=up] (4,0);
         \draw[-,thick] (0,0) to[out=up,in=left] (2,1.5) to[out=right,in=up] (4,0);%(cap)
    % \draw[-,thick] (8,2.5) to[out=up,in=down] (0,5);

    \draw (0,-0.34) node{\footnotesize{$1$}};
     \draw (6,5.7) node{\footnotesize{$\bar 2$}};
      \draw (0,5.7) node{\footnotesize{$\bar 1$}};
     \draw (4,-0.34) node{\footnotesize{$3$}};
      \draw (2,-0.34) node{\footnotesize{$2$}};
              \draw (6,-0.34) node{\footnotesize{$4$}};

           \end{tikzpicture},\qquad  d_3=\begin{tikzpicture}[baseline = 25pt, scale=0.35, color=\clr]
     \draw[-,thick] (2,0) to[out=up,in=down] (0,5);
        \draw[-,line width=4pt,white](0,0) to[out=up,in=left] (2,1.5) to[out=right,in=up] (4,0);
         \draw[-,thick] (0,0) to[out=up,in=left] (2,1.5) to[out=right,in=up] (4,0);
         \draw[-,thick] (-1.1,4) to (-1,4) to[out=left,in=up] (-2,3) to[out=down,in=left] (-1,2)to[out=right,in=down] (0,3) to[out=up,in=right] (-1,4);%(yuan)

    \draw (0,-0.34) node{\footnotesize{$1$}};
     %\draw (6,5.7) node{\footnotesize{$\bar 2$}};
      \draw (0,5.7) node{\footnotesize{$\bar 1$}};
     \draw (4,-0.34) node{\footnotesize{$3$}};
     \draw (2,-0.34) node{\footnotesize{$2$}};
        %      \draw (6,-0.34) node{\footnotesize{$4$}};
              % \draw[-,thick] (6,0) to[out=up,in=down] (6,5);%(shu xian)
                \draw[-,thick] (6.6,4.5) to (6.5,4.5) to[out=left,in=up] (4.5,2.5)
                        to[out=down,in=left] (6.5,0.5)
                        to[out=right,in=down] (8.5,2.5)
                        to[out=up,in=right] (6.5,4.5);
                        \draw[-,line width=4pt,white] (5,0) to [out=up,in=down](6,1.7); \draw[-,line width=4pt,white] (6,1.7)to [out=up,in=down](5,5);
                        \draw[-,thick] (5,0) to [out=up,in=down](6,1.7);  \draw[-,thick] (6,1.7)to [out=up,in=down](5,5);
                         \draw (5,-0.34) node{\footnotesize{$4$}};
     \draw (5,5.5) node{\footnotesize{$\bar 2$}};
           \end{tikzpicture}~~.\end{equation}
Any $(m, s)$-tangle diagram
  $d$ decomposes $\{1, 2, \ldots, m, \bar s, \ldots, \bar 2, \bar 1 \}$ into  $\frac12(m+s)$  disjoint  pairs, say $(i_k, j_k)$'s such that $ i_k<j_k$, $ 1\le k\le \frac12(m+s)$ and  $i_o<i_l$,  $1\le o<l\le \frac12(m+s)$. We use  $(i_k, j_k)$ to denote the strand of $d$ such that the corresponding endpoints are labelled by $i_k$ and $j_k$.  Let
  \begin{equation}\label{conn1} \text{conn}(d)=\{(i_k, j_k)\mid  1\le k\le \frac12(m+s)\}\end{equation}
   and call it  the $(m, s)$-connector of $d$. So $\text{conn}(d_i)=\{(1,3), (2, \bar 1),(4, \bar 2)\}$, $1\le i\le 3$
   where $d_i$'s are given in \eqref{diagramc}. Let $$\text{conn}(m, s)=\{\text{conn}(d)\mid    d\in \mathbb B_{m,s}\}.$$

A strand connecting
a pair on different rows (resp., the same row) is called a vertical (resp., horizontal) arc.
Moreover, a horizontal arc connecting a pair on the top (resp., bottom) row is also
called a cup (resp., cap). So $\lcup $ is a cup and $\lcap$ is a cap.
%The following definition is motivated by \cite{Mor}.

 \begin{Defn}\cite[\S~2.3]{Mor}
Suppose  $d\in \mathbb B_{m,s}$  such that $\text{conn}(d)$ is given in \eqref{conn1}. Let $k=\frac 12 (m+s)$ and choose  $i_1,i_2, \ldots, i_{k}$ and one point on each loop as a sequence of base-points. Then $d$ is said to be
  \textsf{totally descending} if on traversing all the strands of $d$, starting from the base point of each component in order, each crossing is first met as an over-crossing.
  \end{Defn}
 All  tangle diagrams in \eqref{diagramc}
   are totally descending whereas  all tangle diagrams in \eqref{diagramc1}  are not:
    \begin{equation}\label{diagramc1}d_4=\begin{tikzpicture}[baseline = 25pt, scale=0.35, color=\clr]
  \draw[-,thick] (0,0) to[out=up,in=left] (2,1.5) to[out=right,in=up] (4,0);%(cap)
  \draw[-,line width=5pt,white] (2,0) to  (0,5);
     \draw[-,thick] (2,0) to  (0,5);
    \draw (0,-0.34) node{\footnotesize{$1$}};
     \draw (6,5.7) node{\footnotesize{$\bar 2$}};
      \draw (0,5.7) node{\footnotesize{$\bar 1$}};
     \draw (4,-0.34) node{\footnotesize{$3$}};
      \draw (2,-0.34) node{\footnotesize{$2$}};
              \draw (6,-0.34) node{\footnotesize{$4$}};
               \draw[-,thick] (6,0) to[out=up,in=down] (6,5);%(shu xian)
           \end{tikzpicture},\qquad
            d_5=\begin{tikzpicture}[baseline = 25pt, scale=0.35, color=\clr]
       %\draw[-,thick] (2,0) to[out=up,in=down] (0,5);
     %  \draw[-,line width=4pt,white] (0,0) to[out=up,in=left] (2,1.5) to[out=right,in=up] (4,0);
       \draw[-,thick] (6,2.3) to[out=up,in=down] (6,5);%(shu xian)
       \draw[-,line width=4pt,white] (8,2.5) to[out=up,in=down] (0,5);
     \draw[-,thick] (8,2.5) to[out=up,in=down] (0,5);
      \draw[-,thick] (6,0) to[out=up,in=down] (6,2.4);%(shu xian)
      \draw[-,thick] (0,0) to[out=up,in=left] (2,1.5) to[out=right,in=up] (4,0);%(cap)
       \draw[-,line width=4pt,white]  (2,0) to[out=up,in=down] (8,2.5);
        \draw[-,thick] (2,0) to [out=up,in=down] (8,2.5);
        %\draw[-,line width=4pt,white] (0,0) to[out=up,in=left] (2,1.5) to[out=right,in=up] (4,0);
    % \draw[-,thick] (8,2.5) to[out=up,in=down] (0,5);

    \draw (0,-0.34) node{\footnotesize{$1$}};
     \draw (6,5.7) node{\footnotesize{$\bar 2$}};
      \draw (0,5.7) node{\footnotesize{$\bar 1$}};
     \draw (4,-0.34) node{\footnotesize{$3$}};
      \draw (2,-0.34) node{\footnotesize{$2$}};
              \draw (6,-0.34) node{\footnotesize{$4$}};

           \end{tikzpicture},\qquad  d_6=\begin{tikzpicture}[baseline = 25pt, scale=0.35, color=\clr]
         \draw[-,thick] (0,0) to[out=up,in=left] (2,1.5) to[out=right,in=up] (4,0);\draw[-,thick] (-1.1,4) to (-1,4) to[out=left,in=up] (-2,3) to[out=down,in=left] (-1,2)to[out=right,in=down] (0,3) to[out=up,in=right] (-1,4);%(yuan)
       \draw[-,line width=4pt,white] (2,0) to[out=up,in=down] (0,5);
       \draw[-,thick] (2,0) to[out=up,in=down] (0,5);
    \draw (0,-0.34) node{\footnotesize{$1$}};
     %\draw (6,5.7) node{\footnotesize{$\bar 2$}};
      \draw (0,5.7) node{\footnotesize{$\bar 1$}};
     \draw (4,-0.34) node{\footnotesize{$3$}};
     \draw (2,-0.34) node{\footnotesize{$2$}};
        %      \draw (6,-0.34) node{\footnotesize{$4$}};
              % \draw[-,thick] (6,0) to[out=up,in=down] (6,5);%(shu xian)
                \draw[-,thick] (6.6,4.5) to (6.5,4.5) to[out=left,in=up] (4.5,2.5)
                        to[out=down,in=left] (6.5,0.5)
                        to[out=right,in=down] (8.5,2.5)
                        to[out=up,in=right] (6.5,4.5);
                        \draw[-,line width=4pt,white] (5,0) to [out=up,in=down](6,1.7); \draw[-,line width=4pt,white] (6,1.7)to [out=up,in=down](5,5);
                        \draw[-,thick] (5,0) to [out=up,in=down](6,1.7);  \draw[-,thick] (6,1.7)to [out=up,in=down](5,5);
                         \draw (5,-0.34) node{\footnotesize{$4$}};
     \draw (5,5.5) node{\footnotesize{$\bar 2$}};
           \end{tikzpicture}~~.\end{equation}
\begin{Lemma} \label{spann}The morphism set $\Hom_{\mathcal B}( m, s)$ is spanned by all totally descending $(m,s)$-tangle diagrams.
\end{Lemma}
\begin{proof} One can check the result by arguments similar to those in \cite[Theorem~2.6]{Mor}. More explicitly, these arguments are about induction firstly on the number of crossings and then on the number of non-descending crossings. The only difference is that we need to use Definition~\ref{pqbc}(R2) to replace a non-descending crossing in a tangle diagram  and obtain tangle diagrams with fewer non-descending crossings or fewer crossings.
\end{proof}

\begin{Defn}\label{reduced} A tangle diagram is called  reduced if
two strands cross each other at most once and neither a strand crosses itself nor there is a loop. A reduced  totally descending tangle diagram is both totally descending and reduced. \end{Defn}

\begin{Cor}\label{sjiws} The morphism space $\Hom_{\mathcal B}(\ob m,\ob s)$ is spanned by reduced  totally descending $(m,s)$-tangle diagrams.
\end{Cor}
\begin{proof}  First, we assume $T$ is a totally descending tangle diagram  containing  at least one loop. Thanks to Definition~\ref{pqbc}(R1), loops of $T$ are unknotted and lie  away from  horizontal and vertical arcs. By  Lemma~\ref{basicrel}(4)(5) and  Definition~\ref{pqbc}(R5), $T$ has to be  zero. Next, we consider $T$ without any loop.
By  Lemma~\ref{basicrel}(4)(5) and  Definition~\ref{pqbc}(R1), two strands cross each other at most once and no strand crosses itself. Consequently, by Lemma~\ref{spann}, any $(m, s)$-tangle diagram can be expressed as a linear combination of reduced totally descending tangle diagrams.
\end{proof}
%Using Definition~\ref{pqbc}(R1),(R4), Lemma~\ref{basicrel} and Lemma~\ref{spann}, we see that,
%up to a linear combination of reduced totally tangle diagrams with fewer crossings,  neither two strands of  $T$ crosses each other twice nor a strand crosses itself.
%Moreover,  Finally,  applying  Definition~\ref{pqbc}(R5), Lemma~\ref{spann} together with arguments about induction on the number of crossings, we see that
%$T$ can be expressed as  a linear combination of reduced totally descending  tangle diagrams.???

Let $\mathbb{RB}_{m,s}$ be the set of all reduced  totally descending diagrams in $\mathbb B_{m,s}$. We remark that  there are many reduced totally descending diagrams with the same connector $c\in \text{conn}(m,s)$.
\begin{Lemma}\label{jsjsjs} Suppose $d, d_1\in \mathbb{RB}_{m,s} $ such that  $ \text{conn}(d)=\text{conn}(d_1)$. Then   $d=d_1$ up to a sign.
\end{Lemma}
\begin{proof} Since we are assuming that  $d, d_1\in \mathbb{RB}_{m,s} $ and $ \text{conn}(d)=\text{conn}(d_1)$, $d$ can be obtained from $d_1$
by a sequence of following  movements :
\begin{multicols}{2}
\item[(1)] changing of the order of crossings,
\item[(2)] \begin{tikzpicture}[baseline = 6pt, scale=0.5, color=\clr]
         \draw[-,thick,darkblue] (1.25,0) to[out=90,in=-90] (0.5,1);
         \draw[-,line width=4pt,white](1,1) to[out=down, in=right] (0.53,0) to[out=left, in=right] (0.47,0);
        \draw[-,thick] (1,1) to[out=down, in=right] (0.53,0) to[out=left, in=right] (0.47,0);
        \draw[-,thick] (0.49,0) to[out=left,in=down] (0,1);
\end{tikzpicture} $ \leftrightsquigarrow$
  \begin{tikzpicture}[baseline = 6pt, scale=0.5, color=\clr]
        \draw[-,thick,darkblue] (-0.25,0) to[out=90,in=-90] (0.55,1);
        \draw[-,thick] (1,1) to[out=down, in=right] (0.53,0) to[out=left, in=right] (0.47,0);
        \draw[-,line width=4pt,white](0.49,0) to[out=left,in=down] (0,1);
        \draw[-,thick] (0.49,0) to[out=left,in=down] (0,1);
  \end{tikzpicture},
  \item[(3)] \begin{tikzpicture}[baseline = 3pt, scale=0.5, color=\clr]
\draw[-,thick] (0.6,0) to (1.2,0.8);
\draw[-,line width=4pt,white](1.2,.3) to[out=left,in=down] (1,.6);
        \draw[-,thick] (1.2,0) to[out=up, in=right] (0.63,0.6) to[out=left, in=right] (0.57,0.6);
        \draw[-,thick] (0.59,0.6) to[out=left,in=up] (0,0);
  \end{tikzpicture} $\leftrightsquigarrow$
     \begin{tikzpicture}[baseline = 3pt, scale=0.5, color=\clr]
\draw[-,thick] (0.5,0) to (0,0.8);
\draw[-,line width=4pt,white](.3,.4) to[out=left,in=down] (.1,.3);
        \draw[-,thick] (1.2,0) to[out=up, in=right] (0.63,0.6) to[out=left, in=right] (0.57,0.6);
        \draw[-,thick] (0.59,0.6) to[out=left,in=up] (0,0);
  \end{tikzpicture}~,
  \item[(4)] \begin{tikzpicture}[baseline = 3pt, scale=0.5, color=\clr]
        \draw[-,thick] (1,0) to[out=up, in=right] (0.53,0.9) to[out=left, in=right] (0.47,0.9);
        \draw[-,thick] (0.49,0.9) to[out=left,in=up] (0,0);
        \draw[-,thick] (2.5,0) to[out=up, in=right] (2.03,0.5) to[out=left, in=right] (1.97,0.5);
        \draw[-,thick] (1.99,0.5) to[out=left,in=up] (1.5,0);
\end{tikzpicture}~$\leftrightsquigarrow$~\begin{tikzpicture}[baseline = 3pt, scale=0.5, color=\clr]
        \draw[-,thick] (1,0) to[out=up, in=right] (0.53,0.5) to[out=left, in=right] (0.47,0.5);
        \draw[-,thick] (0.49,0.5) to[out=left,in=up] (0,0);
        \draw[-,thick] (2.5,0) to[out=up, in=right] (2.03,0.9) to[out=left, in=right] (1.97,0.9);
        \draw[-,thick] (1.99,0.9) to[out=left,in=up] (1.5,0);
\end{tikzpicture}~,
\item[(5)]  \begin{tikzpicture}[baseline = 3pt, scale=0.5, color=\clr]
        \draw[-,thick] (0,0.5)to  [out=down, in=left] (0.5,0) to[out=right, in=down](1,0.5);
        \draw[-,thick] (1.4,0.8)to  [out=down, in=left] (1.9,0.2) to[out=right, in=down](2.4,0.8);
        \end{tikzpicture}~$\leftrightsquigarrow$~
        \begin{tikzpicture}[baseline = 3pt, scale=0.5, color=\clr] \draw[-,thick] (0,0.8)to  [out=down, in=left] (0.5,0.2) to[out=right, in=down](1,0.8);
        \draw[-,thick] (1.4,0.5)to  [out=down, in=left] (1.9,0) to[out=right, in=down](2.4,0.5);
        \end{tikzpicture}~,
\item[(6)]\begin{tikzpicture}[baseline = -0.5mm]
%\draw[-,thick,darkblue] (0,0) to (0,0.3);
\draw[-,thick,darkblue] (0.5,0) to[out=down,in=left] (0.75,-0.25) to[out=right,in=down] (1,0);
\draw[-,thick,darkblue] (0,0) to[out=up,in=left] (0.25,0.25) to[out=right,in=up] (0.5,0);
\draw[-,thick,darkblue] (1,0) to (1,0.3);
\draw[-,thick,darkblue] (0,0) to (0,-0.3);
  \end{tikzpicture}~$\leftrightsquigarrow$~\begin{tikzpicture}[baseline = -0.5mm]
  \draw[-,thick,darkblue] (0,-0.25) to (0,0.25);
  \end{tikzpicture}~, and
 \begin{tikzpicture}[baseline = -0.5mm]
\draw[-,thick,darkblue] (0,0) to (0,0.3);
\draw[-,thick,darkblue] (0,0) to[out=down,in=left] (0.25,-0.25) to[out=right,in=down] (0.5,0);
\draw[-,thick,darkblue] (0.5,0) to[out=up,in=left] (0.75,0.25) to[out=right,in=up] (1,0);
\draw[-,thick,darkblue] (0,0) to (0,0.3);
\draw[-,thick,darkblue] (1,0) to (1,-0.3);
  \end{tikzpicture}
$\leftrightsquigarrow$  \begin{tikzpicture}[baseline = -0.5mm]
  \draw[-,thick,darkblue] (0,-0.25) to (0,0.25);
  \end{tikzpicture}~.\end{multicols}
\noindent Then the result follows by applying Definition~\ref{pqbc}(R1), (R3)-(R4) and Lemma~\ref{basicrel}.
\end{proof}

 Although there are many reduced totally descending diagrams with the same connector $c$, by Lemma~\ref{jsjsjs} it is reasonable   to  define
 \begin{equation}\label{dbms} \mathbb{DB}_{m,s}=\{D_c\mid c\in \text{conn}(m,s)\}.\end{equation} where $D_c$ is one of such  diagrams with connector $c$.  Obviously, $\mathbb{DB}_{m,s}=\emptyset$ if $m+s$ is odd.
 In the remaining case, the cardinality of $\mathbb{DB}_{m,s}$ is $(m+s-1)!!$.
 %Now, we state the first main result of this paper.

\begin{Theorem}\label{main1} If   $m+s$ is even then   $\Hom_{\mathcal B}(\ob m,\ob s)$ is free over $\Bbbk$ with basis  $\mathbb{DB}_{m,s}$.\end{Theorem}
The proof of Theorem~\ref{main1} will be given in section 4.
 In the remaining part of this section, we illustrate that  we can assume  $s=0$ when we  prove Theorem~\ref{main1}.
For any positive integer $m$, define two morphisms $\gamma_m$ and $\eta_m$ in $\mathcal B$ as follows:  \begin{equation}\label{varep}
 \gamma_{ m}=\begin{tikzpicture}[baseline = 25pt, scale=0.35, color=\clr]
        \draw[-,thick] (0,5) to[out=down,in=left] (2.5,2) to[out=right,in=down] (5,5);
        \draw[-,thick] (0.5,5) to[out=down,in=left] (2.5,2.5) to[out=right,in=down] (4.5,5);
        \draw(1.5,4.5)node{$\cdots$};\draw(3.5,4.5)node{$\cdots$};
        \draw[-,thick] (2,5) to[out=down,in=left] (2.5,4) to[out=right,in=down] (3,5);
        \draw (1,5.5)node {$ m$}; \draw (4,5.5)node {$ m$};
           \end{tikzpicture} \quad \text{ and }\quad
         \eta_{ m}= \begin{tikzpicture}[baseline = 5pt, scale=0.35, color=\clr]
           \draw[-,thick] (0,0) to[out=up,in=left] (2.5,4) to[out=right,in=up] (5,0);
           \draw[-,thick] (0.2,0) to[out=up,in=left] (2.3,3.5) to[out=right,in=up] (4.8,0);
           \draw[-,thick] (1.8,0) to[out=up,in=left] (2.5,1.5) to[out=right,in=up] (3.2,0);
           \draw(1,0.5)node{$\cdots$};\draw(4,0.5)node{$\cdots$};
          \draw (1,-0.5)node {$ m$}; \draw (4,-0.5)node {$ m$};
           \end{tikzpicture}.
           \end{equation}

\begin{Lemma}\label{tangle equa} For any positive integer $m$, we have
\begin{multicols}{2}
 \item[(1)] $(1_{ m}\otimes \eta_{ m})\circ (\gamma_{ m}\otimes 1_{ m})= (-1)^{\frac{m(m+1)}{2}} 1_m$,\item [(2)] $(\eta_{ m}\otimes 1_{ m})\circ ( 1_{ m}\otimes\gamma_{ m})=(-1)^{\frac{m(m-1)}{2}}1_{ m}$. \end{multicols} \end{Lemma}
\begin{proof} Both (1) and (2)  follow immediately from Definition~\ref{pqbc}(R3) and \eqref{wwcirrten}.\end{proof}
 %In the following, we always assume that $m+s$ is even.

 \begin{Lemma}\label{isoskks} For any $m, s\in \mathbb N$ with  $ 2\mid (m+s)$,   $\Hom_{\mathcal B}(m,  s)\cong \Hom_{\mathcal B}(m + s,0)$ as  $\Bbbk$-modules.
 \end{Lemma}
 \begin{proof}  By Lemma~\ref{tangle equa},  $ \bar\eta_s$ is the two-sided inverse of $\bar\gamma_s$ up to a sign, where
 $\bar\eta_s:$ $\Hom_{\mathcal B}(m,s)$ $\rightarrow \Hom_{\mathcal B}(  m+  s,0)$ is the $\Bbbk$-homomorphism such that $\bar\eta_s(d)=\eta_s\circ(1_s\otimes d)$ for any  $d\in\Hom_{\mathcal B}(  m,  s) $, and
 $\bar\gamma_s$
is the $\Bbbk$-homomorphism  such that $\bar \gamma_s(e)=(1_s\otimes e)\circ (\gamma_s\otimes 1_m)$ for any $e\in \Hom_{\mathcal B}(  m+s,0)$.
\end{proof}

\section{The quantized enveloping superalgebra $\mathrm{U}_q(\mathfrak{p}_n)$}

In this section, we will recall some basic results of the quantized enveloping superalgebras $\mathrm{U}_q(\mathfrak{p}_n)$ in \cite{AGG}.
%We are going to  use them to prove Theorem~\ref{main1} next section.

%We also briefly summarize a few key facts about the Periplectic $q$-Brauer algebra %$\mathcal{B}_{q,l}$ which appears as centralizers of the action of the quantized enveloping %superalgebra  $\mathrm{U}_q(\mathfrak{p}_n)$ on tensor product of the nature representation, %see \cite{AGG} for more details.
\subsection{The quantized enveloping superalgebras $\mathrm{U}_q(\mathfrak{p}_n)$}
For a positive integer $n$, we set $I_{n|n}:=\{-n,\ldots,-1,1,\ldots,n\}$, on which we define the parity $[i]$  of $i\in I_{n|n}$ such that
$[i]=\bar 0$ (resp., $\bar 1$) if $i>0$ (resp., $i<0$).
Let $V$ be the $\mathbb C$-superspace with basis $\{v_i\mid i\in I_{n|n}\}$ such that the parity of $v_i$ is $[i]$. Then the even subspace of $V$  is  $V_{\bar{0}}=\mathbb C$-$\{v_i\mid 1\le i\le n\}$ and the odd subspace of $V$ is $V_{\bar{1}}=\mathbb C$-$\{v_i\mid -n\le i\le -1\}$. The endomorphism algebra
$\mathrm{End}(V)$ is an associative superalgebra with standard basis $E_{i,j}$ whose   parity is  $[i]+[j]$ for all  $i,j\in I_{n|n}$. Under the standard supercommutator, $\mathrm{End}(V)$ is {\em the general linear Lie superalgebra} $\mathfrak{gl}_{n|n}$. The super-transpose $(\ . \ )^{\mathrm{st}}$ on $\mathfrak{gl}_{n|n}$ is given by the formula $$(E_{i,j})^{\mathrm{st}}=(-1)^{[i]([j]+1)}E_{j,i}.$$ Let $\pi: \mathfrak{gl}_{n|n}\longrightarrow\mathfrak{gl}_{n|n}$ be the linear map such that $\pi(E_{i,j})=E_{-i,-j}$. Then the linear map $\ell: \mathfrak{gl}_{n|n}\longrightarrow\mathfrak{gl}_{n|n}$ sending $x$ to
 $-\pi(x^{\mathrm{st}})$ is an involution on $\mathfrak{gl}_{n|n}$. The {\em periplectic Lie superalgebra} is   $$\mathfrak{p}_n=\{x\in \mathfrak {gl}_{n|n}\mid \ell (x)=x\},$$    the subalgebra of  $\mathfrak{gl}_{n|n}$ fixed by the involution $\ell$. In terms of matrices,
\begin{equation}
\mathfrak{p}_n=\left\{\begin{pmatrix}A&B\\C&-A^t\end{pmatrix}\in\mathfrak{gl}_{n|n}\middle|B^{t}=B,\  C^{t}=-C\right\},
\end{equation}
where $t$ is  the usual  transpose.  The  periplectic bilinear form is
$\beta: V\otimes V\rightarrow \mathbb{C}$ such that
\begin{equation}\label{pform} \beta(v_i,v_j)=\delta_{i,-j}(-1)^{[i]} \text{ for all }i,j\in I_{n|n}.\end{equation}
It  is skew-symmetric, odd, and non-degenerate.
 Then the periplectic Lie superalgebra is spanned by all homogeneous $ x\in \mathfrak{gl}_{n|n}$ such that
 \begin{equation}\label{pform1} \beta(xu,v)+(-1)^{[x][u]}\beta(u,xv)=0, \text{for all  homogeneous  $u,v\in V$.}\end{equation}

 %could also be realized as the Lie subalgebra of $\mathfrak{gl}_{n|n}$ preserving the form %$\beta$, i.e., it is spanned by all homogeneous elements x which satisfy %$\beta(xu,v)+(-1)^{[x][u]}\beta(u,xv)=0$ for all $u,v\in V$ and $x\in\mathfrak{gl}(n|n)$.

Let $V_q=V\otimes\mathbb{C}(q)$ where
$\mathbb{C}(q)$ is the field of rational functions in an indeterminate $q$.
 %Inspired by the definition of quantum queer superalgebra $\mathrm{U}_q(\mathfrak{q}_n)$ in \cite{Olshan},
   Ahmed-Grantcharov-Guay~\cite{AGG}   defined
 $S \in\mathrm{End}(V_q)^{\otimes 2}$ such that
$$\begin{aligned}
S=&1+\sum\limits_{i=1}^n\big((q-1)E_{i,i}+(q^{-1}-1)E_{-i,-i}\big)\otimes (E_{i,i}+E_{-i,-i})\\
&+(q-q^{-1})\sum\limits_{1\leqslant|j|<|i|\leqslant n}(-1)^{[j]}\big(E_{i,j}-(-1)^{[i]([j]+1)}E_{-j,-i}\big)\otimes E_{j,i}\\
&+(q-q^{-1})\sum\limits_{i=-n}^{-1}E_{i,-i}\otimes E_{-i,i}.
\end{aligned}$$

%The following definition is due to Ahmed-Grantcharov-Guay.
 %introduce the notion of  quantum periplectic superalgebra $\Uq(\mathfrak{p}_n)$ with the Faddeev-Reshetikhin-Takhtajan (FRT) formalism as
\begin{Defn}\cite[Definition~3.6]{AGG}
The quantum periplectic superalgebra $\Uq(\mathfrak{p}_n)$ is the unital associative superalgebra over $\mathbb{C}(q)$ generated by homogeneous elements $t_{i,j}, t_{i,i}^{-1}$ with $1\leqslant |i|\leqslant |j|\leqslant n$ and $i,j\in I_{n|n}$, which satisfy the following defining relations:
\begin{eqnarray}
&t_{i,i}=t_{-i,-i},\ t_{-i,i}=0 \text{ if }i>0,\ t_{i,j}=0 \text{ if }|i|>|j|;&\\
&T^{12}T^{13}S^{23}=S^{23}T^{13}T^{12},&\label{eq:STT}
\end{eqnarray}
where $T=\sum\limits_{|i|\leqslant |j|} t_{i,j}\otimes E_{i,j}$ and the relation (\ref{eq:STT}) holds in $\Uqpn\otimes_{\mathbb C(q)}\left(\mathrm{End}(V_q)\right)^{\otimes 2}$. The parity of $t_{i,j}$ is $[i]+[j]$.
\end{Defn}

Thanks to \cite[Remark~3.7]{AGG},  $\Uq(\mathfrak{p}_n)$ is a Hopf superalgebra with the  comultiplication $\Delta$, counit $\varepsilon$ and antipode $\mathrm{S}$ given by
\begin{align}
\Delta(t_{i,j})=\sum\limits_{|i|\leqslant |k|\leqslant |j|}(-1)^{([i]+[k])([k]+[j])}t_{i,k}\otimes t_{k,j}, \quad \varepsilon(T)=1,\quad  \mathrm{S}(T)=T^{-1} .
\end{align}
%The counit and antipode on $\Uq(\mathfrak{p}_n)$ are given by $\varepsilon(T)=1$ and $\mathrm{S}(T)=T^{-1}$, respectively.
\subsection{The category $\mathrm{U}_q(\mathfrak{p}_n)$-mod}
Given two left $\mathrm{U}_q(\mathfrak{p}_n)$-supermodules $M, N$, a homogeneous homomorphism $f: M\rightarrow N$ is a $\mathbb C$-linear map such that $$f(xm) = (-1)^{[x][f]}x f (m)$$ for homogeneous elements $m\in M, x\in \Uqpn$. We allow morphisms to be odd (i.e. to change the parity of elements they are applied to). Let $\mathrm{U}_q(\mathfrak{p}_n)$-mod be the category of left $\mathrm{U}_q(\mathfrak{p}_n)$-supermodules such that $\mathrm{Hom}_{\Uqpn}(M,N)$ is the set of all  linear combinations of homogeneous homomorphisms. Then $\mathrm{U}_q(\mathfrak{p}_n)$-mod admits a monoidal structure induced by
the coproduct and counit of a Hopf superalgebra in the usual sense.
Suppose    $$ P=\sum\limits_{a,b=-n}^n(-1)^{[b]}E_{a,b}\otimes E_{b,a}.$$ It is not difficult to verify
\begin{equation}\label{psaction} \begin{aligned}
PS(v_a\otimes v_b)
=&(-1)^{[a][b]}v_b\otimes v_a+(q-1)\delta_{a>0}(\delta_{a,-b}+\delta_{a,b})v_b\otimes v_a\\
&+(q^{-1}-1)\delta_{a<0}(\delta_{a,-b}-\delta_{a,b})v_b\otimes v_a+(q-q^{-1})\delta_{a>0}\delta_{a,-b}v_{-b}\otimes v_{-a}\\
&+(q-q^{-1})\delta_{|a|<|b|}v_{a}\otimes v_{b}+(q-q^{-1})(-1)^{[a]}\delta_{a,-b}\sum\limits_{1\leqslant |j|<|a|\leqslant n}v_j\otimes v_{-j}.
\end{aligned}\end{equation}

 %where \begin{equation}\label{blt} \mathfrak{t}=PS: V_q^{\otimes2}\rightarrow V_q^{\otimes 2}\end{equation}

%We also have the following two odd $\Uqpn$-module homomorphisms.

\begin{Lemma}\cite[Lemma 5.3]{AGG}\label{jjudd}
Let $\mathbb{C}(q)$ be a purely even $\Uqpn$-module.\footnote{In \cite{AGG},  $\mathbb{C}(q)$ is assumed to be an odd  $\Uqpn$-module.}  We have two odd $\Uqpn$-module homomorphisms $\vartheta: V_q\otimes V_q\rightarrow \mathbb{C}(q)$ and $\epsilon: \mathbb{C}(q)\rightarrow V_q\otimes V_q$ given by $\vartheta(v_a\otimes v_b)=\delta_{a,-b}(-1)^{[a]}$ and $\epsilon(1)=\sum\limits_{a=-n}^nv_a\otimes v_{-a}$.
\end{Lemma}
Later on $\vartheta$ and  $\epsilon$ are called evaluation map and coevaluation map, respectively.
Composing $\vartheta$ and $\epsilon$ yields  an even $\Uqpn$-module homomorphism  $\mathfrak{c}=\vartheta\circ\epsilon: V_q^{\otimes 2}\rightarrow V_q^{\otimes 2}$.
In \cite{AGG}, Ahmed-Grantcharov-Guay also defined  another even $\Uqpn$-module homomorphism  $\mathfrak t=PS: V_q^{\otimes2}\rightarrow V_q^{\otimes 2} $.
For any $1\le i\le l-1$, let \begin{equation}\label{blt}  \mathfrak{t}_i=1^{\otimes {i-1}}\otimes \mathfrak t\otimes 1^{l-i-1}: V_q^{\otimes l}\rightarrow V_q^{\otimes l} \ \ \text{ and }\ \
 \mathfrak{c}_i=1^{\otimes {i-1}}\otimes \mathfrak c\otimes 1^{l-i-1}: V_q^{\otimes l}\rightarrow V_q^{\otimes l}.
\end{equation}

\subsection{ The periplectic $q$-Brauer algebra $\mathcal{B}_{q,l}$}
The following definition is given in \cite[Definition 5.1]{AGG} when the ground ring $\Bbbk$ is $\mathbb C(q)$.
\begin{Defn}\label{defn:qBrauer} Suppose that $\Bbbk$ is an integral domain of characteristic not $2$ and $l\in\mathbb N$.
The periplectic $q$-Brauer algebra $\mathcal{B}_{q,l}$ is the associative $\Bbbk$-algebra generated by elements $T_i$ and $E_i$ for $1\leqslant i\leqslant l-1$ satisfying the following relations:
\begin{itemize}\item[(1)]
$(T_i-q)(T_i+q^{-1})=0,\quad T_iT_j=T_jT_i,\quad T_kT_{k+1}T_k=T_{k+1}T_kT_{k+1}$,
\item[(2)] $E_i^2=0,\quad E_iE_j=E_jE_i,\quad E_{k+1}E_{k}E_{k+1}=-E_{k+1},\quad  E_{k}E_{k+1}E_{k}=-E_{k}$,
\item [(3)] $T_iE_j=E_jT_i,\quad E_iT_i=-q^{-1}E_i,\quad T_iE_i=qE_i$,
\item [(4)] $ T_kE_{k+1}E_k=-T_{k+1}E_k+(q-q^{-1})E_{k+1} E_k, \quad E_{k+1}E_{k}T_{k+1}=-E_{k+1}T_k+(q-q^{-1})E_{k+1} E_k$,
\end{itemize}
where  $1<|i-j|$ and $1\leqslant k\leqslant l-2$.
\end{Defn}
By convention, $\mathcal B_{q,0}=\mathcal B_{q,1}=\Bbbk$.
\begin{Lemma}\label{antiin} There is a  $\Bbbk$-linear anti-involution $\sigma:  \mathcal B_{q, l}\rightarrow \mathcal B_{q, l}$ such that  $\sigma(T_i)=-T^{-1}_i$ and $\sigma(E_i)=-E_i$, $1\le i\le l-1$.\end{Lemma}
\begin{proof}The result follows directly from  Definition~\ref{defn:qBrauer}. \end{proof}
Let   $\mathfrak S_{l}$ be the symmetric group on $l$ letters. As Coxeter group, $\mathfrak S_l$ is generated by $s_i, 1\le i\le l-1$ subject to the relations
$$s_i^2=1, \quad s_is_j=s_js_i \quad \text{and} \quad s_ks_{k+1}s_k=s_{k+1}s_ks_{k+1},$$
for any $1\le i\le l-1$, $ |i-j|>1 $ and $1\le k\le l-2$.
Let $H_l$ be the Hecke algebra associated to  $ \mathfrak S_l$.
Then $H_l$ is the unital associative $\Bbbk$-algebra generated by $T_i$, $1\leq i\leq l-1$, satisfying the following relations
\begin{equation} \label{kkk1} (T_i-q)(T_i+q^{-1})=0,\ \  T_iT_j=T_jT_i,\ \ T_kT_{k+1}T_k=T_{k+1}T_kT_{k+1}\end{equation}
where  $1<|i-j|$ and $1\le k\le l-2$. Thanks to Definition~\ref{defn:qBrauer}, there is a surjective homomorphism from $\mathcal B_{q,l}$ to $H_l$, sending $T_i$ and  $E_i$ to $T_i$ and 0, respectively. We will see that $H_l$ is also a subalgebra of $\mathcal B_{q,l}$ after we prove a basis theorem for $\mathcal B_{q,l}$ next section.
%The following relations in $\mathcal B_{q,l}$ are useful latter.
\begin{Lemma} \label{key1}
In $\mathcal B_{q,l}$, we have
\begin{itemize}
\item[(1)] $E_{i+1}=T_iT_{i+1}E_iT_{i+1}^{-1}T_i^{-1}$, for $1\leq i<l-1$,
\item[(2)] $T_i^{-1}E_{i+1}E_i=-T_{i+1}E_i$, $E_{i+1}E_iT_{i+1}^{-1}=-E_{i+1}T_i$, for $1\leq i<l-1$,
\item[(3)] $E_iE_{i+1}T_i=E_iT_{i+1}^{-1}$, $T_{i+1}E_iE_{i+1}=T_i^{-1}E_{i+1}$, for $1\leq i<l-1$,
\item[(4)] $E_{i+1}T_iE_{i+1}=q^{-1}E_{i+1}$, $E_iT_{i+1}E_i=-qE_i$, for $1\leq i<l-1$,
\item[(5)] $E_iE_{i+2}T_{i+1}T_i=-E_iE_{i+2} T_{i+1}^{-1}T_{i+2}^{-1}$, $T_{i+2}T_{i+1} E_{i}E_{i+2}=- T_{i}^{-1}T_{i+1}^{-1}E_iE_{i+2} $, for $1\leq i<l-2$,
\item[(6)]$E_iT_{i+1}T_{i+2}T_{i}T_{i+1}E_i= -E_iE_{i+2}$, for $1\leq i<l-2$.
\end{itemize}
\end{Lemma}
\begin{proof} Thanks to Definition~\ref{defn:qBrauer}, one can check (1)-(4) easily.   We verify (5)-(6) as an example. We have
$$\begin{aligned} E_iE_{i+2}T_{i+1}T_i&=-E_iE_{i+2}E_{i+1}E_{i+2}T_{i+1}T_i= -E_iE_{i+2}E_{i+1}T_{i+2}^{-1}T_i=-E_{i+2}E_iE_{i+1}T_iT_{i+2}^{-1}\\ & = -E_{i+2}E_iT_{i+1}^{-1}T_{i+2}^{-1}.\\
\end{aligned}$$ Applying the anti-involution $\sigma$  in Lemma~\ref{antiin} on the above equation yields the last equation in (5). By (2) and Definition~\ref{defn:qBrauer}(2), $E_iT_{i+1}T_{i+2}T_{i}T_{i+1}E_i=E_iE_{i+2}E_{i+1}E_i=-E_{i+2}E_i$, proving
(6).
\end{proof}

\begin{Theorem}\cite[Theorem 5.5]{AGG}\label{algebsbs} Suppose $\Bbbk=\mathbb C(q)$.   There is an algebra epimorphism  $$\phi: \mathcal{B}_{q,l}\longrightarrow \mathrm{End}_{\Uqpn}(V_q^{\otimes l})$$ such that   $\phi(T_i)= \mathfrak{t}_i$ and $\phi(E_i)= \mathfrak{c}_i$,   $1\leqslant i\leqslant l-1$.  When $n\geqslant l$, $\phi$    is injective and hence an isomorphism.
\end{Theorem}

\begin{Prop}\label{ksksmsk}Suppose  $\Bbbk=\mathbb C(q)$. There is a monoidal  superfunctor $\Phi: \mathcal{B}\longrightarrow \Uqpn$-mod such that  $\Phi(1)=V_q$,
$\Phi(\mathord{\begin{tikzpicture}[baseline = 2.5mm]
	\draw[-,thick,darkblue] (0.28,0) to[out=90,in=-90] (-0.28,.6);
	\draw[-,line width=4pt,white] (-0.28,0) to[out=90,in=-90] (0.28,.6);
	\draw[-,thick,darkblue] (-0.28,0) to[out=90,in=-90] (0.28,.6);
\end{tikzpicture}})= \mathfrak{t}$ ,$\Phi(\lcup)=\epsilon$ and $\Phi(\lcap)= \vartheta$.
\end{Prop}
\begin{proof} It is enough to verify that the images of generating morphisms of $\mathcal B$ satisfy the defining relations in Definition~\ref{pqbc}~(R1)-(R5). In fact,
Definition~\ref{pqbc}~(R1)-(R2) have already been verified in \cite[Theorem~5.5]{AGG}. One can check other relations by straightforward computation.
  \end{proof}

\section{A basis theorem of the periplectic  $q$-Brauer category }
%The aim of this section is to prove Theorem~\ref{main1}. As a corollary, we prove that the periplectic $q$-Brauer algebra  $\mathcal B_{q,l}$ is isomorphic to $\End_{\mathcal B}(l)$.

\subsection{Proof of Theorem~\ref{main1}} The following definition is motivated by \cite[Lemma~3.1]{GRS}.

\begin{Defn} \label{codeg} For any $i\in \mathbb Z$, define  $\mathcal A_i=\{x\in \mathbb C(q)\mid
ev(x)\geq i\}$, where $ev: \mathbb C(q)\rightarrow \mathbb Z\cup\{\infty\}$ such that $ev(0)=\infty$ and
$ev(x)=j$ if $0\neq x=(q-1)^j g/h$ and  $g, h\in \mathbb C[q]$ such that $(g, q-1)=1$, $(h, q-1)=1$. \end{Defn}

\begin{Defn}For any $r\in \mathbb N$ and $j=0,1$, let $V^{\otimes r}_j$  be the free $\mathcal A_0$-submodule of $V_q^{\otimes r}$
with basis $\{(q-q^{-1})^jv_{\mathbf i}\mid \mathbf i\in I_{n\mid n}^r\}$, where  $v_\mathbf i=v_{i_1}\otimes v_{i_2}\otimes\ldots\otimes v_{i_r}$. In particular, $V^{\otimes 0}_j=\mathcal A_j$ for $j=0,1$.
\end{Defn}

\begin{Lemma}\label{jisjss} Suppose $j\in \{0,1\}$ and $a, b\in I_{n\mid n}$.
\begin{multicols}{2}
%\begin{itemize}

\item[(1)]$\mathfrak{t}^{\pm}(v_a\otimes v_b )\equiv(-1)^{[a][b]}v_b\otimes v_a  \pmod {V^{\otimes 2}_1}$.\item[(2)]$\mathfrak{t}^{\pm} V^{\otimes 2}_j\subseteq  V^{\otimes 2}_j$, $\vartheta V^{\otimes 2}_j\subseteq \mathcal A_j$ and $\epsilon \mathcal A_j\subseteq V^{\otimes 2}_j$.
\end{multicols}
\end{Lemma}
\begin{proof} Obviously, (1) and the first inclusion in (2)  follow from \eqref{psaction},  Definition~\ref{pqbc}(R2) and Proposition~\ref{ksksmsk}.
The remaining inclusions in
(2)  follow form Lemma~\ref{jjudd}.
\end{proof}

 Throughout this subsection, we assume $n>r$, where $r$ is any fixed natural number.
 For  any  $d\in \mathbb{RB}_{2r,0}$ with $\text{conn}(d) =\{(i_1,j_1),\ldots,(i_r,j_r )\}$, define
\begin{equation}\label{hdd} h_d=(h_1, h_2,\ldots,h_{2r})\end{equation} such that $h_{i_k}=k$ and $h_{j_k}=-k$ for $k=1,2,\ldots, r$. Since $n>r$,  $h_d\in I_{n|n}^{2r}$. Recall  $\mathbb{DB}_{2r,0}$ in \eqref{dbms}.
\begin{Lemma}\label{acyijs}
For any $e_1, e_2\in \mathbb{DB}_{2r,0}$,
$e_1v_{h_{e_2}}\equiv \pm \delta_{\text{conn}(e_1), \text{conn}(e_2)} \pmod {\mathcal A_1}$.
\end{Lemma}
\begin{proof}
By Lemma~\ref{jsjsjs}, we can choose $e_1$ such that all caps of $e_1$ are not at the same height. Then $$ e_1= \begin{tikzpicture}[baseline = 3pt, scale=0.5, color=\clr]
\draw[-,thick] (-1,1.5) to (3,1.5)to(3,2.5)to (-1,2.5) to (-1,1.5);\draw (1,2)node {$ d_1$};
      \draw[-,thick] (0,0) to [out=up, in=left](1,1)to   [out=right, in=up] (2,0) ;\draw (2,-0.5)node {$ j_k$};\draw (0,-0.5)node {$ i_k$};
 \draw[-,thick] (-1,1.4) to (-1,0)  ; \draw (-1,-0.5)node {$ 1$};\draw[-,thick] (-0.6,1.4) to (-0.6,0)  ; \draw (-0.6,-0.5)node {$ 2$};
\draw (-0.2,0.5)node {$ \ldots$}; \draw[-,thick] (0.2,1.4) to (0.2,0)  ;\draw (1,0.5)node {$ \ldots$}; \draw[-,thick] (1.6,1.4) to (1.6,0)  ;
\draw[-,thick] (2.4,1.4) to (2.4,0)  ;\draw (2.5,0.5)node {$ \ldots$}; \draw[-,thick] (3,1.4) to (3,0)  ; \draw (3,-0.5)node {$ 2r$};
 \end{tikzpicture}$$
 for some   $d_1\in \mathbb{DB}_{2r-2,0}$ and some $k, 1\le k\le r$.  So $e_1=d_1\circ d_1'$ where
 $d_1'$ is determined via above diagram in an obvious way. Similarly, we have  $e_2=d_2\circ d_2'$ where
 $d_2\in \mathbb{DB}_{2r-2,0}$.
Each crossing in $d_1'$ may be over crossing or under crossing, which   depends  on $e_1$.
By Lemma~\ref{jisjss}, we have
\begin{equation}\label{ssj} e_1v_{h_{e_2}}\equiv\pm  \delta_{\text{conn}(d_2'), \text{conn}(d_1')} d_1 v_{h_{d_2}}\pmod{ V^{\otimes {2r-2}}_1}
\end{equation}
where $v_{h_{d_2}}\in V^{\otimes {2r-2}}$ is obtained from $v_{h_{e_2}}$ by deleting its $i_k$-th and  $j_k$-th tensor factors. Now the results follow from  induction assumption for $r-1$.
\end{proof}

\textbf{Proof of Theorem~\ref{main1}:} Thanks to Lemma~\ref{isoskks}, we can assume $s=0$ and $m=2r$ for some $r\in \mathbb N$ when we prove Theorem~\ref{main1}. By Corollary~\ref{sjiws} and Lemma~\ref{jsjsjs}, it remains to show that $\mathbb{DB}_{2r, 0}$
is linear independent over $\Bbbk$. First, we assume that $\Bbbk=\mathbb C(q)$.
Suppose
\begin{equation}\label{lish}
\sum_{d\in B}c_d d=0
\end{equation}
for some   $B\subset \mathbb{DB}_{2r, 0}$ with $0\neq c_d\in\mathbb C(q) $.
Without loss of any generality, we can assume  $c_d\in \mathcal A_0$ for all such $d$'s. Otherwise, we can multiply some powers of $q-1$ on both sides of \eqref{lish} such that all coefficients are in $\mathcal A_0$.

Let $a$ be the minimal positive integer such that all $c_d\in \mathcal A_a  \setminus \mathcal A_{a-1}$.
Then there exists a $d_0\in B$ such that  $(q-1)^{-a}c_{d_0}\in \mathcal A_0\setminus \mathcal A_1$.
Multiplying  both sides of \eqref{lish} by $(q-1)^{-a}$, we
 can assume that there exists some $d_0\in B$ such that $c_{d_0}\in \mathcal A_0\setminus \mathcal A_1$. By Lemma~\ref{acyijs},
$0=\sum_{d\in B}c_d dv_{h_{d_0}}\equiv \pm c_{d_0} ~~(\text{mod }\mathcal A_1)$,
forcing $c_{d_0}=0$, a contradiction. So  $\mathbb{DB}_{2r, 0}$
is linear independent over $\mathbb C(q)$ and hence over $\mathbb Z[q, q^{-1}]$.
This proves   Theorem~\ref{main1} over $\mathbb Z[q, q^{-1}]$.
By  standard arguments on base change, we have Theorem~\ref{main1} over an arbitrary integral domain with characteristic not $2$.

 %From here to the end of this section, we assume that $\Bbbk$ is an algebraically closed field with characteristic not $2$.
 %We want to construct more bases of any morphism space in
% $\mathcal B$ over $\Bbbk$. We start by giving the following definition.

 \begin{Defn} \label{rll}
 Given an $(m, s)$-connector $c$,  a \textsf{reduced lift} with respect to $c$ is either $D_c\in \mathbb{DB}_{m,s}$ or
  an $(m, s)$-tangle diagram obtained from  $D_c$ by exchanging some under or  over crossings.
  Let \begin{equation}\label{rl} \mathbb{RL}_{m,s}=\{r_c\mid r_c \text{ is any fixed reduced lift of $c$, and $c\in \text{conn}(m,s)$}\}.\end{equation}\end{Defn}
%By Definition~\ref{rll}, a reduced lift of a connector $c$  is  not  totally descending if it is not $D_c$.

\begin{Cor}\label{newbasis}
The morphism space $\Hom_{\mathcal B}(m,s)$ has  basis given by $\mathbb{RL}_{m, s}$.
\end{Cor}
\begin{proof} Thanks to Definition~\ref{pqbc}(R2), $r_c=D_c$ up to a linear combination of some $(m, s)$-diagrams in $\mathbb{DB}_{m,s}$  with fewer crossings.
This shows that the transition matrix between $\mathbb{RL}_{m, s}$ and $\mathbb{DB}_{m, s}$ is uni-uppertriangular.
Now, the result follows from  Theorem~\ref{main1}.
\end{proof}

\subsection{A basis of  periplectic $q$-Brauer algebra $\mathcal B_{q,l}$}
From here to the end of this section, $\Bbbk$ is always an integral domain with characteristic not $2$.
 %By abusing of notations, generators in $\mathcal B_{q, l}$ and $\mathcal B_{q, l-2f}$ are denoted by $E_i$ and $T_j$'s in an obvious way.
Later on, write $s_{i,j}=s_is_{i+1, j}$ if $i<j$, and $s_{i, i}=1$, and $s_{i,j}=s_{i, j+1}s_j$ if $i>j$.
 Define $ D_{0,l}=\tilde D_{0, l}=\{1\}$ and
\begin{equation}\label{dlf}\begin{aligned}
D_{f,l}& =\{s_{2f,i_f}s_{2f-1,j_f}\cdots s_{2,i_1}s_{1,j_1}\mid  i_1<\ldots<i_f, k\leq j_k<i_k\leq l, \text{ for }1\leq k\leq f \}, \\
\tilde D_{f,l}&= \{s_{2f,i_f}s_{2f-1,j_f}\cdots s_{2,i_1}s_{1,j_1}\mid   k\leq j_k<i_k\leq l, \text{ for }1\leq k\leq f \},\\
\end{aligned}
\end{equation}
if
$0< f\leq \lfloor l/2\rfloor$. Let $\mathcal B_l(f)$ be the subalgebra of $\mathcal B_{q,l}$ generated by $E_{2f+1}$, $T_j$, $2f+1\leq j\leq l-1$. By Lemma~\ref{key1}, we see that there is an algebra epimorphism
\begin{equation}\label{algepi} \psi: \mathcal B_{q, l-2f}\twoheadrightarrow \mathcal B_{l}(f)\end{equation}
 sending  the generators $E_i, T_i$ of $\mathcal B_{q, l-2f}$ to the generators
$E_{2f+i}$ and $T_{2f+i}$ of $\mathcal B_{q, l}$.
To simplify the notation, we write
$T_{i, j}=T_{i, j+1} T_j$ if $ i>j$, and $T_{i, i}=1$ and $T_{i, j}=T_i T_{i+1, j}$ if $i<j$. Let $E^0=1$ and for any $1\le f\le \lfloor l/2\rfloor$, define  \begin{equation}\label{ef} E^f=E_1E_3\cdots E_{2f-1}.\end{equation}

\begin{Lemma}\label{fque1} Let $N^f$ be the left $\mathcal B_l(f)$-module generated by $ V^f=\{E^fT_d\mid d\in D_{f,l} \}$, and $\tilde N^f$ be the left $\mathcal B_l(f)$-module generated by $\tilde V^f=\{E^fT_d\mid d\in \tilde D_{f,l} \}$.
 Then \begin{multicols}{3}
  \item[(1)] $N^1$ is a right $\mathcal B_{q,l}$-module.
\item[(2)]  $\tilde N^f$ is a right $\mathcal B_{q,l}$-module.
\item[(3)]   $\tilde N^f=N^f$.
\end{multicols}
\end{Lemma}
\begin{proof} We prove (1) by showing that  $N^1$ is stable under the right action of $E_1$ and $T_j$'s. Write $d=s_{2,i_1}s_{1,j_1}$ with $1\leq j_1<i_1$. If $j_1=1$, then $i_1\ge 2$. Thanks to Lemma~\ref{key1}(4) and Definition~\ref{defn:qBrauer}(2), $E_1T_dE_1=-qE_1T_{3,i_1}\in N^1$ if $i_1\geq 3$ and $E_1T_dE_1=0$ if $i_1=2$. Suppose $j_1>1$. Then $i_1\ge 3$.
By Lemma~\ref{key1}(2),(4) and Definition~\ref{defn:qBrauer}(3),
$$E_1 T_{2,i_1}T_{1,j_1}E_1=\begin{cases}
                               -q^2  T_{3,i_1}E_1  & \hbox{if $j_1=2$,} \\
                                -E_3 T_{4,i_1}T_{3,j_1} E_1 & \hbox{if $j_1>2$.}
                              \end{cases}
$$
In any case,  $N^1E_1\subset N^1$.
One can check that $N^1T_i\subset N^1$ by using  Definition~\ref{defn:qBrauer}. In this case, it is enough to prove $ E_1T_{2,i}T_{1,j}\in N^1$ when $i\le j$. In fact, since $s_{2, j}s_{1, i-1}\in D_{1, l}$,
$$ E_1T_{2,i}T_{1,j}=E_1T_{1,j}T_{1,i-1}=-q^{-1}E_1T_{2,j}T_{1,i-1}\in N^1.$$ This completes the proof of (1). Using (1)  repeatedly yields (2).
%To prove (3), we need the following claim:

 Suppose $i_1\ge i_2$ and $d=s_{4,i_2}s_{3,j_2}s_{2,i_1}s_{1,j_1}\in\tilde D_{2,l}$. We claim
$E^2T_d\in \tilde N^2_{i_1}$, where $\tilde N^2_{i_1}$ is the left $\mathcal B_l(2)$-submodule of $\tilde N^2$ generated by $V_{i_1}^2=\{E^2T_{d'}\mid d'=s_{4,i_2'}s_{3,j_2'}s_{2,i_1'}s_{1,j_1'}\in \tilde D_{2,l}, i_1'<i_1 \}$.

 In fact, since  $i_1\geq i_2>j_2$, by Lemma~\ref{key1}(5) we have
\begin{equation}\label{key2}  E^2T_d=E^2T_{2,i_1}T_{3,i_2-1}T_{2,j_2-1}T_{1,j_1}
=-E^2T_2^{-1}T_1^{-1}T_{4,i_1}T_{3,i_2-1}T_{2,j_2-1}T_{1,j_1}.\end{equation}
In order to prove our claim,  it is enough  to verify that the RHS of \eqref{key2} is in $\tilde N^2_{i_1}$.

If  $j_1\geq i_2>j_2$, then
the RHS of \eqref{key2} is $-E^2 T_{4,i_1}T_{3,j_1}T_{2,i_2-2}T_{1,j_2-2}\in \tilde N^2_{i_1}$. The last inclusion follows from  inequalities
   $i_1>i_2-2$, $i_2-2>j_2-2$ and $i_1>j_1$.

Suppose  $ j_1<i_2$.  There are two cases we need to discuss
%We want to verify that the  RHS of \eqref{key2} is in $\tilde N^2_{i_1}$ in two cases:
(1) $j_1\geq j_2-1$, (2)  $j_1<j_2-1$.
In  case (1),   the RHS of \eqref{key2} is $-E^2 T_{4,i_1}T_2^{-1}T_{3,i_2-1}T_{2,j_1}T_{1,j_2-2}$. So,
\begin{equation}\label{jsjhd}\text{ RHS of \eqref{key2}}
= (q-q^{-1})E^2T_{4,i_1}T_{3,i_2-1}T_{2,j_1}T_{1,j_2-2}-E^2T_{4,i_1}
T_{2,i_2-1}T_{2,j_1}T_{1,j_2-2}.
\end{equation}
Since  $j_2-2<j_1<i_1$ and $i_2-1<i_1$, the first summand in the RHS of  \eqref{jsjhd} is in $\tilde N^2_{i_1}$. Let $Y$ be the second summand in the  RHS of  \eqref{jsjhd}. It is enough to prove $Y\in \tilde N^2_{i_1}$.
Note that we are assuming that $j_1<i_2$, we have    $j_1\leq i_2-1$.
Suppose  $j_1<i_2-1$.  Since
$i_2-1<i_1$, $j_2-2<i_2-1$ and $j_1+1<i_1$, we have  $Y=E^2T_{4,i_1}T_{3,j_1+1}T_{2,i_2-1}T_{1,j_2-2}\in \tilde N^2_{i_1}$.
 Suppose $ j_1=i_2-1$. Since $j_1<i_1$ and  $j_2-2<j_1-1$,  $Y=E^2T_{4,i_1}T_{3,j_1}T_{2,j_1-1}T_{j_1-1}^2T_{1,j_2-2}$.  So
\begin{equation}\label{ksjdjsd}
Y
= (q-q^{-1})E^2T_{4,i_1}T_{3,j_1}T_{2,j_1-1}T_{1,j_2-2}+ E^2T_{4,i_1}T_{3,j_1}T_{2,j_1}T_{1,j_2-2}\in \tilde N^2_{i_1}.
\end{equation}
This completes the proof of our claim in case (1).
In case (2), we have  $j_1<j_2-1<i_2-1<i_1$ and  $\text{ RHS of \eqref{key2}} =-E^2 T_{4,i_1}T_2^{-1}T_{3,i_2-1}T^{-1}_1T_{2,j_2-1}T_{1,j_1}$.
So
\begin{equation}\label{key3} \text{ RHS of \eqref{key2}}
=-q(q-q^{-1})E^2 T_{4,i_1}T_{3,i_2-1}T^{-1}_1T_{2,j_2-1}T_{1,j_1} -E^2 T_{4,i_1}T_{2,i_2-1}T^{-1}_1T_{2,j_2-1}T_{1,j_1}.
\end{equation}
The first summand  in the RHS of \eqref{key3} is in $\tilde N^2_{i_1} $ since $j_2-1<i_1$, $i_2-1<i_1$ and $j_1<j_2-1$.
We have
$$ T_{4,i_1}T_{2,i_2-1}T^{-1}_1T_{2,j_2-1}T_{1,j_1}= (q-q^{-1})T_{4,i_1}T_{3,j_2}T_{2,i_2-1}T_{1,j_1} + T_{4,i_1}T_{3,j_1+2}T_{2,i_2-1}T_{1,j_2-1} $$
and $i_2-1<i_1$, $j_2<i_1$, $j_1<i_2-1$, $j_1+2<i_1$ and $j_2-1<i_2-1$. This proves that  the second summand in the RHS of \eqref{key3} is also in $\tilde N^2_{i_1} $. So far, we have proved our claim in any case.

 Thanks to \eqref{dlf},  $ D_{f, l}\subset \tilde D_{f, l}$ forcing $\tilde N^f\supseteq N^f$. In order to prove (3), it is enough to prove  \begin{equation}\label{key4} E^f T_d\in N^f \text{ for any $d\in\tilde D_{f,l}$.}\end{equation}
We prove \eqref{key4} by induction on    $f$ and $i_1$.  When $f=1$  this result is trivial since $N^1=\tilde N^1$. In general
by induction assumption on $f$, we can assume that $i_2<i_3<\ldots< i_f$. If $i_1=2$, then $i_1<i_2$ and hence $d\in D_{f,l}$ and $ E^f T_d\in N^f$.  Suppose  $i_1>2$. It is enough to assume $i_1\geq i_2 $. By our previous  claim,
$ E^f T_d \subseteq M$ where $M$ is the left $ \mathcal B_l(f) $-module generated by $\{E^f T_{d'}\mid d'\in \tilde D_{f,l}^{i_1} \}$, and  $\tilde D_{f,l}^{i_1}\subset \tilde D_{f,l}$
consisting  of elements $d'=s_{2f,i_f'}s_{2f-1,j_f'}\cdots s_{2,i_1'}s_{1,j_1'}$ such that $i_1'<i_1$.
Using  induction assumption on $i_1$ yields  $ E^f T_d\in N^f$. This completes the proof of (3).
\end{proof}
%The following result  can be considered as the left version of Lemma~\ref{fque1}(3).
\begin{Cor}\label{jdijss1} Suppose  $f\le \lfloor l/2\rfloor$. Then
\begin{itemize}\item[(1)] $E^f \mathcal B_{q, l}$ is the left $\mathcal B_l(f)$-module generated by $ V^f=\{E^fT_d\mid d\in D_{f,l} \}$.
 \item[(2)] $\mathcal B_{q, l} E^f$ is the  right $\mathcal B_l(f)$-module generated by $\{\sigma(T_d)E^f\mid d\in D_{f,l}\}$
where $\sigma$ is $\Bbbk$-linear anti-involution of $\mathcal B_{q,l}$  in Lemma~\ref{antiin}.\end{itemize}
\end{Cor}
\begin{proof} (1) follows from Lemma~\ref{fque1} and (2) follows from (1) by applying the anti-involution $\sigma$. \end{proof}
\begin{Theorem}\label{thm123}  Suppose $ \Bbbk$ is an integral domain with characteristic not $2$.
\begin{itemize}\item [(1)] The algebra $\mathcal B_{q, l}$ has basis  $\mathcal M$, where
$\mathcal M=\{ \sigma(T_d)E^f T_w T_v\mid 1\leq f\leq \lceil l/2 \rceil, d,v\in D_{f,l}, w\in\mathfrak S_{l-2f} \}$. In particular, the rank of $\mathcal B_{q,l}$ is $(2l-1)!! $.
\item [(2)]  As $\Bbbk$-algebras,  $\mathcal B_{q,l}\cong \Hom_{\mathcal B}(l,l)$ for any $l\in \mathbb N$.\end{itemize}
\end{Theorem}
\begin{proof}Let $\tilde H_{l-2f}$ be the Hecke algebra associated to the symmetric group $\mathfrak S_{l-2f}$ on $l-2f$ letters $2f+1, 2f+2, \ldots, l$.
There is a $\Bbbk$-algebra epimorphism from $\tilde H_{l-2f}$ to $\mathcal B_l(f)/I$ sending $T_{2f+i}$'s to $T_{2f+i}$'s,
  where $I$ is the two-sided ideal of $\mathcal B_l(f)$ generated by $E_{2f+1}$.  Let \begin{equation}\label{2sidei} \mathcal B^f_{q,l}=\mathcal B_{q, l} E^f \mathcal B_{q, l}.\end{equation}
  Thanks to  Lemma~\ref{fque1}(3) and Corollary~\ref{jdijss1},  $\mathcal B^f_{q,l}/\mathcal B^{f+1}_{q,l}$ is spanned by
$$\{ \sigma(T_d)E^f T_w T_v\mid  d,v\in D_{f,l}, w\in\mathfrak S_{l-2f} \}.$$  This proves that $\mathcal B_{q, l}$ is spanned by $\mathcal M$ as $\Bbbk$-module.
In order to prove that $\mathcal M$ is linear independent over $\Bbbk$, it is enough to prove that  $\mathcal M$ is linear independent over $F$, the quotient field of $\Bbbk$.

Suppose  both $\mathcal B_{q, l}$
and $\mathcal B$ are defined over $F$. There is an algebra  homomorphism
 $ \varphi: \mathcal B_{q,l} \rightarrow \Hom_{\mathcal B}(l,l)$
 sending  $T_i$ and $E_i$ to $1_{i-1}\otimes \begin{tikzpicture}[baseline = 2.5mm]
	\draw[-,thick,darkblue] (0.28,0) to[out=90,in=-90] (-0.28,.6);
	\draw[-,line width=4pt,white] (-0.28,0) to[out=90,in=-90] (0.28,.6);
	\draw[-,thick,darkblue] (-0.28,0) to[out=90,in=-90] (0.28,.6);
\end{tikzpicture}\otimes 1_{l-i-1}$ and $1_{i-1}\otimes \begin{tikzpicture}[baseline = 3pt, scale=0.5, color=\clr]
        \draw[-,thick] (1,-0.1) to[out=up, in=right] (0.53,0.4) to[out=left, in=right] (0.47,0.4);
        \draw[-,thick] (0.49,0.4) to[out=left,in=up] (0,-0.1);
        \draw[-,thick] (1,1.1) to[out=down, in=right] (0.53,0.6) to[out=left, in=right] (0.47,0.6);
        \draw[-,thick] (0.49,0.6) to[out=left,in=down] (0,1.1);
\end{tikzpicture}\otimes 1_{l-i-1} $, respectively.
  Using induction on the number of crossings for a reduced $(l, l)$-tangle diagram, one can write   any reduced $(l,l)$-tangle diagram as a linear combination of   $T_w d T_{w'}$ for some $w,w'\in \mathfrak S_l$, where
$d=( \begin{tikzpicture}[baseline = 3pt, scale=0.5, color=\clr]
        \draw[-,thick] (1,-0.1) to[out=up, in=right] (0.53,0.4) to[out=left, in=right] (0.47,0.4);
        \draw[-,thick] (0.49,0.4) to[out=left,in=up] (0,-0.1);
        \draw[-,thick] (1,1.1) to[out=down, in=right] (0.53,0.6) to[out=left, in=right] (0.47,0.6);
        \draw[-,thick] (0.49,0.6) to[out=left,in=down] (0,1.1);
\end{tikzpicture})^{\otimes r} \otimes 1_{l-2r}$ for some $1\leq r\leq l/2$. This proves that
 $\varphi$ is surjective. Thanks to Theorem~\ref{main1},   $\dim_F  \Hom_{\mathcal B}(l,l)=(2l-1)!!=\# \mathcal M$ and hence  $ \dim_F \mathcal B_{q, l}=(2l-1)!!$, forcing $\mathcal M$ to be linear independent over $F$. This completes the proofs of (1). Now, (2) follows from (1) and the fact that there is  an algebra epimorphism $ \varphi: \mathcal B_{q,l} \rightarrow \Hom_{\mathcal B}(l,l)$ over $\Bbbk$.
\end{proof}

\section{Representations of the periplectic $q$-Brauer category $\mathcal B$}
In this section we assume that  $\Bbbk$ is an algebraically closed field with characteristic not $2$. We show that $\mathcal B$ admits a split triangular decomposition in the sense of \cite{BS}. Consequently,  the category of locally finite dimensional right $\mathcal B$-modules is an upper finite fully stratified category in the sense of~\cite[Definition~3.34]{BS}.

\subsection{Split triangular decomposition}
Let $A$ be the locally unital associative $\Bbbk$-algebra associated to $\mathcal B$. Then
\begin{equation}\label{locaglebra}
A=\oplus_{a,b\in \mathbb N}1_aA1_b,\quad 1_a A1_b=\Hom_{\mathcal B}(b,a)
\end{equation}
with multiplication being induced by composition in $\mathcal B$. The mutually orthogonal idempotents $\{1_a\mid a\in \mathbb N\}$ serves as a family of
 distinguished local units of $A$. Although   $\mathcal B$ is a supercategory, we  regard $A$ as an ordinary associative $\Bbbk$-algebra by forgetting the super structure of $\mathcal B$.

\begin{Defn}\label{pm0} Let $A$ be the locally unital algebra associated to $\mathcal B$ in \eqref{locaglebra}. Define
\begin{itemize}
\item[(1)] $A^\circ:=\oplus_{r\in \mathbb N}A^\circ_r$, where  $A^\circ_r$ is the  $\Bbbk$-span of all $(r, r)$-tangle diagrams on which  there are neither caps nor cups.
\item [(2)] $A^-$: $\Bbbk$-span of all tangle diagrams on which  there are neither caps  nor crossings among   vertical strands.
\item [(3)] $A^+$: $\Bbbk$-span of all tangle diagrams on which  there are neither  cups nor crossings among   vertical strands.
\item [(4)] $A^\sharp$: $\Bbbk$-span of all tangle diagrams on which  there are no caps.
\item [(5)] $A^b$: $\Bbbk$-span of all tangle diagrams on which  there are no cups.
\end{itemize}
\end{Defn}
Thanks to Definition~\ref{pqbc}, it is not difficult  to verify that all tangle diagrams in  $A^\diamond$  are closed under composition and hence are  locally unital subalgebras of $A$ where
$\diamond\in \{\circ, -, +,\sharp, b\}$.
  Moreover,
  \begin{equation}\label{uppd}  m-s=2k \text{ for some $ k\in\mathbb N$, if either
  $1_mA^\diamond1_s\neq 0$ for $\diamond\in \{-,\sharp\}$ or  $1_sA^{\star}1_m\neq 0$ for $\star\in \{+,b\}$.}
  \end{equation} Recall
  $ \mathbb{RL}_{m,s} $ in \eqref{rl} for all $m, s\in \mathbb N$. Define
  \begin{equation}\label{rls} \mathbb{RL}^\diamond=\bigcup_{m, s} \mathbb{RL}^\diamond_{m,s},\end{equation}
  where  $\mathbb{RL}^\diamond_{m,s}=\mathbb{RL}_{m,s}\cap A^\diamond$ and $\diamond\in \{\circ, -, +,\sharp, b\}$.
  \begin{Lemma}\label{localsa} The locally unital subalgebra $A^\diamond$ has $\Bbbk$-basis given by  $\mathbb{RL}^\diamond$ where $\diamond\in \{\circ, -, +,\sharp, b\}$.
Moreover,  the locally unital subalgebra  $A^\circ_r$ has basis given by $\mathbb{RL}^\circ_{r,r}$.
\end{Lemma}
    \begin{proof}Using induction on the number of crossings of tangle diagrams, we see that the result is true if $\mathbb{RL}_{m,s}$ is chosen as $\mathbb {DB}_{m,s}$. In general, by  Definition~\ref{pqbc}(R2) and induction arguments on the number of crossings of tangle diagrams, any reduced lift of a connector $c$ is equal to the corresponding reduced totally descending tangle
    diagram $D_c$ up to a linear combination of reduced totally descending tangle diagrams with fewer crossing numbers. This implies the result in general case.
    \end{proof}
%Define It is a locally unital subalgebra of $A$. We have the following triangular decomposition of $A$.

\begin{Lemma}\label{triangula} As $\Bbbk$-spaces, $A^-\otimes _{\mathbb K}A^\circ\otimes_{\mathbb K} A^+\cong A$, $A^-\otimes _{\mathbb K}A^\circ\cong A^\sharp$ and
$A^\circ\otimes_{\mathbb K} A^+\cong A^b$, where  $\mathbb K:= \oplus _{a\in \mathbb N}\Bbbk 1_a$.
Moreover, these isomorphisms are given by the multiplication of $A$.
\end{Lemma}
\begin{proof}Fix a $\Bbbk$-basis  $\mathbb{RL}=\bigcup_{m, s}  \mathbb{RL}_{m,s}$ of $A$  where  $\mathbb{RL}_{m,s}$ is given in \eqref{rl}. By Lemma~\ref{localsa}, we have $\Bbbk$-basis of
$A^\diamond$ for any $\diamond\in \{\circ, -, +\}$.
Let  $\phi: A^-\otimes _{\mathbb K}A^\circ\otimes_{\mathbb K} A^+\rightarrow A $  be the linear map induced by the multiplication of $A$.
The image of $(b^-, b^\circ, b^+)\in \mathbb{RL}_{m,t}^-\times \mathbb {RL}_{t,t}^\circ\times \mathbb{RL}_{t,s}^+$
is a reduced lift of  $\text{conn}(b^-b^\circ b^+)$ and further $\text{conn}( b^-b^\circ b^+)\neq \text{conn}( b_1^-b_1^\circ b_1^+)$ if $(\text{conn}( b^-), \text{conn}( b^\circ), \text{conn}( b^+))\neq (\text{conn}( b_1^-), \text{conn}( b_1^\circ), \text{conn}( b_1^+))
$ and $(b_1^-, b_1^\circ, b_1^+)\in \mathbb{RL}_{m, t'}^-\times \mathbb {RL}_{t',t'}^\circ\times \mathbb{RL}_{t', s}^+$. By Lemma~\ref{localsa}  $\phi$ is injective.  Since the dimension of  $1_m A 1_s$ is $\sum_{t}|\mathbb {RL}^-_{m, t}| |\mathbb{RL}^\circ_{t,t}|| \mathbb{RL}^+_{t,s}|$, the image of $\bigcup_t \mathbb{RL}_{m,t}^-\times \mathbb {RL}_{t,t}^\circ\times \mathbb{RL}_{t,s}^+$ is a $\Bbbk$-basis of  $1_m A 1_s$. This proves that
$\phi$ is surjective and hence an isomorphism as required. One can check the second and the third isomorphisms similarly.
\end{proof}

We briefly recall the notion of a split triangular decomposition of the locally unital algebra associated to  small finite dimensional $\Bbbk$-linear category $\mathcal C$  whose object set is $I$~\cite{BS}. In this case $\dim_{\Bbbk} \Hom_{\mathcal C}( a, b)<\infty$   for all $a,  b\in I$. Let $C$ be the locally unital algebra associated to $\mathcal C$.

 \begin{Defn}\label{WT}\cite[Remark~5.32]{BS} The data $(I, C^-, C^\circ, C^+)$ is called a
 \textsf{ split  triangular decomposition}  of  $C$ (or $C$ admits a split triangular decomposition)   if:
\begin{itemize} \item[(1)]  $(I, \preceq)$ is an upper finite  poset in the sense that
 $\{b\in I\mid a\prec b\}$ is a finite set for any $a\in I$.
\item[(2)] $C^-, C^\circ$ and $  C^+$ are  locally unital subalgebras of $C$ in the sense that
 $C^\pm =\bigoplus_{b, c \in I} 1_b C^\pm 1_c $ and  $ C^\circ =\bigoplus_{a\in I} 1_a C^\circ 1_a$.
\item [(3)] $1_a C^- 1_{  b}=0$ and   $1_b C^+ 1_{      a}=0$ unless $   a\preceq    b$.  Furthermore,  $1_a C^- 1_{ a }=1_aC^+1_{a}=\Bbbk 1_{a}$.
    \item[(4)] $C^-C^\circ$ and $C^\circ C^+$ are locally unital subalgebras of $C$.
    \item[(5)]  $C^-\otimes_{\mathbb K }  C^\circ \otimes_{\mathbb K}  C^+\cong C$
as  $\Bbbk$-spaces where $\mathbb K=\oplus_{  b\in I}  \Bbbk 1_{ b}$.   The required isomorphism  is given by the multiplication on $C$.\end{itemize}\end{Defn}

Now, we consider the category $\mathcal B$ whose object set is $\mathbb N$. Let $\preceq$ be the partial order on $\mathbb N$ such that $s\preceq m$ if $s-m\in 2\mathbb N$. Then  $(\mathbb N, \preceq)$ is upper finite.

\begin{Theorem}\label{tria} Let $A$ be the locally unital algebra associated to $\mathcal B$. Then $A$ admits a split triangular decomposition. The corresponding date is $(\mathbb N, A^-, A^\circ, A^+)$, where $ A^-, A^\circ, A^+$ are defined in Definition~\ref{pm0}.
\end{Theorem}
\begin{proof} Definition~\ref{WT}(1) is obvious.  By Definition~\ref{pm0}, \eqref{uppd}  and  Lemma~\ref{triangula}, we see that the locally unital algebra $A$ associated to $\mathcal B$ satisfies Definition~\ref{WT}(2)--(5). \end{proof}

\subsection{Representations of   $\mathcal B$}\label{hhh1}
Recall $H_l$ is the Hecke algebra associated to symmetric group $ S_l$.
 Thanks to  Lemma~\ref{localsa},  $ H_l\cong 1_l A^\circ 1_l$ and  the required isomorphism sends   $T_i$ to $1_{i-1}\otimes \begin{tikzpicture}[baseline = 2.5mm]
	\draw[-,thick,darkblue] (0.28,0) to[out=90,in=-90] (-0.28,.6);
	\draw[-,line width=4pt,white] (-0.28,0) to[out=90,in=-90] (0.28,.6);
	\draw[-,thick,darkblue] (-0.28,0) to[out=90,in=-90] (0.28,.6);
\end{tikzpicture}\otimes 1_{l-i-1}$,
$1\leq i\leq l-1$. Consequently, there is a $\Bbbk$-algebra isomorphism  \begin{equation}\label{zerop} A^\circ\cong \bigoplus_{l\in \mathbb N} H_l.\end{equation}

Recall that a composition  $\lambda$ of $l$ is a sequence $(\lambda_1, \lambda_2, \cdots )$ of non-negative integers such that $|\lambda|:=\sum_i \lambda_i=l$. If $\lambda_i\ge \lambda_{i+1}$ for all possible $i$, then $\lambda$ is called a partition.  Given a positive integer $e$, a partition $\lambda$ of $l$ is called $e$-restricted if $\lambda_i-\lambda_{i+1}<e$ for all possible $i$. When $e$ is infinity, any partition of $l$ is $e$-restricted. Let $\Lambda(l)$ (resp., $\Lambda^+(l)$, resp.,   $\Lambda^+_e(l)$) be the set of all compositions (resp., partitions, resp.,  $e$-restricted partitions) of $l$, where $e$ is always the quantum characteristic of $q^2$ (so $e=\infty$ means that  $q^2$ is not a root of unity).
Define \begin{equation}\label{pare} \Lambda=\bigcup_{l\in\mathbb N}\Lambda^+(l), \text{  and } \
 \Lambda_e=\bigcup_{l\in\mathbb N}\Lambda^+_e(l).\end{equation}
 For each $\lambda\in \Lambda^+(l)$, let  $S^\lambda$ be the dual Specht module for $H_l$ (which is not the classical Specht module). This is a right $H_l$-module defined via Jucys-Murphy basis of $H_l$ in Theorem~\ref{mur}. At moment, we do not need its explicit construction. It is known that $S^\lambda$ has the simple head, say $D^\lambda$ if and only if $\lambda\in \Lambda^+_e(l)$ and $\{D(\lambda)\mid \lambda\in \Lambda_e\}$ is a complete set of inequivalent irreducible $A^\circ$-modules. For each $\lambda\in \Lambda^+_e(l)$, let   $Y(\lambda)$ be the projective cover of $D(\lambda)$.

Let $A^\circ$-fdmod be the category of finite dimensional right $A^\circ$-modules and  $A$-lfdmod be the category of
locally finite dimensional right $A$-modules. Recall $A^b$ and $A^\sharp$ in Definition~\ref{pm0}. % and the split triangular decomposition result in Theorem \ref{tria}.
 It follows from Theorem~\ref{tria} and  \cite[(5.13)-(5.14)]{BS}\footnote{Brundan and Stroppel considered the category of locally finite dimensional left $A$-modules.}  that there are  two exact functors $\Delta, \nabla: A^\circ\text{-fdmod}\rightarrow A\text{-lfdmod} $ called induction and coinduction functors respectively  such that
 \begin{equation} \label{func} \Delta=?\otimes _{A^\sharp } A\text{ and $\nabla=\oplus_{m\in \mathbb N}\Hom_{A^b}(1_mA, ?)$.}\end{equation}
 The following result follows from  \cite[Theorem~5.38]{BS} which is about  locally unital and locally finite dimensional $\Bbbk$-algebra  admitting  a split triangular decomposition.  Such a result  holds for upper finite  weakly triangular categories, a generalized version of split triangular decomposition~\cite[Theorem~3.5]{GRS2}.

 \begin{Theorem} (c.f. \cite[Theorem~5.38]{BS})\label{strdcdcj}
 \begin{itemize}
 \item [(1)] The category  $A$-lfdmod  is an upper finite fully stratified category in the sense of Brundan and Stroppel\cite[Definition~3.34]{BS}. In this case, the  corresponding stratification is $ \rho: \Lambda_e \rightarrow\mathbb N $  such that $\rho(\lambda)=l$ if $\lambda\in \Lambda_e^+(l)$.
\item[(2)] The proper standard  module $\bar \Delta(\lambda):=\Delta(D(\lambda))$ has a unique maximal submodule, say $\text{rad}\bar\Delta(\lambda)$, for any $\lambda\in\Lambda_e$.
    \item[(3)] For any $\lambda\in \Lambda_e$, let $L(\lambda)$ be the simple head of $\bar \Delta(\lambda)$. Then
$\{L(\lambda)\mid \lambda\in \Lambda_e\}$ is a complete set of pairwise inequivalent irreducible $A$-modules.
\item[(4)] When $q^2$ is  not a root of unity,  the category  $A$-lfdmod  is an upper finite highest weight category in the sense of \cite[Definition~3.34]{BS}.
 \end{itemize}
 \end{Theorem}
 \begin{proof} The results (1)-(3) follow immediately from Theorem~\ref{tria} and  \cite[Theorem~5.38]{BS}. Finally,
 (4) follows from \cite[Corollary~5.36]{BS} since $A^\circ$ is semisimple when $q^2$ is not a root of unity.  \end{proof}

 \subsection{Simple modules for periplectic   $q$-Brauer algebras}

 Thanks to Theorem~\ref{thm123}(2), $\mathcal B_{q, l}\cong 1_l A 1_l$ where $A$ is the locally unital algebra associated to $\mathcal B$. The idempotent element $1_l$ gives an exact functor from $A$-lfdmod to $\mathcal B_{q, l}$-fdmod, the category of finite dimensional right $\mathcal B_{q, l}$-modules. In Theorem~\ref{simsi}
 we assume that $\Bbbk$ is an algebraically closed field of characteristic not $2$ although we can use standard basis of  $\mathcal B_{q, l}$ in section~6 to classify simple $\mathcal B_{q, l}$-modules over an arbitrary field with characteristic not $2$.

\begin{Lemma}\label{emptysnozero}Suppose $M\in \{\bar\Delta(\lambda), L(\lambda)\}$ where $\lambda\in\Lambda^+_e(r)$. Then  $M 1_l\neq 0$  only if  $l-r\in2\mathbb N$. Moreover,  $L(\lambda)1_r\cong D(\lambda)$.
\end{Lemma}
\begin{proof}By \eqref{uppd} and Lemma~\ref{triangula}, \begin{equation}\label{ttt1} \bar \Delta(\lambda)=D(\lambda) \otimes _{A^\sharp}A\cong D(\lambda)\otimes_{\mathbb K} A^+=\bigoplus _{k\in\mathbb N}  D(\lambda) \otimes_{\mathbb K} 1_{r}A^+1_{r+2k},\end{equation}  proving the first statement for $\bar\Delta(\lambda)1_l$. Since there is an epimorphism from $\bar \Delta(\lambda)$ to  $L(\lambda)$ and $1_a$ is exact for any $a\in \mathbb N$, we have the result for $L(\lambda)1_l$, immediately. By \eqref{ttt1}, $\bar \Delta(\lambda)1_r=D(\lambda)\otimes_{\mathbb K} 1_r  $. Since $D(\lambda)$ is simple, we have $L(\lambda)1_r\cong  \bar \Delta(\lambda)1_r$, proving the last statement. \end{proof}

\begin{Lemma}\label{empytset} For any positive integer $l$, $L(\emptyset)1_l=0$ and  $\dim L(\emptyset)=1$.
\end{Lemma}
\begin{proof}
Let $N=\bigoplus _{r=1}^\infty \bar\Delta(\emptyset)1_{2r} =\bigoplus _{r=1}^{\infty}1_0A^-1_{2r} $.
 Obviously, $N 1_m A 1_n\subset N$ if $n\neq 0$.  Suppose $n=0$. Then  $N 1_m A 1_0=0$ unless $m$ is even and $m\neq 0$.  Suppose $m=2r$ for some positive integer $r$ and $N1_{m}A1_{0}=N 1_{2r} A 1_0\subset 1_0A^{-} 1_{2r} A 1_0$. Since $r>0$,  $1_0A^{-} 1_{2r} A 1_0$ is spanned by diagrams  with at least one loop.
Using Definition~\ref{pqbc}(R2), any one of these diagrams can be written as a linear combination of free loops without crossings. Thanks to Definition~\ref{pqbc}(R5), $1_0A^-1_{2c }A 1_0=0$. This proves that $N$ is a submodule of $\bar\Delta(\emptyset)$ and
 $\bar\Delta(\lambda) 1_l\subset N$ for all positive integers $l$. So $L(\emptyset)1_l=0$.
  Finally, $\dim L(\emptyset)=1$ since
$L(\emptyset)=\Bbbk$.
\end{proof}

For any $l\in \mathbb N$ define
\begin{equation}\label{sig123} \Sigma^+(l)=\bigcup_{0\le f\le \lfloor l/2\rfloor} \Lambda^+(l-2f)\ \ \text{ and } \ \ \Sigma_e^+(l)=\bigcup_{0\le f\le \lfloor l/2\rfloor} \Lambda_e^+(l-2f),\end{equation}
where $e$ is the quantum characteristic of $q^2$.
\begin{Theorem}\label{simsi} For any positive integer $l$, let
$\mathcal B_{q, l}$ be  defined over an algebraically closed field with characteristic not $2$ and $\lambda\in \Sigma_e^+(l)$.  \begin{itemize}\item[(1)] If $l$ is odd,  then  $L(\lambda)1_l\neq 0$ if and only if $\lambda\in \Sigma^+_{e}(l)$. Moreover,  the  set $\{L(\lambda)1_l\mid \lambda\in \Sigma^+_{e}(l)\}$ is a complete set of inequivalent irreducible  $\mathcal B_{q,l}$-modules.
\item[(2)] If $l$ is even,  then  $L(\lambda)1_l\neq 0$ if and only if $\lambda\in \Sigma^+_{e}(l)$ and $\lambda\neq \emptyset$. Moreover, the  set $\{L(\lambda)1_l\mid \lambda\in \Sigma^+_{e}(l)\setminus\{\emptyset\}\}$ is a complete set of inequivalent irreducible  $\mathcal B_{q,l}$-modules.
\end{itemize}
\end{Theorem}
\begin{proof}  By Lemmas~\ref{emptysnozero}--\ref {empytset}, we have  $\lambda\in \Sigma^+_{e}(l)$ (resp., $\Sigma^+_{e}(l)\setminus\{0\}$)
if $L(\lambda)1_l\neq 0$ and $l$ is  odd (resp., even).  Conversely, suppose $\lambda\in \Sigma^+_{e}(l)$. Then $\lambda \in \Lambda_e^+(r)$ and $l-r$ is even.  By  Lemma~\ref {emptysnozero}, $L(\lambda)1_r\cong D(\lambda)$. Since  $1_r=(1_{r-1}\otimes\lcap^{\otimes \frac12(l-r)}\otimes 1_1  )\circ (1_r \otimes \lcup^{\otimes \frac12(l-r)})$,  $(1_{r-1}\otimes\lcap^{\otimes \frac12(l-r)}\otimes 1_1  )\in A 1_l$ and   $L(\lambda)1_r\neq 0$, we have $L(\lambda)1_l\neq 0$. Finally, using the exact functor
$1_l: A\text{-lfdmod}\rightarrow \mathcal B_{q,l}\text{-fdmod}$ sending  $M$ to $  M1_l$, we have the remaining statements immediately.
\end{proof}

\subsection{Duality functor}
By Definition~\ref{pqbc},  there is  a contravariant auto-equivalence $\phi: \mathcal B\rightarrow \mathcal B$
such that $\phi$ fixes any  object of $\mathcal B$  and sends generating morphisms
 $\begin{tikzpicture}[baseline = 2.5mm]
	\draw[-,thick,darkblue] (0.28,0) to[out=90,in=-90] (-0.28,.6);
	\draw[-,line width=4pt,white] (-0.28,0) to[out=90,in=-90] (0.28,.6);
	\draw[-,thick,darkblue] (-0.28,0) to[out=90,in=-90] (0.28,.6);
\end{tikzpicture}, \begin{tikzpicture}[baseline = 2.5mm]
	\draw[-,thick,darkblue] (-0.28,-.0) to[out=90,in=-90] (0.28,0.6);
	\draw[-,line width=4pt,white] (0.28,-.0) to[out=90,in=-90] (-0.28,0.6);
	\draw[-,thick,darkblue] (0.28,-.0) to[out=90,in=-90] (-0.28,0.6);
\end{tikzpicture}$, $\lcup,\lcap$ to $-~\begin{tikzpicture}[baseline = 2.5mm]
	\draw[-,thick,darkblue] (-0.28,-.0) to[out=90,in=-90] (0.28,0.6);
	\draw[-,line width=4pt,white] (0.28,-.0) to[out=90,in=-90] (-0.28,0.6);
	\draw[-,thick,darkblue] (0.28,-.0) to[out=90,in=-90] (-0.28,0.6);
\end{tikzpicture}, -~\begin{tikzpicture}[baseline = 2.5mm]
	\draw[-,thick,darkblue] (0.28,0) to[out=90,in=-90] (-0.28,.6);
	\draw[-,line width=4pt,white] (-0.28,0) to[out=90,in=-90] (0.28,.6);
	\draw[-,thick,darkblue] (-0.28,0) to[out=90,in=-90] (0.28,.6);
\end{tikzpicture}, -\lcap,\lcup$, respectively, and $\phi(b\otimes b')=\phi(b)\otimes \phi(b')$ for all morphisms $b, b'$ in $\mathcal B$.  Standard arguments in \cite[\S3.3]{GRS2} shows that  $\phi$ induces an anti-automorphism of $A$.  Restricting it on  $1_mA1_m$ and $A^\circ $ yield  the corresponding  anti-involutions on $1_mA1_m$ and $A^\circ$, respectively.

 For any $M\in A$-lfdmod, let
 $M^*:=\oplus_{a\in \mathbb N}\Hom_{\Bbbk}(M1_a,\Bbbk)$. Then $M^*$ is a locally finite dimensional right $A$-module such that
 $$ (fx)(m)=f(m\phi(x)), \text { for }x\in A, f\in M^*, m\in M1_a.$$
 This induces an exact contravariant duality functor on $A$-lfdmod. Similarly, we have the exact contravariant duality functors on $1_mA1_m$-fdmod and $A^\circ$-fdmod.
 Let $*:  \Lambda_e\rightarrow \Lambda_e $ be the involution  map such that $D(\lambda)^*\cong D(\lambda^*)$.

 %If $e=0$, then $S(\lambda)=D(\lambda)$. By \cite[Exercise~3.14(iii)]{Ma}\footnote{Note that the relation of the Hecke algebra in \cite{Ma} is different from the paper. They use $(T_i-q)(T_i+1)=0$. So, our anti-involution induced by $\phi$ corresponds to the anti-involution of the Hecke algebra in \cite{Ma} sending $T_i$ to $-qT_i^{-1}$.  },   $\lambda^*=\lambda^t$, the dual partition of $\lambda$.
 %\begin{Lemma}\cite{GRS2}\label{duals} $*\circ \Delta\cong \nabla\circ *$.
 %\end{Lemma}

 \begin{Cor}\label{blmul} Suppose $\lambda\in \Lambda_e$ and $\mu\in \Lambda$.
 We have
 \begin{itemize}
 \item[(1)]  $\bar\Delta(\lambda)^*\cong \bar\nabla(\lambda^*)$, $L(\lambda)^*\cong L(\lambda^*)$, where $\bar\nabla(\lambda):=\nabla(D(\lambda))$.
 \item [(2)] $\dim \Hom_A(\Delta(\mu),\bar\Delta(\lambda)^*)=\delta_{\mu,\lambda^*}$, where $\Delta(\lambda):=\Delta(Y(\lambda))$.
 \item[(3)] If $e>l$ and $\lambda,\mu \in\sum^{+}(l)$,  then  $(P(\lambda):\Delta(\mu))=[\Delta(\mu^t): L(\lambda^t)] $, where $P(\lambda)$ is the projective cover of $L(\lambda)$ and $\lambda^t$ is the transpose of $\lambda$.
 \end{itemize}
 \end{Cor}
 \begin{proof} A category which admits a split triangular decomposition is an upper finite weakly triangular category in the sense of \cite[Definition~2.1]{GRS2}. This enables us to use \cite[Lemma~3.11]{GRS2} to get
   $*\circ \Delta\cong \nabla\circ *$. Applying both of them on $D(\lambda)$ yields isomorphisms in (1). By general results in \cite[Lemma~3.48]{BS},  $\dim \Hom_A(\Delta(\mu),\bar\nabla(\lambda))=\delta_{\lambda,\mu}$ and hence (2) follows from (1), immediately.
Suppose $e>l$. Then $H_l$ is semisimple. In this case,   $S(\lambda)=D(\lambda)=Y(\lambda)$,  $\Delta(\lambda)=\bar\Delta(\lambda)$ and $ \nabla(\lambda)=\bar\nabla(\lambda)$, where $\nabla(\lambda):=\nabla(I(\lambda)) $ and $I(\lambda)$ is the injective hull of $D(\lambda)$.
Since the Hecke algebra $H_l$ is a symmetric algebra, $I(\lambda)=Y(\lambda)$ for all $\lambda\in \Lambda^+_e(l)$.
Thanks to \cite[Exercise~3.14(iii)]{Ma}, $\lambda^*=\lambda^t$, where $\lambda^t$ is the conjugate of $\lambda$.\footnote{In \cite{Ma}, $(T_i-q)(T_i+1)=0$ whereas we use $(T_i-q)(T_i+q^{-1})=0$.
 So our anti-involution induced by $\phi$ corresponds to the anti-involution of the Hecke algebra in \cite{Ma} which  sends $T_i$ to $-qT_i^{-1}$.} By \cite[Lemma~3.36]{BS} or \cite[Proposition~3.9(2)]{GRS2}, $(P(\lambda):\Delta(\mu))=[\nabla(\mu):L(\lambda)]$. Since $*$ is an exact functor, by (1)
 $ [\nabla(\mu):L(\lambda)]=[\nabla(\mu)^*:L(\lambda)^* ]=[\Delta(\mu^*):L(\lambda^*)] =[\Delta(\mu^t):L(\lambda^t)]$, proving  (3).
 \end{proof}

 %In the remaining of this paper,  we are going to  construct  Jucys-Murphy basis and Jucys-Murphy elements of the quantum periplectic Brauer algebras  From now on, we focus on classification of blocks of $\mathcal B$ over the complex number field $\mathbb C$ when $q$ is not a root of unity.
 %Corollary~\ref{blmul}(3) will be very useful.

 %\end{rem}

\section{Classical Brunching rules for   periplectic $q$-Brauer algebras}
The aim of this section is  to give the classical branching rule for right standard modules of   $\mathcal B_{q, l}$  over $\Bbbk$. We will use it to construct Jucys-Murphy basis for  the right standard modules of  $\mathcal B_{q, l}$ next section.

\subsection{ Cell filtration of cell  modules and  permutation modules for $\tilde H_{l-2f}$}
Suppose $0\le f\le \lfloor l/2\rfloor$. Let $\widetilde{\mathfrak S}_{l-2f}$ be the symmetric group on letters $\{2f+1, 2f+2, \ldots, l\}$. When $l$ is even and $f=l/2$, we set  $\mathfrak S_{l-2f}=1$. Let  $\tilde H_{l-2f}$ be the Hecke algebra associated to $\widetilde{\mathfrak S}_{l-2f}$. Then $\tilde H_{l-2f}$, a subalgebra of $H_{l}$, is isomorphic to $H_{l-2f}$. The corresponding isomorphism  sends  $T_{2f+i}$'s to $T_i$'s in \eqref{kkk1}.
 %In the remaining part of this subsection, $i$ should be considered as $2f+i$.

For any $\lambda\in \Lambda(l-2f) $, let   $$x_\lambda=\sum_{w\in \widetilde{\mathfrak S}_\lambda}q^{\ell(w)}T_w,$$ where $\widetilde{\mathfrak S}_\lambda$ is the Young subgroup of $\widetilde{\mathfrak S}_{l-2f}$  with respect to $\lambda$. Each $\lambda$ corresponds to the Young diagram $[\lambda]$ such that there are $\lambda_i$ boxes in the $i$-th row of $[\lambda]$. A $\lambda$-tableau $\t$ is obtained from $[\lambda]$ by inserting $2f+1, 2f+2, \ldots, l$ into $[\lambda]$ without repetition. In this case, we write $shape(\t)=\lambda$ and call $\lambda$ the shape of $\t$.
 Let $\t^\lambda$ be the $\lambda$-tableau obtained from $[\lambda]$ by inserting $2f+1, 2f+2, \ldots, l$ successively from left to right along the rows of $[\lambda]$. If the entries in $\t$ increase from left to right along row and down column, $\t$ is called standard. In this case, $\lambda\in \Lambda^+(l-2f)$.
Let $\Std(\lambda)$ be the set of all standard $\lambda$-tableaux.
 Then each $\t\in \Std(\lambda)$ satisfies $\t=\t^\lambda d(\t)$ where $d(\t)$ is a distinguished right coset representative of $\widetilde{\mathfrak S}_\lambda$ in $\widetilde{\mathfrak S}_{l-2f}$.
The following result gives  the Jucys-Murphy basis of $\tilde H_{l-2f}$.
\begin{Theorem}\label{mur}\cite[Theorem~3.20]{Ma}  $\tilde H_{l-2f}$ has $\Bbbk$-basis given by
$\{ x_{\s,\t}\mid  \s,\t\in\Std(\lambda),  \lambda\in\Lambda^+(l-2f)  \}$,
where $x_{\s,\t}=T_{d(\s)}^*x_\lambda T_{d(\t)}$ and $*$ is the anti-involution on $\tilde H_{l-2f}$ fixing the generators $T_i$'s. It is a cellular basis in the sense of \cite[Definition~1.1]{GL}. \end{Theorem}

Suppose  $\lambda\in\Lambda^+(l-2f)$ and $\mu\in \Lambda^+(l-2f-1)$. We say that $\mu$ is obtained from $\lambda$ by removing the removable node $p=(k, \lambda_k)$ of $\lambda$ if $\mu_j=\lambda_j$, $j\neq k$ and $\mu_{k}=\lambda_k-1$. In this case, $\lambda$ is obtained from $\mu$ by adding the addable node $p$ of $\mu$.  We write $\lambda=\mu\cup p$ or $\mu=\lambda\setminus p$.
Let $\mathcal R_\lambda$ be the set of all
partitions obtained from $\lambda$ by removing a removable node and $\mathcal A_\lambda$ the set of all
partitions obtained from $\lambda$ by adding an addable node.

For any  $\lambda\in \Lambda^+(l-2f)$, let  $S^{\lambda}$ be  the cell module of   $\tilde H_{l-2f}$ with respect to the Jucys-Murphy basis in Theorem~\ref{mur}.  When $f=0$, it is the dual Specht module in subsection~\ref{hhh1}. By definition,    $S^{\lambda}$  is the free $\Bbbk$-module with basis
$\{x_{\t}\mid \t\in \Std(\lambda)\}$ where $x_\t:=x_{\t^\lambda, \t}+ \tilde H_{l-2f}^{\rhd \lambda}$ and $\tilde H_{l-2f}^{\rhd \lambda}$ is the two-sided ideal of $\tilde H_{l-2f}$ with $\Bbbk$-basis
$\{x_{\u, \s} \mid \u, \s \in \Std(\mu), \mu\rhd \lambda\}$. Since $\mathcal R_\lambda$
is a finite set, we arrange them as $\mu^{(1)}, \mu^{(2)},\ldots, \mu^{(a)}$ for some positive integer $a$ such that $\mu^{(1)}\rhd \mu^{(2)}\rhd \cdots\rhd \mu^{(a)}$. For any standard $\lambda$-tableau $\t$, let
$\t\downarrow_{l-1}$ be obtained from  $\t$ by removing the entry $l$. Then $\t\downarrow_{l-1}\in \Std(\mu)$ for some  $\mu\in \mathcal R_\lambda$. Let $S_i$ be the free $\Bbbk$-submodule of $S^{\lambda}$ spanned by $\{x_\t\mid  \t \in \Std(\mu^{(j)}), 1\le j\le i\}$. Then $S_i$ is a right $\tilde H_{l-2f-1}$-module such that $S^{\lambda}=S_a\supset S_{a-1}\supset \cdots \supset S_1\supset 0$, where $\tilde H_{l-2f-1}$ is the Hecke algebra associated to the symmetric group on $\{2f+1, 2f+2, \ldots, l-1\}$.

 \begin{Theorem}\label{hecbr}\cite[Proposition~6.1]{Ma} As $\tilde H_{l-2f-1}$-modules, $ S_i/S_{i-1}\cong S^{\mu^{(i)}}, 1\leq i\leq a$,
where $ S^{\mu^{(i)}}$ is the cell module  with respect to the Jucys-Murphy basis of
 $\tilde H_{l-2f-1}$ in Theorem~\ref{mur}. \end{Theorem}

%We need (semistandard) $\lambda$-tableau of type $\mu$ in \cite{Ma}. In this paper, we do not need their explicit definitions.  What we need is Lemma~\ref{matsho} which gives a basis and a cell filtration  of a permutation module for
 %$\tilde H_{l-2f}$.
 %First, we need a map sending a standard $\lambda$-tableau to a $\lambda$-tableau of type $\mu$ as follows.

   Suppose $(\lambda,\mu) \in\Lambda(l-2f)\times \Lambda(l-2f)$. Recall that a $\lambda$-tableaux of type $\mu$ is obtained from $[\lambda]$ by inserting integers  $i$ into
   $[\lambda]$ such that $i$ appears $\mu_i$ times.  A $\lambda$-tableau $S$ of type $\mu$ is called row semi-standard if the  entries in each row of  $S$ are non decreasing.
    $S$ is semistandard if $\lambda$ is a partition and $S$ is row semistandard and the entries
    in each column of $S$ are strictly increasing.
    Let $\mathcal T(\lambda,\mu)$ be the set of all semistandard $\lambda$-tableaux of type $\mu$.
 For any $\t\in\Std(\lambda)$, let $\mu(\t) $ be the  $\lambda$-tableau obtained from $\t$ by replacing each entry $i$ in $\t$ with $k$ if $i$ appears in the $k$-th row of $\t^{\mu}$.
 By~\cite[Example~4.2(ii)]{Ma},  $\mu(\t)$ is a row semistandard $\lambda$-tableau of type $\mu$.
Following \cite[Chapetr~4, \S2]{Ma}, write \begin{equation}\label{sst} x_{S,\t}=\sum_{\s\in\Std(\lambda), \mu(\s)=S}q^{\ell(d(\s))}x_{\s,\t}\end{equation}
where  $(S, \t)\in \mathcal T(\lambda,\mu)\times\Std(\lambda)$.

\begin{Lemma}\label{matsho}\cite[Corollary~4.10]{Ma} For any $\mu\in \Lambda(l-2f)$, the right $\tilde H_{l-2f}$-module $M^\mu:=x_\mu \tilde H_{l-2f}$  has basis $\{x_{S,\t}\mid S\in \mathcal T(\lambda,\mu),\t\in\Std(\lambda), \lambda\in\Lambda^+(l-2f) \}$. Arrange all such semi-standard tableaux as $S_1, S_2, \cdots, S_k$ such that $S_i\in  T(\lambda_i,\mu)$ and $i>j$ whenever $\lambda_i\rhd \lambda_j$. Let $M_i$ be the $\Bbbk$-submodule of $M^\mu$ spanned by
 $ \{x_{S_j,\t}\mid  \t\in \text{Std}(\lambda_j), i\le j\le k\}$. Then
 $M^\mu$ has a filtration $$M^\mu=M_1\supset M_2\supset \ldots\supset M_k\supset 0$$ such that $M_i/M_{i+1}\cong S^{\lambda_i}$. Further, the required isomorphism sends   $ x_{S_i,\t}+M_{i+1} $ to $x_{\lambda_i}T_{d(\t)} +\tilde H^{\rhd \lambda_i}_{l-2f} $.
\end{Lemma}

\subsection{A standard filtration of a standard module for $\mathcal B_{q,l}$} In this subsection, we prove that $\mathcal B_{q,l}$ is a standardly based algebra in the sense of \cite{DR}. We give a filtration of a standard module, which will  be used to construct its Jucys-Murphy basis.

Recall  $\Sigma^+(l)$ in \eqref{sig123}. There is a partial order $\unlhd$ on  $ \Sigma^+(l)$  such that $ \mu\unlhd \lambda $ for $\mu, \lambda\in \Sigma^+(l)$ if $|\mu|>|\lambda|$ or $|\mu|=|\lambda|$ and $\sum_{j=1}^k \mu_j\leq \sum_{j=1}^k\lambda_j$ for all possible  $k$.
For any $  w,v\in D_{f,l}, \s,\t\in \Std(\lambda)$ and $\lambda\in \Lambda^+(l-2f)$, define
 \begin{equation}\label{cst} C_{(w,\s),(v,\t)}^\lambda= \sigma (T_{w})E^f x_{\s,\t} T_v,\end{equation}
 where  $\sigma$ is the anti-involution of  $\mathcal B_{q,l}$ in Lemma~\ref{antiin} and
 $x_{\s,\t} $'s are given in Theorem~\ref{mur}. In Theorem~\ref{sstan}, $\Bbbk$ can be an arbitrary field with characteristic not $2$.

\begin{Theorem}\label{sstan}  Suppose  $\mathcal B_{q,l}$  is the periplectic $q$-Brauer algebra over $\Bbbk$. For any $\lambda\in \Lambda^+(l-2f)$, define  $I(\lambda)= D_{f,l} \times \Std(\lambda)$.
\begin{itemize}\item[(1)]  $\mathcal B_{q,l}$ has $\Bbbk$-basis
$S=\{ C_{(w,\s),(v,\t)}^\lambda \mid  (w, \s), (v, \t)\in I(\lambda), \lambda\in \Lambda^+(l-2f), 0\leq f\leq \lfloor l/2\rfloor \}$.
\item [(2)] For any $a\in  \mathcal B_{q,l}$, we have
$$\begin{aligned} a C_{(w,\s),(v,\t)}^\lambda & \equiv \sum_{(w_1, \s_1)\in I(\lambda)} r_{(w_1,\s_1),\lambda}(a, (w, \s)) C_{(w_1,\s_1),(v,\t)}^\lambda \pmod{\mathcal B_{q,l}^{\rhd \lambda}}\cr
 C_{(w,\s),(v,\t)}^\lambda a& \equiv\sum_{(v_1, \t_1)\in I(\lambda)} r_{\lambda, (v_1,\t_1)}( (v, \t), a ) C_{(w,\s),(v_1,\t_1)}^\lambda \pmod{\mathcal B_{q,l}^{\rhd \lambda}}\cr
\end{aligned} $$
where $\mathcal B_{q,l}^{\rhd \lambda}$ is the free $\Bbbk$-submodule generated by
 all $C_{(y,\u),(z,\u_1)}^\mu$'s such that $ (y, \u), (z, \u_1)\in I(\mu)$ and $\mu\rhd\lambda$.
Further, the coefficient $r_{(w_1,\s_1),\lambda}(a, (w, \s))$ (resp., $r_{\lambda, (v_1,\t_1)}( (v, \t), a )$ is independent of $(v, \t)$ (resp., $(w, \s)$).
\end{itemize}
\end{Theorem}
\begin{proof} In fact, (1) follows immediately from Theorems~\ref{thm123} and ~\ref{mur} and (2) follows from  Lemma~\ref{fque1} and Theorem~\ref{mur}.\end{proof}

The set  $S$  is  a standard basis of  $\mathcal B_{q,l}$  in the sense of \cite[(1.2.1)]{DR}. It    may not be a cellular basis of $\mathcal B_{q,l}$ in the sense of \cite{GL} since we have  not found an anti-involution of $\mathcal B_{q,l}$ sending $C_{(w,\s),(v,\t)}^\lambda$ to $C_{(v,\t),(w,\s)}^\lambda$.

 \begin{Cor}Let  $\mathcal B_{q, l}$ be  the periplectic Brauer algebra over an algebraically closed field of characteristic not $2$. Then  $\mathcal B_{q, l}$ is a quasi-hereditary algebra if and only if $e>l$ and $l$ is odd.\end{Cor}
\begin{proof} It is proved in \cite[Theorem~3.2.1]{DR} that a  standardly based algebra over a field is quasi-hereditary if and only if it is ''full'' in the sense of \cite[Definition~1.3.1]{DR}.  By Theorem~\ref{sstan}, $\mathcal B_{q, l}$   is quasi-hereditary if and only if the simple $\mathcal B_{q, l}$-modules are parameterized by $\Sigma^+(l)$. Now, the result follows from Theorem~\ref{simsi}.\end{proof}

 For any $\lambda\in \Sigma^+(l)$, let $C(\lambda)$ be the (right) standard module of $\mathcal B_{q,l}$ with respect to the standard basis in  Theorem~\ref{sstan}.  Thanks  to \cite[Definition~2.1.2]{DR},  up to an isomorphism, $C(\lambda)$  can be considered as  the free $\Bbbk$-module with basis
$\{E^f x_\lambda T_{d(\t)}T_v +\mathcal B_{q, l}^{\rhd \lambda}\mid  (v, \t)\in I(\lambda)\} $.
The following result establishes a relationship between $\Delta(S^\lambda)$ in $A$-lfdmod and  $C(\lambda)$.

\begin{Lemma}\label{iso321}  Suppose $\Bbbk$ is an algebraically closed field with characteristic not $2$ and $\lambda\in \Sigma^+(l)$. \begin{itemize} \item [(1)] As right $\mathcal B_{q,l}$-modules, $ C(\lambda)\cong \Delta(S^\lambda)1_l$.
\item[(2)] As right $H_l$-modules, $\Delta(S^\lambda)1_l\cong \text{Ind}_{H_{2f}\otimes \tilde H_{l-2f}}^{H_l} \Hom_{\mathcal B}(2f, 0) \boxtimes S^\lambda$.\end{itemize}
\end{Lemma}
\begin{proof} Thanks to Theorem~\ref{tria} and \eqref{func},
 $\Delta(S^\lambda)1_l$ has   basis  \begin{equation}\label{basis322} \{x_\lambda T_{d(\t)}\otimes \tilde E^f T_v \mid \t\in \Std(\lambda), v\in D_{f,l}\}\end{equation}
where $\tilde E^f = \lcap^{\otimes f}\otimes 1_{l-2f}$. Obviously, there is a $\Bbbk$-linear isomorphism $\phi:   C(\lambda)\rightarrow  \Delta(S^\lambda)1_l$ such that
 \begin{equation} \label{zzz1} \phi (E^f x_\lambda T_{d(
\t)} T_v +\mathcal B_{q,l}^{\rhd \lambda})=x_\lambda T_{d(\t)}\otimes \tilde E^f T_v.\end{equation}
 Since  $\tilde E^f=a_f E^f$, where $a_f=1_1\otimes \lcap^{\otimes f}\otimes 1_{l-2f-1}$, by Theorem~\ref{thm123}(2)  and Corollary~\ref{jdijss1}, both
$E^f \mathcal B_{q, l}$ and $\tilde E^f \mathcal B_{q, l}$ are left $\mathcal B_{l}(f)$-module spanned by
$\{E^f T_d\mid d\in D_{f,l}\}$ and $\{\tilde E^f T_d\mid d\in D_{ f,l}\}$, respectively.
Recall the two-sided ideal $\mathcal B_{q, l}^{f+1}$ in \eqref{2sidei}.
If we consider $E^f \mathcal B_{q, l}$ (resp., $\tilde E^f \mathcal B_{q, l}$) up to $\mathcal B_{q, l}^{f+1}$ (resp., $\tilde I_{f+1} $, the subspace of $1_{l-2f}A1_l$ spanned by tangle diagrams with at least $f+1$ cups), we can use Hecke algebra $\tilde H_{l-2f}$ to replace $\mathcal B_{l}(f)$. Consequently,   the $\Bbbk$-linear isomorphism $\phi$ in \eqref{zzz1} is a right $\mathcal B_{q, l}$-homomorphism, proving (1).

Obviously, $\text{Ind}_{H_{2f}\otimes \tilde H_{l-2f}}^{H_l} \Hom_{\mathcal B}(2f, 0) \boxtimes S^\lambda$ has basis $$\{ ( \lcap^{\otimes f}T_{d_1}\otimes x_\t)\otimes T_{d_2}\mid (\t, d_1, d_2)\in \Std(\lambda)\times D_{f,2f}\times \mathcal D_{2f,l-2f} \}$$
 where $ \mathcal D_{2f, l-2f}$ is the distinguished right coset representatives of   $\mathfrak S_{2f}\times \mathfrak S_{l-2f}$ in $\mathfrak S_l$. There is a bijection between $D_{f, l}$ and
 $D_{f,2f}\times  \mathcal D_{2f,l-2f}$ such that any $d\in D_{f, l}$ can  uniquely be written as $d_1d_2$ such that $(d_1, d_2)\in D_{f,2f}\times  \mathcal D_{2f,l-2f} $. Consequently, there is a $\Bbbk$-linear
 isomorphism $\phi:  \Delta(S^\lambda)1_l\rightarrow   \text{Ind}_{H_{2f}\otimes \tilde H_{l-2f}}^{H_l} \Hom_{\mathcal B}(2f, 0) \boxtimes S^\lambda$  such that
 \begin{equation}\label{rrr321} \phi(x_\t\otimes \tilde E^f T_d)=  (\lcap^{\otimes f}T_{d_1} \otimes x_\t)\otimes T_{d_2}\end{equation} where
 $d=d_1d_2$ and $(d_1, d_2)\in D_{f,2f}\times  \mathcal D_{2f,l-2f} $. Finally, it follows from Corollary~\ref{jdijss1} that  $\phi$ is a right $H_l$-homomorphism, proving (2).
 \end{proof}

\begin{Cor}\label{ijres2} Suppose  $\lambda\in\Lambda^+(l-2)$ and $e>l$.
As right $H_l$-modules, $C(\lambda)\cong \oplus _{\mu} S^\mu $
where $\mu\in \Lambda^+(l)$ obtained from  $\lambda$ by  adding two boxes which are  not in the same row.
\end{Cor}
\begin{proof}Suppose $f=1$. Thanks to Lemma~\ref{iso321}, we have right $H_l$-isomorphism
\begin{equation} \label{LRC} C(\lambda)\cong \text{Ind}_{H_{2}\otimes \tilde H_{l-2}}^{H_l} \Hom_{\mathcal B}(2,0) \boxtimes S^{\lambda}.\end{equation}  However, by Lemma~\ref{basicrel}(3) and Definition~\ref{defn:qBrauer}(R2),  $\Hom_{\mathcal B}(2,0)\cong S^{(1,1)}$ as right  $H_2$-modules.  Now the result follows from Littlewood-Richardson rule for semi-simple Hecke algebras. In this case, all Littlewood-Richardson coefficients appearing on the decomposition of RHS of \eqref{LRC}  is $1$.
\end{proof}

 We are going to construct  a $\mathcal B_{q,l-1}$-filtration of $C(\lambda)$. For any  $\lambda\in \Sigma^+(l)$,
let \begin{equation}\label{arrowlambda}
\mathcal{RA}(\lambda)=  \Sigma^+(l-1)\cap (\mathcal A_\lambda\cup \mathcal R_\lambda).\end{equation}

%We write $\mu\rightarrow \lambda$ if $\mu\in \mathcal{RA}(\lambda)$.

\begin{Defn}\label{deofymulambda}
 Suppose $\lambda\in\Lambda^+(l-2f)$ and $\mu\in \mathcal{RA}(\lambda)$.
\begin{itemize}
\item [(1)] If $\mu=\lambda\setminus p $ and $p=(k,\lambda_k)$,
let  $y_{\mu}^\lambda= E^fx_\lambda T_{a_k,l}$, where $a_k=2f+\sum_{i=1}^k\lambda_i$.
\item [(2)] If $\mu=\lambda\cup p $   and $p=(k,\lambda_k+1)$, let $y_{\mu}^\lambda=E_{2f-1}T_{l,2f}^{-1}T_{b_k,2f-1}^{-1}E^{f-1}x_{\mu} $, where $b_k=2f-1+\sum_{i=1}^k\lambda_i$.
\end{itemize}
\end{Defn}
 Definition~\ref{deofymulambda}(2) happens only if $f>0$.
Thanks to Definition~\ref{defn:qBrauer},  we rewrite  $y_{\mu}^\lambda$ as follows:
\begin{equation}\label{symulambda}
y_{\mu}^\lambda=\begin{cases}
                        \sum_{i=a_{k-1}+1}^{a_k}q^{a_k-i}T_{i,l}E^f x_\mu & \hbox{if $\mu=\lambda\setminus p$,} \\
                        E^fx_\lambda T_{l,2f}^{-1}T_{b_k,2f-1}^{-1}\sum _{i=b_{k-1}+1}^{b_k}q^{b_k-i}T_{b_k,i} & \hbox{if $\mu=\lambda\cup p $.}\\
                        \end{cases}
\end{equation}

For any $\lambda\in\Lambda^+(l-2f)$  , both $\mathcal R_\lambda$ and $\mathcal A_\lambda$ are finite sets.
There are two positive integers $a$ and $m$ such that  $\mathcal R_\lambda=\{\mu^{(1)},\mu^{(2)},\ldots, \mu^{(a)}\}$, $\mathcal A _\lambda=\{\mu^{(a+1)},\mu^{(a+2)},\ldots, \mu^{(m)}\}$ and
$$\mathcal{RA}(\lambda)= \left\{
                           \begin{array}{ll}
                             \mathcal R_\lambda & \hbox{ if $f=0$;} \\
                            \mathcal R_\lambda \cup\mathcal A_\lambda & \hbox{ if $f>0$.}
                           \end{array}
                         \right.
$$
Moreover, we can arrange elements in $\mathcal{RA}(\lambda)$ such that
 \begin{equation}\label{ppp11} \mu^{(1)} \rhd  \mu^{(2)}\rhd \ldots \rhd \mu^{(a)}\rhd\mu^{(a+1)}\rhd  \mu^{(a+2)}\rhd \ldots \rhd \mu^{(m)}
 \end{equation}
with respect the partial order $\rhd$ on $\Sigma^+(l-1)$.
For $1\leq k\leq m$, define \begin{equation}\label{nkk} N_k=\sum_{j=1}^ky_{\mu^{(j)}}^\lambda \mathcal B_{q,l-1} + \mathcal B_{q,l}^{\rhd\lambda}\end{equation}
 where $y_{\mu^{(j)}}^\lambda$'s are given in \eqref{symulambda}. Then   $N_k$  is  a  right $\mathcal B_{q,l-1}$-submodule of $C(\lambda)$. When $k>a$, $N_k$  is  defined only if $f>0$.

\begin{Lemma}\label{fiels1}As $\mathcal B_{q,l-1}$-modules,  $N_k/N_{k-1}\cong C(\mu^{(k)}) $,  $1\leq k\leq a$.\end{Lemma}

\begin{proof} %We compute $y_{\mu^{(k)}}^\lambda \mathcal B_{q,l-1}$.
Thanks to  Corollary~\ref{jdijss1} and Theorem~\ref{hecbr},   any element in  $y_{\mu^{(k)}}^\lambda \mathcal B_{q,l-1}+N_{k-1}$ can be written as a linear combination of elements in \begin{equation}\label{bbasis} M_k:=\{y_{\mu^{(k)}}^\lambda T_{d(\u)}T_w+N_{k-1}\mid  \u\in\Std(\mu^{(k)}), w\in D_{f,l-1}\}.\end{equation} Further, $M_k$ is $\Bbbk$-linear independent since $y_{\mu^{(k)}}^\lambda T_{d(\u)}T_w+ \mathcal B_{q,l}^{\rhd \lambda}=E^fx_\lambda T_{d(\t)}T_w+\mathcal B_{q,l}^{\rhd \lambda}$ with
 $\t\downarrow_{l-1}=\u$, a basis element of $C(\lambda)$. So $N_k/N_{k-1}$ is  free over $\Bbbk$ with basis $M_k$. We have   a well-defined  $\Bbbk$-linear isomorphism  $\phi: N_k/N_{k-1}\rightarrow  C(\mu^{(k)}) $ such that
$$\phi(y_{\mu^{(k)}}^\lambda T_{d(\u)}T_w+N_{k-1})= E^f x_{\mu^{(k)}}T_{d(\u)}T_w + \mathcal B_{q,l-1}^{\rhd\mu^{(k)}},$$  for all  $\u\in\text{Std}(\mu^{(k)})$ and $w\in D_{f,l-1}$. Using Corollary~\ref{jdijss1}  and Theorem~\ref{hecbr} again, we see that $\phi$ is a right $\mathcal B_{q,l-1}$-homomorphism and hence $\phi$ is a  $\mathcal B_{q,l-1}$-isomorphism
\end{proof}
If $\lambda\in\Lambda^{+}(l)$ (i.e., $f=0$), then $C(\lambda)$ has  a standard  filtration  in Lemma~\ref{fiels1}.
From here to the end of this section, we assume  $\lambda\in \Lambda^{+}(l-2f)$ with $f>0$ and  deal with $N_k/N_{k-1}$, $a+1\leq k\leq m$.
The following result has already been given  in the proof of \cite[Corollary 5.4, Lemma~5.5]{Eng}.

\begin{Lemma}\label{tablsji}
Suppose that $\mu\in\Lambda^+(l-2f+1)$ and $\mu\rhd \mu^{(m)}$.
\begin{itemize}
\item[(1)]Suppose  $\mu\not\in \{\mu^{(a+1)}, \ldots, \mu^{(m)}\}$. Then  $\text{shape}(\s\downarrow_{l-2})\rhd \lambda$ if $(\s,\mu^{(m)}(\s)) \in \Std(\mu) \times\mathcal T(\mu,\mu^{(m)} )$.
\item[(2)] Suppose  $\mu=\mu^{(j)}=\lambda\cup (k,\lambda_k+1)$ for some  $a+1\leq j\leq m$. Let  $\s_j=\t^{\mu^{(j)}}s_{b_j,l-1}$, where $b_j=2f-1+\sum_{i=1}^k\lambda_i$.  Then   $\s_j$ is the unique standard $\mu^{(j)}$-tableau $\s$ satisfying
    $\mu^{(m)}(\s)\in \mathcal T(\mu^{(j)},\mu^{(m)} )$. In this case, $\text{shape}(\s\downarrow_{l-2})=\lambda$.
\end{itemize}
\end{Lemma}

\begin{Lemma}\label{jaijai} Suppose $(\s,\mu^{(m)}(\s)) \in \Std(\mu)\times \mathcal T(\mu,\mu^{(m)} )$. Then $E^f T_{l,2f}^{-1}T_{l-1,2f-1}^{-1} T^*_{d(\s)}x_{\mu} \in  \mathcal B_{q,l}^{ \rhd \lambda}$
whenever  $\text{shape}(\s\downarrow_{l-2})\rhd \lambda$.
\end{Lemma}
\begin{proof} The arguments in the proof \cite[Lemma~5.3]{Eng}  depends only on  the braid relations for the  Hecke algebras. Therefore, one can imitate arguments there to verify our current result. We leave details to the reader. \end{proof}

\begin{Lemma}\label{spans}As $\mathcal B_{q, l-1}$-modules, $C(\lambda)=N_m$ where $N_m$ is given in \eqref{nkk}.
\end{Lemma}
\begin{proof}
Recall that the standard module  $C(\lambda)$ has basis $\{E^f x_\lambda T_{d(\t)}T_v +\mathcal B_{q, l}^{\rhd \lambda}\mid  (v, \t)\in I(\lambda)\} $.
If $v\in D_{f,l-1}$, then \begin{equation}\label{key123}  E^f x_\lambda T_{d(\t)}T_v+\mathcal B_{q, l}^{\rhd \lambda}\in E^f x_\lambda T_{d(\t)}\mathcal B_{q,l-1} +\mathcal B_{q, l}^{\rhd \lambda}\subset N_a.\end{equation} Otherwise  $T_v=T_{2f,l}b$ for some $b\in H_{l-1}$. So,
\begin{equation}\label{key12} E^f x_\lambda T_{d(\t)}T_v+\mathcal B_{q, l}^{\rhd \lambda}\in E^f x_\lambda T_{2f,l}  \mathcal B_{q,l-1} +\mathcal B_{q, l}^{\rhd \lambda}
\subset\sum_{j=2f+1}^l  E^f x_\lambda T_{j, l} \mathcal B_{q,l-1}+ E^f x_\lambda T_{l,2f}^{-1}\mathcal B_{q,l-1}+\mathcal B_{q, l}^{\rhd \lambda}
\end{equation} By \eqref{nkk}, the last term in the RHS of  \eqref{key12} is  in $N_m$. Thanks to \eqref{key123}, the summation in the RHS of  \eqref{key12} is in $N_a\subset N_m$. This implies  $C(\lambda)\subset N_m$ and hence $C(\lambda)= N_m$ as required.
\end{proof}

\begin{Lemma}\label{fiels2} As $\mathcal B_{q,l-1}$-modules, $N_k/N_{k-1}\cong  C(\mu^{(k)}) $ for  any  $a+1\leq k\leq m$.\end{Lemma}

\begin{proof} First, we find a set which spans $N_k/N_{k-1}$ as $\Bbbk$-module.  Suppose  $h\in \mathcal B_{q,l-1}$. Thanks to Corollary~\ref{jdijss1},   $E^{f-1}  h+ \mathcal B_{q,l}^{f+1}$ can be written as a linear combination of elements  $ h_d E^{f-1}T_d+ \mathcal B_{q,l}^{f+1}$ and $b+\mathcal B_{q,l}^{f+1}$, where  $b\in E^{f-1}\mathcal B_{q,l-1}\cap \mathcal B_{q,l}^f$, $ d\in D_{f-1,l-1}$ and $ h_d\in \tilde H_{l-2f+1}$, the Hecke algebra associated to the symmetric group on $\{2f-1, 2f, \ldots, l-1\}$. Hence
$$y_{\mu^{(k)}}^\lambda h \equiv \sum_{ d\in D_{f-1,l-1}}y_{\mu^{(k)}}^\lambda h_dT_d + E_{2f-1}T_{l,2f}^{-1}T_{b_k,2f-1}^{-1} x_{\mu^{(k)}}b ~(\text{mod }  \mathcal B_{q,l}^{f+1}). $$
By  arguments similar to those  for the proof of \cite[Claim~5.7]{Eng} we have   $$ E_{2f-1}T_{l,2f}^{-1}T_{b_k,2f-1}^{-1} x_{\mu^{(k)}} b+ \mathcal B_{q,l}^{\rhd\lambda}\in N_a$$
for any $b\in E^{f-1}\mathcal B_{q,l-1}\cap \mathcal B_{q,l}^f $.  Write $x=T_{l,2f}^{-1}T_{b_k,2f-1}^{-1} x_{\mu^{(k)}}$. For $a+1\leq k\leq m$,
\begin{equation}\label{sss123}
y_{\mu^{(k)}}^\lambda h
%\equiv\sum_{d\in D_{f-1,l-1}}y_{\mu^{(k)}}^\lambda h_dT_d ~(\text{mod }  N_a)
\equiv
\sum_{ d\in D_{f-1,l-1}}E^fx h_dT_d~ \equiv \sum_{ d\in D_{f-1,l-1}}E^f T_{l,2f}^{-1}T_{l-1,2f-1}^{-1} x_{\s_k,\t^{\mu^{(k)}}} h_dT_d~\pmod {  N_a}\end{equation}
where  $\s_k\in \Std(\mu^{(k)})$  is given in Lemma~\ref{tablsji}(2).
Let  $S_k=\mu^{(m)}(\s_k)$. By Lemma~\ref{tablsji}(2), \begin{equation}\label{Pndjhdh} x_{\s_k,\u}=q^{-\ell(d(\s_k))}x_{S_k,\u }
\end{equation}
for any $\u\in \Std(\mu^{(k)})$. Rewriting  RHS of \eqref{sss123} via  \eqref{Pndjhdh} and
  Lemma~\ref{matsho}, we have
 $$
y_{\mu^{(k)}}^\lambda h
\equiv q^{-\ell(d(\s_k))}\sum_{ d\in D_{f-1,l-1}}E^f T_{l,2f}^{-1}T_{l-1,2f-1}^{-1}
\sum_{(\t, S)\in\text{Std}(\mu)\times\mathcal T(\mu, \mu^{(m)}), \mu\unrhd\mu^{(k)} }a_{S,\s} x_{S,\s} T_d
\pmod {  N_a}.$$
If $\mu\neq \mu^{(j)}$, by   Lemma~\ref{tablsji}(1) and
   Lemma~\ref{jaijai}, the corresponding terms on the RHS of the above equality  is in $\mathcal B_{q, l}^{\rhd \lambda}$. If $\mu=\mu^{(j)}$,  by Lemma~\ref{tablsji}(2),  $S= S_j$. Using \eqref{Pndjhdh} again, we see that   $N_k/N_{k-1}$ is $\Bbbk$-module spanned by
\begin{equation}\label{mkspandedd}
M_k =\{ y_{\mu^{(k)}}^\lambda T_{d(\t)} T_d+N_{k-1}\mid  \t\in \Std(\mu^{(k)}), d\in D_{f-1,l-1}\}.\end{equation}
So, $\dim N_k/N_{k-1}\leq \dim C(\mu^{(k)})=|M_k|$.
By the branching rule for the cell module of BMW algebras in \cite[Theorem~2.3]{Wen}, we have $\dim C(\lambda)=\sum _{j=1}^m \dim C(\mu^{(j)})$. By  Lemma~\ref{spans},
 $\dim N_m=\dim C(\lambda)$, forcing  $\dim N_k/N_{k-1}=\dim C(\mu^{(k)})=|M_k|$.
In particular,  $M_k$ is a basis of $N_k/N_{k-1}$ for all $1\le k\le m$.  We have a well-defined $\Bbbk$-linear isomorphism $\phi: N_k/N_{k-1}\rightarrow C(\mu^{(k)})$ for any $a+1\le k\le m$ such that
\begin{equation}\label{pppqqq} \phi(y_{\mu^{(k)}}^\lambda T_{d(\t)} T_d+N_{k-1})=E^f x_{\mu^{(k)}}T_{d(\t)} T_d+\mathcal B_{q, l-1}^{\rhd \mu^{(k)}}\end{equation}
for any  $\t\in \Std(\mu^{(k)})$ and $ d\in D_{f-1,l-1}$.  By arguments similar to those above (more explicitly, using Corollary~\ref{jdijss1} and Lemma~\ref{matsho})
   we see that  the $\Bbbk$-linear isomorphism  $\phi$ in \eqref{pppqqq}  is a right  $\mathcal B_{q,l-1}$-homomorphism.
\end{proof}

\begin{Theorem}\label{Burnch}Suppose $\lambda\in \Sigma^+(l)$.
The standard module $C(\lambda)$ has a $\mathcal B_{q,l-1}$-filtration $0\subset N_1\subset \ldots \subset N_m$ such that $N_k/N_{k-1}\cong C(\mu^{(k)})$ for $1\leq k\leq m$.
\end{Theorem}
\begin{proof}The result follows immediately  from  Lemmas~\ref{fiels1} and \ref{fiels2}.\end{proof}

\section{Jucys-Murphy elements and Jucys-Murphy basis}

\subsection{The Jucys-Murphy basis of any standard module for $ \mathcal B_{q,l}$}
Suppose $\lambda\in \Sigma^+(l)$. Write $\mu\rightarrow \lambda$ if $\mu\in \mathcal {RA}(\lambda)$ (see \eqref{arrowlambda}).
 An up-down tableau $\mathbf t$  of type $\lambda$ is a sequence of partitions $\mathbf t=(\mathbf t_0, \mathbf t_1,\ldots,\mathbf t_l)$ such that
$\mathbf t_0=\emptyset$, $\mathbf t_i\rightarrow\mathbf t_{i+1}$, $0\leq i\leq l-1$  and  $\mathbf t_l=\lambda$.
Let $\mathscr T_l^{ud}(\lambda)$ be the set of up-down tableaux of type $\lambda$.
There is a partial order $\lhd$ on $\mathscr T_l^{up}(\lambda)$ such that   $\mathbf t\lhd\mathbf s $ if $\mathbf t_k\lhd\mathbf s_k$ for some $k$  and $\mathbf t_j=\mathbf s_j$ for all $k<j\leq l$.

\begin{Defn}\label{murr} Suppose $\mathbf t\in \mathscr T_l^{up}(\lambda)$ and  $\lambda\in \Lambda^+(l-2f)$. Let $\mathrm m_{\mathbf t_1}=1$. If $\mathrm m_{\t_{l-1}}$ is available, we define $\mathrm m_{\mathbf t_l}$ inductively as follows:
\begin{itemize}\item [(1)]
$ \mathrm m_{\mathbf t_l}= \sum_{j=a_{k-1}+1}^{a_k}T_{j,i}\mathrm m_{\mathbf t_{l-1}} $ if $ \lambda=\mu\cup p$ with $p=(k,\mu_k+1)$, and $a_i=2f+\sum_{j=1}^i \lambda_j$.
\item [(2)] $ \mathrm m_{\mathbf t_l}= E_{2f-1}T_{l,2f}^{-1}T_{b_k,2f-1}^{-1}\mathrm m_{\mathbf t_{l-1}} $ if $\lambda=\mu\setminus p$ with $p=(k,\mu_k)$, and $b_i=2f-1+\sum_{j=1}^i \lambda_j$.
\end{itemize}
\end{Defn}
Later on, we denote $\mathrm m_{\t_l}$ by $\mathrm m_\t$.
Thanks to Definition~\ref{murr},  $ \mathrm m_{\mathbf t}= E^f x_\lambda b_{\mathbf t}$ where $b_{\mathbf t}=b_{\mathbf t_l}$ which can be defined inductively as
\begin{equation}\label{bbb123}  b_{\mathbf t_l} =\begin{cases}
                        T_{a_k,l}b_{\mathbf t_{l-1}} & \hbox{ if  $ \lambda=\mu\cup p$ with $p=(k,\mu_k+1)$, } \\
                        T^{-1}_{l,2f}\sum_{j=b_{k-1}+1}^{b_k}q^{b_k-j}T^{-1}_{j,2f-1} b_{\mathbf t_{l-1}} , & \hbox{ if $\lambda=\mu\setminus p$ with $p=(k,\mu_k)$.}\\
                      \end{cases}
\end{equation}

\begin{Theorem}\label{murphyb}
Suppose that $\lambda\in \Lambda^+(l-2f)$.
\begin{itemize}
\item[(1)] $\{\mathrm m_{\mathbf t}+\mathcal B_{q,l}^{\rhd \lambda}\mid \mathbf t\in \mathscr T_l^{up}(\lambda) \}$ is a $\Bbbk$-basis of $C(\lambda)$.
\item[(2)] For any  $1\leq j\leq m$, let $M_{j}$ be the $\Bbbk$-submodule of $C(\lambda)$ spanned by $$\{ \mathrm m_{\mathbf t} + \mathcal B_{q,l}^{\rhd \lambda} \mid \mathbf t\in \mathscr T_l^{up}(\lambda), \mathbf t_{l-1}\unrhd \mu^{(j)}\}$$ where  $\mu^{(j)}$'s are given  in \eqref{ppp11}. Then $M_j$ is a $\mathcal B_{q,l-1}$-submodule of $C(\lambda)$.
\item[(3)] As $\mathcal B_{q,l-1}$-submodules, $M_j/M_{j-1} \cong  C(\mu^{(j)})$. The required isomorphism  sends $ \mathrm m_{\mathbf t}+M_{j-1}$ to $ \mathrm m_{\mathbf t\downarrow_{l-1}}+ \mathcal B_{q,l-1}^{\rhd \mu^{(j)}}$, where $\mathbf t\downarrow_{l-1}=(\mathbf t_0,\mathbf t_1,\ldots,\mathbf t_{l-1})\in  \mathscr T_{l-1}^{ud}(\mu^{(j)})$.
\item[(4)] $M_j=N_j$ for $1\leq j\leq m$, where $N_j$'s are given in \eqref{nkk}.
 \end{itemize}
\end{Theorem}\begin{proof}Using the explicit isomorphisms in the proof of Lemma~\ref{fiels1} and Lemma~\ref{fiels2} (i.e. the classical branching rule), we immediately have that (e.g., \cite[Theorem~5.9]{Eng} for BMW algebras) $C(\lambda)$ has a basis
$$\{y_{\mu}^\lambda b_\mathbf u+ \mathcal B_{q,l}^{\rhd \lambda}\mid \mathbf u\in \mathscr T_{l-1}^{up}(\mu), \mu \rightarrow\lambda\}$$
 if $C(\mu)$ has a basis $\{E^k x_\mu b_\mathbf u+ \mathcal B_{q,l-1}^{\rhd \mu}\mid \mathbf u\in\mathscr T_{l-1}^{up}(\mu) \}$ for any $\mu\rightarrow\lambda$, where $k$ is the integer such that $\mu\in \Lambda^+(l-1-2k)$. Moreover,
$$y_{\mu^{(j)}}^\lambda b_\mathbf u+M_{j-1}\mapsto  E^k x_{\mu^{(j)}} b_\mathbf u+ \mathcal B_{q,l-1}^{\rhd \mu^{(j)}}$$
determines an isomorphism $M_j/M_{j-1}\cong C(\mu^{(j)})$ of $\mathcal B_{q,l-1}$-modules. Note that $\mathrm m_{\mathbf u}=E^k x_\mu b_\mathbf u$ for $\mathbf u\in \mathscr T_{l-1}^{ud}(\mu)$.
Now (1)--(3) follow from induction on $l$ since $ y_{\mu}^\lambda b_\mathbf u=\mathrm m_\mathbf t$ with $\mathbf t\downarrow_{l-1}=\mathbf u$ by Definition~\ref{murr}.

By  Definition~\ref{murr} we see that $ \mathrm m_\mathbf t+\mathcal B_{q,l}^{\rhd \lambda}\in N_j$ if $\mathbf t_{l-1}=\mu^{(k)}$ for $k\leq j$. So,
$ M_j\subset N_j$ for any $1\leq j\leq m$. Then (4) follows by comparing dimensions.\end{proof}

The basis in Theorem~\ref{murphyb}(1) is called the \textsf{Jucys-Murphy basis} of $C(\lambda)$. we can  give branching rule for standard modules in the category of left $\mathcal B_{q, l}$-modules and hence to give the Jucys-Murphy basis for $\mathcal B_{q, l}$. We will not give details since  we do not need them  in this paper,

%We need Jucys-Murphy basis for  if we compute the Gram determinant associated to the invariant form on standard  modules. For this purpose,

\subsection{Jucys-Murphy elements of $\mathcal B_{q, l}$} For any positive integer $i, 1\le i\le l$ define
\begin{equation}\label{mure} x_i=\begin{cases} 0 &\text{if $i=1$,}\\ \sum_{j=1}^{i-1}(j,i)+q^{-1}\bar{(j,i)} & \text{if $2\leq i\leq l$,}\\
\end{cases}\end{equation}  where
$(j,i)=T_{j, i} T_{i-1, j}$, and $\bar {(j,i)}=T_{j,1}T_{i,2}E_1T_{2,i}T_{1,j}$.
By Definition~\ref{defn:qBrauer}, it is easy to verify the following equality: \begin{equation}\label{sjsssdj}
x_{i+1}=T_ix_iT_i+ T_i +q^{-1}\bar {(i,i+1)}.
\end{equation}
\begin{Lemma}\label{commuhsgxt}  For any $2\leq i\leq l$ and  $y\in \mathcal B_{q,i-1}$, $yx_{i}=x_iy$.
\end{Lemma}
\begin{proof} It is known that  $\sum_{j=1}^{i-1}(j,i) $ is a Jucys-Murphy element for the Hecke algebra $H_{i}$ satisfying $T_k \sum_{j=1}^{i-1}(j,i) =\sum_{j=1}^{i-1}(j,i) T_k$ for any $1\le k\le i-2$.  Further, by direct computation, we have
$$T_k(\bar {(k,i)}+\bar{(k+1,i)} )= (\bar {(k,i)}+\bar{(k+1,i)} )T_k \  \text{ and $T_k\bar{(j,i)}=\bar{(j,i)}T_k $}$$ if either $j<k$ or $j>k+1$. So $T_k  \sum_{j=1}^{i-1}q^{-1}\bar{(j,i)}= \sum_{j=1}^{i-1}q^{-1}\bar{(j,i)}T_k$, proving $T_kx_i=x_iT_k$.  When $i=1,2,3$, it is easy to verify  $E_1x_i=x_iE_1$. It is enough to assume  $i\ge 4$ when we prove $E_1x_i=x_iE_1$.  Thanks to  Lemma~\ref{key1}(5)-(6),  we have $E_1T_{3,1}T_{4,2}E_1 T_{2,4}T_{1,3}= -E_1E_3 T_{2,4}T_{1,3}=E_1E_3$,
and
$ T_{3,1}T_{4,2} E_1 T_{2,4}T_{1,3}  E_1= -T_{3,1} T_{4,2}  E_1E_3 =E_1E_3$.
So $E_1\bar {(3,4)}=\bar {(3,4)}E_1 $. By \eqref{sjsssdj},   $E_1x_4=x_4E_1$. In general, we have $\bar{ (i, i+1)}=T_{i-1}T_i \bar{(i-1,i) }T_i T_{i-1}$ and $E_1\bar{ (i, i+1)}=\bar{ (i, i+1)} E_1$ for $i\ge 4$.  Using   \eqref{sjsssdj} and induction assumption on $i-1$ yields  $E_1x_i=x_iE_1$ as required.
\end{proof}

By Lemma~\ref{commuhsgxt}, $\{x_i\mid 1\leq i\leq l\}$ generates an abelian subalgebra $L_l$ of $\mathcal B_{q,l}$. However, $\sum_{i=1}^l x_i$ may not be a central element of  $\mathcal B_{q,l}$ in general. For example $x_1+x_2$ is not central in $\mathcal B_{q,2}$.

\begin{Lemma}\label{actofxn} Suppose $2f+1\leq i\leq l-1$ and $1\leq j\leq f$.
\begin{itemize}
\item[(1)] $E^f \bar {(i,l)}\equiv 0 \pmod{ \mathcal B_{q, l}^{f+1} }$,
\item[(2)] $ E^f((2j,l)+(2j-1,l)+q^{-1}(\bar{(2j-1,l)}+\bar{(2j,l)}))=0$.
\end{itemize}
\end{Lemma}
\begin{proof}  We have $E^f \bar{(2f+1,l)}=T_{l,2f+2} E^f T_{2f+1,1}T_{2f+2,2}E_1 T_{2,l}T_{1,2f+1}$. By Lemma~\ref{key1}(1) \begin{equation}\label{eee123} T_{l,2f+2} E^f T_{2f+1,1}T_{2f+2,2}E_1 T_{2,l}T_{1,2f+1}= T_{l,2f+2} E^{f+1} T_{2f+1,1}T_{2f+2,2}T_{2,l}T_{1,2f+1}\in \mathcal B_{q, l}^{f+1}.\end{equation}
This proves (1) when $i=2f+1$. In general   (1) follows from \eqref{eee123} and  $ \bar{(i,l)}= T_{i,2f+1} \bar{(2f+1,l)} T_{2f+1,i}$.
We have $E_1(1,l)=-q^{-1}E_1(2,l)T_1$,  $E_1\bar{(1,l)}= -qE_1(2,l)$ and   $E_1\bar{(2,l)} =E_1(2,l)T_1$. Via them, it is easy to verify (2) when $f=1$.
In general, one can check that $E^f\bar {(2f-1,l)}=-qE^f(2f,l)$, $E^f \bar{(2f,l)}=E^f(2f,l)T_{2f-1}$ and $ E^{f}(2f-1,l)=-q^{-1}E^f(2f,l)T_{2f-1}$.
So $$ E^f((2f,l)+(2f-1,l)+q^{-1}(\bar{(2f-1,l)}+\bar{(2f,l)}))=0.$$ This prove (2) when $j=f$.
In particular, we have  $ E^j((2j,l)+(2j-1,l)+q^{-1}(\bar{(2j-1,l)}+\bar{(2j,l)}))=0$. This implies (2) since $E^j$ is a factor of $E^f$ for any $1\le j\le f$.
\end{proof}

%Recall $N_j$'s in \eqref{nkk} for $1\leq j\leq m$.
\begin{Lemma}\label{actofjme} Suppose $\lambda\in \Sigma^+(l)$ and $\mu^{(j)}=\lambda \setminus p_j$,
 $1\leq j\leq a$. Then $$y_{\mu^{(j)}}^\lambda x_l \equiv \frac{q^{2c(p_j)}-1}{q-q^{-1}} y_{\mu^{(j)}}^\lambda \pmod{ N_{j-1}}$$ where $N_{j-1}$ is given  in \eqref{nkk} and $c(p_j)=k-i$ if $p_j=(i, k)$.
\end{Lemma}
\begin{proof} By Lemma~\ref{actofxn},  $y_{\mu^{(j)}}^\lambda x_l\equiv y_{\mu^{(j)}}^\lambda \sum_{i=2f+1}^{l-1}(i,l) \pmod {\mathcal B_{q, l}^{f+1}}$. Since  $y_{\mu^{(j)}}^\lambda=E^f x_\lambda T_{a_j, l}$ (see Definition~\ref{deofymulambda}(1)),  by \cite[Theorem~3.22]{Ma}, we have the formula as required~\footnote{ The  Hecke algebra $H_l$ is defined via
$(T_i-q)(T_i+1)=0$ in \cite{Ma}.     Our $q^{-1} T_i$ (resp., $q^2$) is $T_i$ (resp., $q$) in \cite{Ma}. Consequently, our $q^{-1}x_i$ is $L_i$ in \cite{Ma}.}.
\end{proof}

\begin{Lemma}\label{cumopp} Suppose $\lambda\in\Lambda^+(l-2f)$,  $2f\leq i\leq l-1$ and  $2f<j\leq l-1$.
\begin{multicols}{2}
\item[(1)] $E^fx_\lambda  T_{l,2f}^{-1}(i,l)\in { N_a} $,
\item[(2)] $E^f x_\lambda  T_{l,2f}^{-1}((2f-1,l)+q^{-1})  \in N_a $,
\item[(3)] $E^f x_\lambda  T_{l,2f}^{-1}\bar{(2f-1,l)} =0$,
\item[(4)]$E^f x_\lambda  T_{l,2f}^{-1}(\bar{(j,l)}-q^{-1}(2f-1,j))\in  N_a $,
\end{multicols}
 where  $N_a$ is given in \eqref{nkk}.
\end{Lemma}
\begin{proof} Since we are assuming   $2f\le i\le l-1$,    $$E^fx_\lambda T_{l,2f}^{-1}(i,l)=E^f x_\lambda  T_{i+1,l}T_{i,2f}^{-1}\equiv 0 \pmod{ N_a},$$ proving (1). We have
 $$E^f x_\lambda  T_{l,2f}^{-1}(2f-1,l) =E^fx_\lambda T_{2f-1,l}=-q^{-1}E^f  x_\lambda T_{2f,l} \equiv -q^{-1} E^f x_\lambda  T_{l,2f}^{-1} \pmod {  N_a},$$ proving (2). Write $x= T_{l,2f}^{-1} T_{2f-1,1}T_{l,2}E_1T_{2,l}T_{1,2f-1}$. Then (3) follows since
$$
E^f x_\lambda T_{l,2f}^{-1}\bar{(2f-1,l)}=E^fx_\lambda  x =E^f x_\lambda T_{2f-1,1}T_{2f,2}E_1T_{2,l}T_{1,2f-1}\\
=(-1)^{f-1}x_\lambda E^fE_1T_{2,l}T_{1,2f-1}=0,
$$ where the third equality follows from
Lemma~\ref{key1}(5).
Suppose  $2f<j\leq l-1$, we have
$$\begin{aligned}
E^fx_\lambda  T_{l,2f}^{-1}\bar{(j,l)}&= E^f x_\lambda T_{l,2f}^{-1}T_{j,1}T_{l,2}E_1T_{2,l}T_{1,j}= E^f x_\lambda T_{j+1,2f}^{-1} T_{j,1}T_{j+1,2} E_1T_{2,l}T_{1,j}\\
&=E^f x_\lambda T_{j,2f}^{-1} T_{j,1}T_{j+1,2}T_1^{-1} E_1T_{2,l}T_{1,j}= q^{-1}E^f x_\lambda T_{j,2f}^{-1} T_{j,1}T_{j+1,2} E_1T_{2,l}T_{1,j} \\
&= q^{-1}E^f x_\lambda   T_{2f,1}T_{j+1,2} E_1 T_{2,l} T_{1,j}=-q^{-2}E^fx_\lambda  T_{2f-1,1}T_{j+1,2} E_1T_{2,l}T_{1,j} \\
&= (-1)^f   q^{-2} E^f x_\lambda   T_{j+1,2} E_1 T_{2f+1,3}^{-1}T_{2,l}T_{1,j}\\
&=  (-1)^{f-1} q^{-1} E^f x_\lambda   T_{j+1,3}  T_{2f+1,3}^{-1} T_{2,l} T_{1,j}\\
&= (-1)^{f-1} q^{-1} E^f x_\lambda   T_{j+1,2f+1} T_{2,l}T_{1,j}= (-1)^{f-1} q^{-1} E^f x_\lambda  T_{2,l} T_{j,2f}T_{1,j} \\
&= (-1)^{f-1} q^{-1} E^f x_\lambda  T_{2,2f}T_{1,2f-1} T_{2f,l}T_{j,2f}T_{2f-1,j}=q^{-1}E^f x_\lambda T_{2f,l}(2f-1,j)\\
&\equiv q^{-1} E^f x_\lambda  T_{l,2f}^{-1}(2f-1,j) \pmod {N_a},
\end{aligned}
$$  where the seventh and twelfth  equalities  follow from Lemma~\ref{key1}(5). This completes the proof of (4).\end{proof}
\begin{Lemma}\label{scofjme} If  $\lambda\in \Sigma^+(l)$ and $\mu^{(j)}=\lambda \cup p_j$, $a<j\leq m$, then  $$y_{\mu^{(j)}}^\lambda x_l \equiv\frac{q^{2c(p_j)-2}-1}{q-q^{-1}}y_{\mu^{(j)}}^\lambda \pmod { N_{j-1}}$$where $N_{j-1}$ is given  in \eqref{nkk}.
\end{Lemma}

\begin{proof}Write   $p_j=(k,\lambda_k+1)$ and $h=T_{b_k,2f-1}^{-1}\sum _{i=b_{k-1}+1}^{b_k}q^{b_k-i}T_{b_k,i}\in \mathcal B_{q,l}$ where  $b_k$ is defined in Definition~\ref{murr}(2).
We have $y_{\mu^{(j)}}^\lambda x_l=E^fx_\lambda  T_{l,2f}^{-1}x_l h$.
  Let $$X'=\sum_{2f+1\leq i<j\leq l}(i,j) \  \text{ and } \ X=\sum_{2f-1\leq i<j\leq l-1}(i,j)$$ It is well-known that  $X$ (resp., $X'$) is a central element in  $\tilde H_{l-2f+1}$ (resp., $\tilde H_{l-2f}$) where $\tilde H_{l-2f+1}$ is the Hecke algebra associated to the symmetric group on $\{2f-1, 2f,\ldots, l-1\}$.
  Moreover, $$ x_\lambda X'\equiv \sum_{p\in [\lambda]}\frac{q^{2c(p)}-1}{q-q^{-1}}x_\lambda \pmod{ \tilde H_{l-2f}^{\rhd \lambda}},\ \    x_{\mu^{(j)}} X\equiv \sum_{b\in[\mu^{(j)}]}\frac{q^{2c(b)}-1}{q-q^{-1}}x_{\mu^{(j)}} \pmod {\tilde H_{l-2f+1}^{\rhd \mu^{(j)}}}.$$  Consequently, by straightforward computation (see arguments on \eqref{mkspandedd} for the second equivalence), we have
\begin{equation}\label{pp1} E^fx_\lambda X'\equiv \sum_{p\in [\lambda]}\frac{q^{2c(p)}-1}{q-q^{-1}} E^fx_\lambda \pmod{ \mathcal B_{q,l}^{\rhd \lambda}}, \ \
 y_{\mu^{(j)}}^\lambda X  \equiv \sum_{b\in[\mu^{(j)}]}\frac{q^{2c(b)}-1}{q-q^{-1}}y_{\mu^{(j)}}^\lambda\pmod{ N_{j-1}},
 \end{equation}
By \eqref{symulambda},   Lemmas~\ref{commuhsgxt},~\ref{actofxn}(2) and~\ref{cumopp} we have
\begin{equation} \label{ppp1} \begin{aligned}
 E^fx_\lambda  T_{l,2f}^{-1}x_l =& E^fx_\lambda  T_{l,2f}^{-1}(\sum_{2f-1\leq i\leq l-1}(i,l)+q^{-1}\bar{(i,l)} )\\
\equiv & q^{-2}E^fx_\lambda  T_{l,2f}^{-1}(\sum_{2f-1<j\leq l-1}(2f-1,j) -q ) \pmod { N_a}\\
\equiv & q^{-2}E^fx_\lambda  T_{l,2f}^{-1}X  -q^{-2}E^fx_\lambda X'  T_{l,2f}^{-1} -q^{-1}E^fx_\lambda  T_{l,2f}^{-1} \pmod { N_a}\\
\end{aligned}\end{equation}
 Now, the required formula follows immediately from \eqref{pp1}-\eqref{ppp1}.\end{proof}

Suppose $\mathbf t\in \mathscr T_l^{up}(\lambda) $  and $\lambda\in \Lambda^+(l-2f)$. For any $1\le k\le l$, define
\begin{equation}\label{resid}  c_\mathbf t(k)=\begin{cases}
               \frac{q^{2c(p)}-1}{q-q^{-1}} & \text{if $\mathbf t_k=\mathbf t_{k-1}\cup p$,} \\
               \frac{q^{2c(p)-2}-1}{q-q^{-1}} & \text{if $\mathbf t_{k}=\mathbf t_{k-1}\setminus p$. }
             \end{cases}\end{equation}
%where $c(p)=j-i$ if $p=(i, j)$.

\begin{Theorem}\label{trishactio} Suppose  $\lambda\in \Sigma^+(l)$.  For any $\mathbf t\in\mathscr T_l^{up}(\lambda)$ and $1\leq i\leq l$,
$\mathrm m_{\mathbf t} x_i= c_{\mathbf t}(i) \mathrm m_{\mathbf t} +\sum_{\mathbf s\rhd\mathbf t}a_\mathbf s  \mathrm m_{\mathbf s} $ for some scalars $a_\mathbf s\in\Bbbk$.
\end{Theorem}

\begin{proof}We prove the result by induction on $l$. When $l=1$, the result is trivial since $x_1=0$. In general,  by Theorem~\ref{murphyb}(3) and induction assumption for $l-1$, we have the required formulae on $\mathrm m_{\mathbf t} x_i$, $1\leq i\leq l-1$.
Finally, the required formula on $ \mathrm m_{\mathbf t}x_l$ follows from  Lemmas~\ref{actofjme},~\ref{scofjme} and Theorem~\ref{murphyb}(4).
\end{proof}

  By Theorem~\ref{trishactio},   $\{x_1, x_2, \ldots, x_l\}$ are Jucys-Murphy elements with respect to the Jucys-Murphy basis $\{\mathrm m_{\mathbf t}+\mathcal B_{q,l}^{\rhd \lambda}\mid \mathbf t\in \mathscr T_l^{up}(\lambda)\}$ of any standard module $C(\lambda)$ in the sense of \cite[Definition~2.4]{MA1}. Unlike those for Hecke algebras, Brauer algebras and cyclotimic Nazarov-Wenzl algebras~etc, $\sum_{i=1}^l x_i$ may not be a central element in $\mathcal B_{q,l}$ in general.

\begin{Cor}\label{yyy1}  Suppose   $\lambda\in \Sigma^+(l)$, $\mu\in \Lambda_{e}^+(l)$ and $[C(\lambda):  L(\mu)1_l]\neq 0$.
\begin{itemize}
\item [(1)] There exist $\mathbf t\in \mathscr T_l^{up}(\lambda)$ and $\mathbf s\in \mathscr T_l^{up}(\mu)$ with $c_{\mathbf t}(k)=c_\mathbf s(k), 1\le k\le l$.
\item [(2)] If  $e>l$,  then $\mu$ is obtained from $\lambda$ by adding $2f$ boxes, say  $p_i, q_i$, $1\le i\le f$ such that $c(p_i)-c(q_i)=1$, $1\le i\le f$.
    \end{itemize}
\end{Cor}
\begin{proof} We have explained that the Jucys-Murphy elements $x_1, x_2, \cdots, x_l$ generates a subalgebra $L_l$ of $\mathcal B_{l, q}$. Since   $[C(\lambda):  L(\mu)1_l]\neq 0$, restricting both $C(\lambda)$ and $  L(\mu)1_l$  to $L_l$ yields that a simple $L_l$-module of  $L(\mu)1_l$ (and hence of $C(\mu)$) has to be a composition factor of right $L_l$-module $C(\lambda)$. Now (1) follows from  Theorem~\ref{trishactio}.
 If  $e>l$, then  $H_l$ is semisimple and  the right $H_{2f}$-module $Hom_{\mathcal B}(2f, 0)$ is a direct sum of certain dual Specht modules. Using
Lemma~\ref{iso321} and Littlewood-Richardson rule yields that $\lambda\subset \mu$. By (1) and  \eqref{resid},  $\mu$ is obtained from $\lambda$ by adding $2f$ boxes, say  $p_i, q_i$, $1\le i\le f$ such that
$$ \frac{q^{2c(q_i)}-1}{q-q^{-1}}= \frac{q^{2c(p_i)-2}-1}{q-q^{-1}}.$$
 Since $e>l$, we have $c(p_i)-c(q_i)=1$.
\end{proof}

\section{ Blocks of   periplectic $q$-Brauer category  and its endomorphism algebras}
In this section, $\Bbbk$ is always an  algebraically closed field $\Bbbk$ with characteristic not $2$.
we classify  blocks of the locally unital algebra $A$ associated to the periplectic Brauer category over $\Bbbk$  when $q^2$ is not a root of unity. We also classify blocks of $\mathcal B_{q,l}$. In this  case,  $q^2$ can be a root of unity and the quantum characteristic $e>l$.

\begin{Lemma}\label{sjxshxs} Suppose $e>l$ and $\lambda\in \Lambda^+(l-2f)$.
Then $[\Delta(\lambda): L(\mu)]\neq 0$ only if $\mu \in \Lambda^+(l-2h)$ with $h<f$. In this case, $ \lambda \subset \mu$ and $2(f-h)$ boxes, say $p_i, q_i, 1\le i\le f-h$ in $\mu\setminus \lambda$ can be paired such that
$c(p_i)-c(q_i)=1$.
\end{Lemma}
\begin{proof} Standard results for upper finite fully stratified category (e.g., \cite[Lemma~3.13]{BS}) implies $l-2f<l-2h$, proving the first statement of the result.  Using the idempotent functor $1_{l-2h}$, Lemma~\ref{iso321} and Corollary~\ref{yyy1}(2) yields the last statement of the result.
\end{proof}

\begin{Lemma}\label{keysjs} Suppose $e>l$ and  $(\lambda, \mu)\in \Lambda^+(l-2)\times  \Lambda^+(l)$.
If $\lambda\subset \mu $ and $\mu$ is obtained from $\lambda$ by adding two boxes at  the same column,  then $[\Delta(\lambda):L(\mu)]=[C(\lambda): L(\mu)1_l]=1$.
\end{Lemma}
\begin{proof} Since we are assuming $e>l$, $L(\mu)1_l\cong S^\mu$. By Corollary~\ref{ijres2}, there is a right $H_l$-isomorphism
\begin{equation}\label{sxjhs}
\text{Res}_{H_l}C(\lambda)\cong S^\mu \oplus\bigoplus _{\nu\neq \mu} S^\nu
\end{equation}
where $\nu\neq \mu$ is obtained from $\lambda$ by adding two boxes not at the same row.
In particular,  $[\text{Res}_{H_l}C(\lambda): S^\mu]=1$ and
 $[\Delta(\lambda):L(\mu)]\le 1$.  We claim that
the component $ S^\mu$ in \eqref{sxjhs} is annihilated by $E_1$.
If so, it is a right $\mathcal B_{q, l}$-module and hence   $[\Delta(\lambda):L(\mu)]= 1$.

Thanks to  Corollary~\ref{jdijss1}, we have $E_1 \mathcal B_{q, l} E_1=E_1 \mathcal B_{q, l-2}$. Therefore, $C(\lambda)E_1\cong S^{\lambda}$ as right $H_{l-2}$-modules
and hence  $ S^\mu E_1$ can be considered as a right $H_{l-2}$-submodule of $S^\lambda$.
Using Littlewood-Richardson rule and Frobenius reciprocity, we have
\begin{equation}\label{zzz12} \text{Res}_{H_2\otimes \tilde H_{l-2}} S^\mu=S^{(1,1)}\boxtimes S^\lambda \oplus\bigoplus_{\nu\neq \lambda}S^{\gamma}\boxtimes S^\nu ,\end{equation} where $\gamma\in\{(1,1), (2)\}$.
In particular, \eqref{zzz12} gives a decomposition of right $H_{l-2}$-modules. Since $ S^\mu E_1\subset S^{\lambda}$, we have $ S^{\gamma}\boxtimes S^\nu E_1=0 $ for any $\nu\neq \lambda$. This proves that
$S^\mu E_1=S^{(1,1)} \boxtimes S^\lambda E_1$. However, by Definition~\ref{defn:qBrauer}(R4), $T_1E_1=qE_1$ and $T_1$ acts on $S^{(1, 1)} $ as $-q^{-1}$. Since we are assuming $e>l$, we have $q\neq -q^{-1}$, forcing  $S^{(1,1)}\boxtimes S^\lambda E_1=0$. This completes the proof of our claim.
\end{proof}

Two simple $A$-modules $L(\lambda)$ and $L(\mu)$ are said to be  in the same block if there is a sequence $\lambda^{(1)}=\lambda, \lambda^{(2)}, \ldots, \lambda^{(k)}=\mu$ in $\Lambda_e$  such that there is a
nontrivial extension between  $L(\lambda^{(i)})$ and $L(\lambda^{(i+1)})$ for $i=1,2,\ldots,k-1$.
In the current case, it is equivalent to saying that there is a sequence $\lambda^{(1)}=\lambda, \lambda^{(2)}, \ldots, \lambda^{(k)}=\mu$ in $\Lambda_e$  such that either $L(\lambda^{(i)})$ is a composition factor of $\Delta(\lambda^{(i+1)})$ or  $L(\lambda^{(i+1)})$ is a composition factor of   $\Delta(\lambda^{(i)})$ for all $1\le i\le k-1$.
%All composition factors of $P(\lambda)$ are in the same block.

We recall $2$-core of a partition $\lambda$ in \cite[\S5.3]{Ma} as follows. It is a partition obtained from $\lambda$  by iteratively removing
rim $2$-hooks from the Young diagram of $[\lambda]$ until no more rim $2$-hook can be removed.

\begin{Lemma}\label{bolshs} Suppose that $q^2$ is not a root of unity.  If $\lambda$ and $\mu$ have the same $2$-core, then right $A$-modules $L(\lambda)$ and $L(\mu)$ are in the same block.
\end{Lemma}

\begin{proof}
 Suppose $\lambda $ is obtained from $ \mu$ by removing a rim $2$-hook.
  If  boxes in this $2$-hook are  at the same column,  by Lemma ~\ref{keysjs},
$[\Delta(\lambda):L(\mu)]\neq 0$, forcing  $L(\lambda)$ and $L(\mu)$ to be in the same block.
If boxes in this $2$-hook are at the same row, then $\lambda^t$ is obtained from $ \mu^t$ by two  boxes in the same column. Thanks to Lemma~\ref{keysjs},  $[\Delta(\lambda^t):L(\mu^t)]\neq 0$. Using Corollary~\ref{blmul}(3), we have $[P(\lambda):L(\mu)]=[\Delta(\lambda^t):L(\mu^t)]\neq 0$. Again, $L(\lambda)$ and $L(\mu)$ are in the same block. The general case follows from this observation and the definition on  $2$-core of a partition. \end{proof}

\begin{Lemma}\label{bolshs2} Suppose  $e>l$. If  $\lambda ,\mu \in\Sigma^+_{e}(l)$ and  $\lambda, \mu$  have the same $2$-core, then $L(\lambda)1_l$  and $L(\mu)1_l$ are in the same block of $\mathcal B_{q,l}$.
\end{Lemma}
\begin{proof} We use Lemma~\ref{keysjs} when we prove this result. It is enough  to assume  $e>l$. By Theorem~\ref{simsi}, $L(\lambda) 1_l\neq 0$ if $\lambda\neq 0$. Therefore, if the $2$-core of $\lambda$ is not $\emptyset$, the result follows immediately from  Lemmas~\ref{keysjs}-\ref{bolshs} by applying the idempotent functor $1_l$.

Finally, we assume that the $2$-core of $\lambda$ is $\emptyset$. In this case, by  Lemmas~\ref{keysjs}-\ref{bolshs} again, either  $L(2)1_l $ or $L(1,1)1_l$ is in the block containing $L(\lambda)1_l$.  By Corollary~\ref{blmul}(3) and Lemma~\ref{keysjs} , $[P(2): \Delta(\emptyset)]
=[\Delta(\emptyset): L(1,1)]=1$.
 However, by standard results on upper finite fully stratified category (e.g., \cite[Lemma~3.36]{BS}), each projective cover of a simple module has a finite $\Delta$-flag. In particular,
$$[P(2):L(1,1)]\geq [P(2): \Delta(\emptyset)][\Delta(\emptyset): L(1,1)]\ge 1.$$
 Consequently,  $L(\lambda)1_l$,  $L(2)1_l$  and  $L(1,1)1_l$ are always in the same block. In any case, $L(\lambda)1_l$  and $L(\mu)1_l$ are in the same block of $\mathcal B_{q,l}$.
\end{proof}

\begin{Theorem}Suppose $\lambda, \mu\in \Lambda$.
\begin{itemize}
\item[(1)] If $e=\infty$, then  $L(\lambda)$ and $L(\mu)$ are in the same block if and only if $\lambda$ and $\mu$ have   the same  $2$-core.
\item[(2)] If $e>l$ and $\lambda, \mu \in\Sigma^+_{e} (l)$, then    $L(\lambda)1_l$ and $L(\mu)1_l$ are in the same block if and only if $\lambda$ and $\mu$ have  the same $2$-core.
    \item[(3)] The locally unital algebra $A$ associated to the periplectic Brauer category is always not semisimple.
    \item[(4)] $\mathcal B_{q, 1}$ is always semsimple and  $\mathcal B_{q, l}$ is always not semisimple if $l\ge 2$.
\end{itemize}
\end{Theorem}

\begin{proof} Thanks to Lemmas~\ref{bolshs}-\ref{bolshs2}, it remains to show the only if parts of (1) and (2). Let $\sharp \lambda$ be the number of even contents minus the number of odd contents of $\lambda$.
By Lemma~\ref{sjxshxs},  $\sharp \lambda=\sharp \mu$ if $[\Delta(\lambda):L(\mu)]\neq0$.
It is proved in \cite[Lemma~7.3.3]{Col} that  $\sharp \lambda=\sharp \mu$ if and only if $\lambda$ and $\mu$  have  the same $2$-core. This proves the only if part of (1).
 Similarly, we can check the  only if part of (2). In this case, we have to use  Corollary~\ref{yyy1}(2) instead of Proposition~\ref{sjxshxs}.
   Since there is an algebra epimorphism from $\mathcal B_{q, l}$ to $H_l$, $\mathcal B_{q, l}$ is not semisimple if $H_l$ is not semisimple. Therefore, we can assume $e>l$ when we prove (4). In this case, (4) follows from (2).  Similarly, if $q^2$ is a root of unity, then $A$ is not semisimple.  Otherwise, $1_lA1_l\cong \mathcal B_{q, l}$ is semisimple for any $l\in \mathbb N$. In particular, $H_l$ is semisimple when $l>e$,  a contradiction. When $q^2$ is not a root of unity, (3) follows from (1).  In fact,   one can easily find a block which contains at least two simple modules in generic case.
\end{proof}


\small\begin{thebibliography}{DWH99}


\bibitem{AGG} {\scshape  S. Ahmed,  D. Grantcharov,  and N. Guay}, {\og Quantized enveloping superalgebra of type P\fg}. \emph{Lett. Math. Phys.}  \textbf{111} (2021), no. 3, Paper No. 84, 17 pp.

\bibitem{B+9} {\scshape M. Balagovic, Z. Daugherty, I. Entova-Aizenbud, I. Halacheva, J. Hennig, M. S. Im, G. Letzer, E. Norton, V. Serganova, C. Stroppel}, {\og  The affine VW supercategory \fg},  \emph{Selecta Math.} (N.S.) \textbf{26} (2020), no. 2, Paper No. 20, 42 pp.
   \bibitem {BCV} {\scshape S. Barbier, A. Cox, M. Visscher}, {\og  The blocks of the periplectic Brauer algebra in positive characteristic\fg}, \emph{J. Algebra.} \textbf{534} (2019), 289--312.

\bibitem{BCK} {\scshape J. Brundan, J. Comes, and J. Kujawa}, {\og  A basis theorem for the degenerate affine oriented Brauer-Clifford supercategory\fg}, \emph{ Canad. J. Math.} \textbf{71} (2019), 1061--1101.
\bibitem{BEA} {\scshape J. Brundan and  A.P. Ellis}, {\og  Monoidal supercategories\fg}, \emph{ Commun. Math. Phys.} \textbf{351} (2017), 1045--1089.

\bibitem{BS}{ \scshape J.Brundan, C. Stroppel}, {\og Semi-infinite highest weight categories\fg}, \emph{Memoir AMS}, to appear.
\bibitem{CP}{\scshape C. Chen and Y. Peng}, {\og Affine periplectic Brauer algebras\fg},
\emph{J. Algebra}, \textbf{501} (2018), 345--372.
\bibitem{Col}{\scshape K. Coulembier}, {\og The periplectic Brauer algebra\fg},
\emph{Proc. London Math. Soc.}, \textbf{117} (2018), no.3, 441--482.
\bibitem{Col1}{\scshape K. Coulembier and M. Ehrig}, {\og The periplectic Brauer algebra II: Decomposition multiplicities\fg},
\emph{J. Comb. Algebra}, \textbf{2} (2018), no.1, 19--46.
\bibitem{Col2}{\scshape K. Coulembier and M. Ehrig}, {\og The periplectic Brauer algebra III: The Deligne category\fg},
\emph{Algebr. Represet. Theory}, \textbf{24} (2021),  993--1027.
\bibitem{DKM}{\scshape N.Davidson, J. Kujawa and R. Muth}, {\og Webs of type $P$\fg},
arXiv:2019.03410v1[math. RT].
\bibitem{DKM1}{\scshape N.Davidson, J. Kujawa and R. Muth}, {\og Howe duality  of type $P$\fg},
arXiv:2109.03984v1[math. RT].
\bibitem{DR}{\scshape J. Du and H. Rui} {\og Based algebras and standard bases for quasi-hereditary algebras \fg}, \emph{Trans. A.M.S.
}, \textbf{350}  (1998) 3207--3235.
\bibitem{Eng}{\scshape J. Enyang} {\og Specht modules and semisimplicity criteria for Brauer and Birman-Murakami-Wenzl algebras\fg}, \emph{J. Algebr. Comb.}, \textbf{26}  (2007) 291--341.

\bibitem{GRS}{\scshape M.Gao, H.Rui, and L.Song}, {\og A basis theorem for the affine  Kauffman  category and its    cyclotomic quotients\fg},  \emph{J. Algebra}, \textbf{608}  (2022) 774--846.
\bibitem{GRS2}{\scshape M.Gao, H.Rui, and L.Song}, {\og Representations of weakly triangular categories \fg}, arXiv:2012.02945[math.RT].

\bibitem{GL}{\scshape J.Graham and G. Lehrer}, {\og Cellular algebras\fg}, \emph{Invent. Math.}, \textbf{123} (1996), 1--34. %MR 2232856
\bibitem{KT}{\scshape J.Kujawa and B. Tharp}, {\og The marked Brauer category\fg}, \emph{J. Lond. Math.Soc.}, (2) \textbf{95} (2017), 393--413.
\bibitem{Ma}{\scshape A.  Mathas}, \emph{\og Iwahori-Hecke algebras and Schur algebras of the symmetric group \fg },  \emph{University Lecture Series}, \textbf{15},  American Mathematical Society, Providence, RI, 1999.
    \bibitem{Moon} {\scshape D. Moon}, {\og Tensor product represenation of the Lie superalgebra $\mathfrak p(n)$ and their centralizers \fg}, \emph{ Comm. Algebra}, \textbf{31} (2003), 2095--2140. %MR 1834081
 \bibitem{MA1} {\scshape A. Mathas} {\og  Seminormal forms and Gram determinants for cellular algebras\fg} With an appendix by Marcos Soriano, \emph
{J. Reine Angew. Math.} \textbf{619} (2008), 141-173.
\bibitem{Mor}{\scshape H.R. Morton}, {\og  A basis for the Birman-Wenzl algebras\fg},
arXiv:1012.3116v1[math.QA].

\bibitem{Tu1} {\scshape V.Turaev}, {\og Quantum Invariants of Knots and 3-Manifolds\fg}, \emph{De Gruyter Studies in Mathematics} \textbf{18}, De Gruyter, Berlin, 2016.
        \bibitem{Wen} {\scshape H. Wenzl}, {\og  Quantum groups and subfactors of type B, C, and D.\fg},
\emph{Comm. Math. Phys} \textbf{133} (1990),  383-432.% MR 2855137
\end{thebibliography}
\end{document}